\documentclass[10pt]{article}

\usepackage{amsfonts}
\usepackage{amsmath}   
\usepackage{amssymb}
\usepackage{amsthm}
\usepackage{mathtools}
\usepackage{tikz-cd}

\newtheorem*{Th*}{Theorem}
\newtheorem{Th}{Theorem}[section]
\newtheorem{Prop}{Proposition}[section]   
\newtheorem{Lem}{Lemma}[section]   
\newtheorem{Coro}{Corollary}[section]   
\newtheorem{Def}{Definition}[section]
\newtheorem{Rem}{Remark}[section]

\newcommand{\R}{\mathbb{R}}
\newcommand{\Z}{\mathbb{Z}}
\newcommand{\C}{\mathbb{C}}
\newcommand{\T}{\mathbb{T}}
\newcommand{\TT}{\mathcal{T}}

\newcommand{\U}{{\mathcal U}}
\newcommand{\V}{{\mathcal V}}
\newcommand{\W}{{\mathcal W}}
\newcommand{\CC}{{\mathcal C}}
\newcommand{\DD}{{\mathcal D}}
\newcommand{\PP}{{\mathcal P}}
\newcommand{\h}{\mathfrak{h}}

\newcommand{\Id}{{\rm Id}_{2 \times 2}}
\newcommand{\Mat}{\rm Mat_{2\times 2}}
\newcommand{\p}{{\partial}}
\newcommand{\spec}{\mathop{\rm Spec}\nolimits}
\newcommand{\iso}{\mathop{\rm Iso}\nolimits}
\newcommand{\re}{\mathop{\rm Re}}
\newcommand{\im}{\mathop{\rm Im}}
\newcommand{\Tor}{\mathop{\rm Tor}\nolimits}
\newcommand{\Span}{\mathop{\rm Span}\nolimits}

\begin{document}

\title{Arnold-Liouville theorem for integrable PDEs: a case study of the focusing NLS equation}   
 
\author{T. Kappeler\footnote{T.K. is partially supported by the Swiss National Science Foundation.} 
and P. Topalov\footnote{P.T. is partially supported by the Simons Foundation, Award \#526907}}
  
\maketitle

\begin{abstract}  
We prove an infinite dimensional version of the Arnold-Liouville theorem for integrable
non-linear PDEs: In a case study we consider the {\em focusing} NLS equation with periodic boundary 
conditions.
\end{abstract}    

\medskip


\medskip


\section{Introduction}\label{sec:introduction}
Let us first review the classical Arnold-Liouville theorem (\cite{Liou,Min,Arn}) in the most simple setup:
assume that the phase space $M$ is an open subset of $\R^{2n}=\{(x,y)\,|\,x,y\in\R^n\}$, $n\ge 1$, 
endowed with the Poisson bracket 
\[
\{F,G\}=\sum_{j=1}^n\Big(\p_{y_j}F\cdot\p_{x_j}G-\p_{x_j}F\cdot\p_{y_j}G\Big).
\]
The Hamiltonian system with smooth Hamiltonian $H : M\to\R$ then has the form
\begin{equation*}
\dot x=\p_y H,\quad \dot y=-\partial_x H
\end{equation*}
with corresponding Hamiltonian vector field $X_H=(\p_y H,-\p_x H)$ where 
$\partial_x\equiv(\p_{x_1},...,\p_{x_n})$ and $\p_y\equiv(\p_{y_1},...,\p_{y_n})$.
Such a vector field is said to be {\em completely integrable} if there exist $n$ 
pairwise Poisson commuting smooth integrals $F_1,...,F_n :M\to\R$ (i.e. $\{H,F_k\}=0$ and
$\{F_k,F_l\}=0$ for any $1\le k,l\le n$) so that the differentials 
$d_{(x,y)} F_1$,...,$d_{(x,y)}F_n\in T^*_{(x,y)}\R^{2n}$
are linearly independent on an open dense subset of points $(x,y)$ in $M$.
In this setup the Arnold-Liouville theorem reads as follows (cf. e.g. \cite{Arn} or \cite{MZ}):

\begin{Th*}[Arnold-Liouville]
Assume that $c\in F(M)\subseteq\R^n$ is a regular value of the momentum map
$F=(F_1,...,F_n) : M\to\R^n$ so that a connected component of $F^{-1}(c)$, denoted by $N_c$, is compact. 
Then $N_c$ is an $n$-dimensional torus that is invariant with respect to the
Hamiltonian vector field $X_H$. Moreover, there exists a $X_H$-invariant open neighborhood $U$ of
$N_c$ in $M$ and a diffeomorphism 
\[
\Psi : D\times(\R/2\pi)^n\to U,\quad(I,\theta)\mapsto\Psi(I,\theta),
\]
where $D$ is an open disk in $\R^n$, so that the actions $I=(I_1,...,I_n)$ and
the angles $\theta=(\theta_1,...,\theta_n)$ are canonical coordinates (i.e., $\{\theta_j,I_j\}=1$ for all
$1\le j\le n$ whereas all other brackets between the coordinates vanish) and the pull-back ${\mathcal H}=H\circ\Psi^{-1}$
of the Hamiltonian $H$ depends only on the actions.
\end{Th*}
An immediate implication of this theorem is that the equation of motion, when expressed in 
action-angle coordinates, becomes
\[
\dot\theta=\p_I\mathcal{H}(I),\quad\dot I=-\partial_\theta\mathcal{H}(I)=0.
\]
Hence, the actions are conserved and so is the frequency vector $\omega(I):=\p_I\mathcal{H}(I)$.
Therefore, the equation can be solved explicitly, 
\[
\theta(t)=\theta(0)+\omega(I)t\;\big(\mathop{\rm mod} (2\pi\Z)^n\big),\quad I(t)=I(0).
\]
In particular, this shows that every solution is quasi-periodic.
Furthermore, the invariant tori in the neighborhood $U$ are of maximal dimension $n$ and their
tangent spaces at any given point $(x,y)$ are spanned by the Hamiltonian vectors 
$X_{F_1}(x,y)$,...,$X_{F_n}(x,y)$. Informally, one can say that the Arnold-Liouville
theorem assures that generically there is no room for other types of dynamics besides quasi-periodic motion.
In fact, by Sard's theorem one sees that if $F$ is e.g. proper then, under the conditions above,
the union of invariant tori of maximal dimension $n$ is open and dense in $M$. 
In addition, the set of regular values of the momentum map $F =(F_1,...,F_n) : M\to\R^n$ is open and 
dense in $F(M)$. Typically, action-angle coordinates cannot be extended globally.
One reason for this is the existence of singular values of the momentum map
$F : M\to\R^n$. Such a situation appears e.g. in the case of an elliptic fixed point 
$\xi\in M$. In that case, it might be possible to extend the coordinates $X_j=\sqrt{2I_j}\cos\theta_j$, 
$Y_j=\sqrt{2I_j}\sin\theta_j$, to an open neighborhood of $\xi$ in $M$.
We refer to such coordinates as Birkhoff coordinates and the Hamiltonian $H$, when expressed
in these coordinates, is said to be in Birkhoff normal form.
For results in this direction we refer to \cite{Vey,Eliasson,Zung1} and references therein.
In special cases, such as systems of coupled oscillators, Birkhoff coordinates can be defined on the entire
phase space and are referred to as global Birkhoff coordinates. In what follows we will keep this terminology
and will call $(X_j,Y_j)$ defined in terms of the action-angle coordinates as above again {\em Birkhoff coordinates}, 
even if they are non necessarily related to an elliptic fixed point.  Other obstructions to globally extend action-angle 
coordinates are the presence of hyperbolic fixed points as well as focus-focus fixed points 
(cf. \cite{Duistermaat,Eliasson,Zung2}) where one observes topological obstructions in case of a non-trivial 
monodromy of the action variables. Finally note that action-angle coordinates are used for obtaining KAM type results 
for small perturbations of integrable Hamiltonian systems.

Our aim is to present an extension of the above described version of the Arnold-Liouville theorem to 
integrable PDEs in form of a case study of the focusing nonlinear Schr\"odinger (fNLS) equation 
\begin{equation}\label{eq:nls}
i \p_t u=-\p^2_x u-2|u|^2u,\quad u|_{t=0}=u_0
\end{equation}
with periodic boundary conditions. 
Our motivation for choosing the fNLS equation stems from the fact that it is known to be an integrable PDE
which does {\em not} admit local action-angle coordinates in open neighborhoods of specific potentials (see \cite{KT3}), 
whereas in contrast, integrable PDEs such as the Korteweg-de Vries equation (KdV) or the defocusing nonlinear 
Schr\"odinger equation (dNLS) admit global Birkhoff coordinates (\cite{GK1,KP1}). 

To state our results we first need to introduce some notation and review some facts about the fNLS equation.
It is well known that \eqref{eq:nls} can be written as a Hamiltonian PDE. 
To describe it, let $L^2\equiv L^2(\T,\C)$ denote the Hilbert space of square-integrable 
complex valued functions on the unit torus $\T:=\R/\Z$ and let $L^2_c:=L^2\times L^2$.
On $L^2_c$ introduce the Poisson bracket defined for $C^1$-functions $F$ and $G$ on $L^2_c$ by
\begin{equation}\label{eq:poisson_bracket}
\{F,G\}(\varphi):=-i \int_0^1\big(\p_1 F\cdot\p_2 G-\p_2 F\cdot\p_1 G\big)\,dx
\end{equation}
where $\varphi=(\varphi_1,\varphi_2)$ and $\p_j F\equiv\p_{\varphi_j}F$ for $j=1,2$ are
the two components of the $L^2$-gradient of $F$ in $L^2_c$. More generally, we will consider $C^1$-functions
$F$ and $G$ defined on a dense subspace of $L^2_c$ having sufficiently regular $L^2$-gradients  so that
the integral in \eqref{eq:poisson_bracket} is well defined when viewed as a dual pairing.
The NLS-Hamiltonian, defined on the Sobolev space $H^1_c=H^1\times H^1$, $H^1\equiv H^1(\T,\C)$,
is given by
\[
\mathcal{H}_{NLS}(\varphi):=
\int_0^1\big(\p_x\varphi_1\cdot\p_x\varphi_2+\varphi_1^2\varphi_2^2\big)\,dx.
\]
The corresponding Hamiltonian equation then reads
\begin{equation}\label{eq:nls'}
\p_t(\varphi_1,\varphi_2)=-i \big(\p_2\mathcal{H}_{NLS},-\p_1\mathcal{H}_{NLS}\big).
\end{equation}
Equation \eqref{eq:nls} is obtained by restricting \eqref{eq:nls'} to the real
subspace $i L^2_r:=\{\varphi\in L^2_c\,|\,\varphi_2=-\overline{\varphi_1}\}$ of the complex 
vector space $L^2_c$. More precisely, for $\varphi=i (u,\overline{u})$, one gets the
fNLS equation $i \p_t u=-\p_x^2 u-2|u|^2u$.
We also remark that when restricting \eqref{eq:nls'} to the real subspace 
$L^2_r:=\{\varphi\in L^2_c\,|\,\varphi_2=\overline{\varphi_1}\}$ one obtains the dNLS equation
mentioned above.
According to \cite{ZS} equation \eqref{eq:nls'} admits a Lax pair representation (cf. \cite{Lax})
\[
\p_t L(\varphi)=[P(\varphi),L(\varphi)]
\]
where $L(\varphi)$ is the Zakharov-Shabat operator (ZS operator)
\[
L(\varphi):=i \begin{pmatrix} 1&0\\0&-1\end{pmatrix}\p_x+
\begin{pmatrix} 0&\varphi_1\\\varphi_2&0\end{pmatrix}
\]
and $P(\varphi)$ is a certain differential operator of second order.
As a consequence, the periodic spectrum of $L(\varphi)$ is invariant with respect to the NLS flow.
Actually, we need to consider the periodic and the anti-periodic spectrum of $L(\varphi)$ or, 
by a slight abuse of notation, of $\varphi$.
In order to treat the two spectra at the same time we consider $L(\varphi)$ on the interval $[0,2]$
and impose periodic boundary conditions. Denote the spectrum of $L(\varphi)$ defined this way by 
$\spec_p L(\varphi)$. In what follows we refer to it as the {\em periodic spectrum} of $L(\varphi)$ or, by a slight 
abuse of terminology, of $\varphi$.
Note that $\spec_p L(\varphi)$ is discrete and invariant with respect to the NLS flow on $L^2_c$. 
We say that the periodic spectrum $\spec_p L(\varphi)$ is {\em simple} if every eigenvalue has 
algebraic multiplicity one. Furthermore, for any $\psi\in i L^2_r$ introduce the isospectral set,
\[
\iso(\psi):=\big\{\varphi\in i L^2_r\,\big|\,\spec_p L(\varphi)=\spec_p L(\psi)\big\}  
\]
where the equality of the two spectra means that they coincide together with the corresponding algebraic 
multiplicities of the eigenvalues.  Denote by $\iso_o(\psi)$ the connected component of 
$\iso(\psi)$ that contains $\psi$. For any integer $N\ge 0$ introduce the Sobolev spaces 
$H^N_c:=H^N\times H^N$ and the real subspace
\[
i H^N_r:=\{\varphi\in H^N_c\,|\,\varphi_2=-\overline{\varphi_1}\}
\]
where $H^N\equiv H^N(\T,\C)$ denotes the Sobolev space of functions $f : \T\to\C$ with
distributional derivatives up to order $N$ in $L^2$. In a similar way introduce the following spaces of 
complex valued sequences
\[
\h^N_c:=\h^N\times\h^N,
\]
\[
\h^N_r:=\big\{(z,w)\in\h^N_c\,\big|\,w_n=\overline{z_{(-n)}}\,\,\forall n\in\Z\big\},
\]
\[
i \h^N_r:=\big\{(z,w)\in\h^N_c\,\big|\,w_n=-\overline{z_{(-n)}}\,\,\forall n\in\Z\big\},
\]
where
\[
\h^N:=\big\{z=(z_n)_{n\in\Z}\,\big|\,\|z\|_N<\infty\big\},\quad
\|z\|_N:=\Big(\sum_{j\in\Z}(1+j^2)^N|z_j|^2\Big)^{1/2}.
\]
Note that $\h^N_r$ and $i\h^N_r$ are real subspaces of $\h^N_c$.
In the case when $N=0$ we set $\ell^2_c\equiv\h^0_c$, $i \ell^2_r\equiv i \h^0_r$, and $\ell^2\equiv\h^0$.
For simplicity, we will also use the same symbols for the spaces of sequence with indices $|n|>R$ with $R\ge 0$.


Finally, we say that the subset $\W\subseteq i L^2_r$ is {\em saturated} if 
$\iso_o(\psi)\subseteq\W$ for any $\psi\in\W$. The main result of this paper is the following Theorem.

\begin{Th}\label{th:main}
Assume that $\psi\in i L^2_r$ has simple periodic spectrum $\spec_p L(\psi)$.
Then there exist a saturated open neighborhood $\W$ of $\iso_o(\psi)$ in $i L^2_r$ and
a real analytic diffeomorphism 
\begin{equation}\label{eq:Psi}
\Psi : \W\to\Psi(\W)\subseteq i\ell^2_r,\quad\varphi\mapsto
\Big(\big(z_n(\varphi)\big)_{n\in\Z},\big(w_n(\varphi)\big)_{n\in\Z}\Big),
\end{equation}
onto the open subset $\Psi(\W)$ of $i\ell^2_r$ so that the following holds:
\begin{itemize} 
\item[(NF1)] $\Psi$ is canonical, i.e., $\{z_n,w_{(-n)}\}=-i$ for any $n\in\Z$ whereas all other brackets between
coordinate functions vanish.
\item[(NF2)] For any integer $N\ge 0$, $\Psi(\W\cap i H^N_r)\subseteq i \h^N_r$ and
\[
\Psi : \W\cap i H^N_r\to\Psi(\W\cap i H^N_r)\subseteq i \h^N_r
\]
is a real analytic diffeomorphism onto its image.
\item[(NF3)] The pull-back $\mathcal{H}_{NLS}\circ\Psi^{-1} : \Psi\big(\W\cap i H^1_r\big)\to\R$ of the fNLS 
Hamiltonian is a real analytic function that depends only on the actions $I_n:=z_n w_{(-n)}$, $n\in\Z$. 
\end{itemize}
\end{Th}

The open neighborhood $\W\subseteq i L^2_r$ in Theorem \ref{th:main} is chosen in such a way that 
for any $\varphi\in\W$ the spectrum $\spec_p L(\varphi)$ has the property that all multiple eigenvalues are 
real with algebraic and geometric multiplicity two whereas all simple eigenvalues are non-real and appear in
complex conjugate pairs. Hence the periodic eigenvalues of $L(\varphi)$ have algebraic multiplicity at most two.
Recall also that a potential $\varphi\in i L^2_r$ is called a {\em finite gap potential} if the number of simple
periodic eigenvalues of $L(\varphi)$ is finite (cf. e.g. \cite{KLT1,KLT2}).
As an immediate application of Theorem \ref{th:main} one obtains the following

\begin{Coro}\label{coro:main}
Assume that $\psi\in i H^N_r$ with $N\in\Z_{\ge 0}$ has simple periodic spectrum $\spec_p L(\psi)$. 
Then for any $\varphi\in\W$ where $\W$ is the open neighborhood of Theorem \ref{th:main},
the following holds:
\begin{itemize}
\item[(i)] The set $\Psi(\iso_o(\varphi))$ is compact in $i \h^N_r$ and can be represented as a direct product of 
countably many circles, one for each pair of complex conjugated simple periodic eigenvalues.
\item[(ii)] For $N\ge 1$, the fNLS equation on $\Psi(\W\cap i H^N_r)\subseteq i \h^N_r$ takes the form
\[
\dot z_n=-i\omega_n z_n,\quad \dot w_n=i \omega_n w_n,\quad n\in\Z
\]
where $\omega_n\equiv\omega_n(I):=\p_{I_n}(\mathcal{H}_{NLS}\circ\Psi^{-1})$ are the NLS frequencies
and $I_n=z_n w_{(-n)}$, $n\in\Z$, are the actions. Hence, the solutions of the fNLS equation with initial data in 
$\W\cap i H^N_r$ are globally defined and almost periodic in time.
\item[(iii)] The finite gap potentials in $\W$ lie on finite dimensional fNLS invariant tori 
contained in $\W\cap i H^N_r$ for any $N\ge 0$. The set of these potentials is dense in $\W\cap i H^N_r$.
\end{itemize}
\end{Coro}

\begin{Rem}
\begin{itemize}
\item[(i)] For any $\varphi\in\W\cap i H^N_r$, $N\ge 1$, and $t\in\R$ denote by $S_t(\varphi)$ the solution of the fNLS equation
obtained in Corollary \ref{coro:main} (ii) with initial data $\varphi$. Then for any given $N\in\Z_{\ge 1}$ and
$t\in\R$, one can prove that the flow map $S_t : \W\cap i H^N_r\to\W\cap i H^N_r$ is a homeomorphism 
(cf. \cite{KT1}). In addition, the map $S :\R\times\big(\W\cap i H^N_r\big)\to\W\cap i H^N_r$, 
$(t,\varphi)\mapsto S_t(\varphi)$, is continuous.
\item[(ii)] It can be shown that the frequencies $\omega_n$, $n\in\Z$, extend real analytically to the larger set
$\W\cap i L^2_r$ (cf. \cite{KT1,KM}).
\end{itemize}
\end{Rem}

The next result addresses the question of how restrictive the assumption of $\spec_p L(\psi)$ being simple is.
Let
\begin{equation}\label{eq:T}
\mathcal{T}:=\big\{\psi\in i L^2_r\,\big|\,\spec_p L(\psi)\,\,\text{is simple}\big\}.
\end{equation}
Recall that a subset $A$ of a complete metric space $X$ is said to be {\em residual} if it is the intersection of 
countably many open dense subsets. By Baire's theorem, such a set is dense in $X$. We prove in \cite{KTPreviato}
the following

\begin{Th}\label{prop:general_position}
For any integer $N\ge 0$, the set $\mathcal{T}\cap i H^N_r$ is residual in $i H^N_r$.
\end{Th}

\begin{Rem}
It is well known that the fNLS equation is wellposed on $i H^N_r$ for any integer $N\ge 0$ (\cite{Bo}).
It then follows from Theorem \ref{prop:general_position} and Corollary \ref{coro:main} (iii) that any solution 
in $i H^N_r$ with $N\ge 0$ can be approximated in $C\big([-T,T],i H^N_r\big)$ by
finite gap solutions for any $T>0$.
\end{Rem}

Finally we mention that the results of Theorem \ref{th:main} apply to any Hamiltonian in the fNLS hierarchy.
In particular, these results hold for the focusing modified KdV equation
\[
\p_t v=-\p_x^3 v-6 v^2 \p_xv,\quad v|_{t=0}=v_0
\]
with periodic boundary conditions which can be obtained as the restriction of the Hamiltonian PDE 
on the Poisson manifold $L^2_c$ with Hamiltonian
\[
\mathcal{H}_{mKdV}(\varphi):= 
\int_0^1\big(-(\p_x^3\varphi_1)\varphi_2+3(\varphi_1\p_x\varphi_1)\varphi_2^2\  \big)\,dx
\]
to the real subspace of $i L^2_r$,
\[
\big\{\varphi=i (v,v)\in i L^2_r\,\big|\,v\,\,\text{real valued}\big\}\cong L^2(\T,\R).
\]
We remark that Theorem \ref{prop:general_position} also holds in this setup meaning that
the subset $\{v\in L^2(\T,\R)\,|\,i(v,v)\in\mathcal{T}\cap i H^N_r\}$ is residual in $H^N(\T,\R)$ for any
integer $N\ge 0$ -- see \cite{KTPreviato} for more details.

\medskip

\noindent{\em Method of proof:} The proof of Theorem \ref{th:main} is based on the following key
ingredients: 

\medskip

(1) Setup allowing to construct analytic coordinates:
One of the principal merits of our analytic setup is that it allows to prove the canonical relations between the
action-angle coordinates by a deformation argument using the canonical relations between these coordinates
in a neighborhood of the zero potential established in \cite{KLTZ} (cf. also \cite{GK1} for the construction of
such coordinates in the defocusing case).

\medskip

(2)  Choice of contours: For an integrable system on a $2n$-dimensional symplectic space $M$ (as
the one discussed at the beginning of the introduction), action coordinates $I_j$, $1\le j\le n$, on
the invariant tori $N_c$ of dimension $n$, smoothly parametrized by regular values $c\in\R^n$
in the image of the momentum map $F : M\to \R^n$, can be defined by Arnold's formula
\[
I_j:=\frac{1}{2\pi}\int_{\gamma_j(c)}\alpha,\quad 1\le j\le n,
\]
in terms of the canonical $1$-form $\alpha$, stemming from the Poisson structure on $M$, and a set 
of cycles $\gamma_j(c)$, $1\le j\le n$, on $N_c$ that form a basis in the first homology group of
$N_c$ and depend  smoothly on the parameter $c$.
For integrable PDEs such as the KdV or the dNLS equations, Arnold's procedure for constructing the action 
coordinates has been successfully implemented in \cite{FMcL} (cf. also \cite{VN} and \cite{McKV}).
More precisely, in the case of the dNLS equation, the Dirichlet eigenvalues of $L(\varphi)$ can be used to define
cycles $\gamma_j$, $j\in\Z$.  Surprisingly, the integrals $\frac{1}{2\pi}\int_{\gamma_j}\alpha$ can be
interpreted as contour integrals on the complex plane (cf. e.g. \cite{FMcL,McKV,GK1}).
We emphasize that in the case of the fNLS equation, treated in the present paper, these contours can {\em not}
be obtained from the Dirichlet spectrum of $L(\varphi)$. It turns out that the contours in the complex plane which 
work in the case of the dNLS equation for potentials $\varphi$ near the origin in $H^N_r$ also work in the case of 
the fNLS for potentials near the origin in $i H^N_r$ (\cite{KLTZ}). We then use a deformation argument along 
an appropriately chosen path, that connects $\psi$ with a small open neighborhood of the zero potential in 
$i H^N_r$, to obtain contours for potentials in an open neighborhood of the isospectral set $\iso_o(\psi)\cap i H^N_r$.

\medskip

(3) Normalized differentials: The angle coordinates are defined in terms of a set of normalized differentials
on an open Riemann surface (of possibly infinite genus), associated to the periodic spectrum $\spec_p L(\varphi)$, 
and the Dirichlet spectrum of $L(\varphi)$ (cf. e.g. \cite{BBEIM,DK} for the case of finite 
gap potentials as well as \cite{GK1,McKT,McKV,FKT} for the case of more general potentials in $H^N_r$).
Such normalized differentials (with properties needed for our purposes) for generic potentials
in $i H^N_r$ have been constructed in \cite{KT2} (cf. also \cite{KLT3}). Note that the case of potentials in 
$i H^N_r$ is more complicated since the operator $L(\varphi)$ is not selfadjoint.
An important ingredient for estimates, needed to construct the angle coordinates, is the rather precise localization
of the zeros of these differentials provided in \cite{KT2}.
We emphasize that no assumptions are made on the Dirichlet eigenvalues of $L(\varphi)$. In particular,
they might have algebraic multiplicities greater or equal to two.
As in the case of the actions we construct the angles using a deformation argument.

\medskip

(4) Generic spectral properties of non-selfadjoint ZS operators:
The deformation argument, briefly discussed in item (1), requires that the path of deformation
stays within the part of phase space $i H^N_r$ which admits action-angle coordinates.
This part of the phase space contains the set of potentials $\varphi\in i H^N_r$ which have the property
that all multiple eigenvalues are real with geometric multiplicity two whereas all simple eigenvalues are 
non-real and appear in complex conjugated pairs.
In \cite{KLT2}, it is shown that this set is open and path connected. We remark that in order to
prove that the actions and the angles, first defined in an open neighborhood of the potential $\psi\in i H^N_r$,
analytically extend to an open neighborhood of $\iso_o(\psi)$, we make use of the assumption that {\em all}
periodic eigenvalues of $L(\psi)$ are simple.

\medskip

(5) Lyapunov type stability of isospectral sets: 
To ensure that the Birkhoff map \eqref{eq:Psi} is injective in an open neighborhood $\W$ of $\iso_o(\psi)$ in $i L^2_r$
with $\psi\in\mathcal{T}$, we show that for any open neighborhood $\U\subseteq\W$ of $\iso_o(\psi)$ in $i L^2_r$ there exists 
an open neighborhood $\V\subseteq\U$ of $\iso_o(\psi)$ in $i L^2_r$ with the property that $\iso_o(\varphi)\subseteq\U$
for any $\varphi\in\V$.

\medskip

\noindent{\em Additional comments:}
Informally, Theorem \ref{th:main} means that $(z_n,w_n)$, $n\in\Z$, can be thought of as nonlinear
Fourier coefficients of $\varphi=(\varphi_1,\varphi_2)$. They are referred to as {\em Birkhoff coordinates}.
The Birkhoff coordinates are constructed in terms of action and angle variables.
In the case of a finite gap potential, the angle variables are defined by {\em real} valued expressions involving 
the Abel map of a special curve of finite genus associated to the finite gap potential.
The question if these expressions are real valued has been a longstanding issue, raised by experts in the field
in connection with special solutions of the fNLS equation, given in terms of theta functions.

In \cite{KLTZ}, coordinates of the type provided by Theorem \ref{th:main} have been constructed 
in a open neighborhood of the origin in $i H^N_r$. Note however that $\spec_p L(0)$ is {\em not} simple
since it consists of real eigenvalues of algebraic and geometric multiplicity two.

\medskip

\noindent{\em Related work:} In the seventies and the eighties, several groups of scientists made pioneering
contributions to the development of the theory of integrable PDEs.
In the periodic or quasi-periodic setup, deep connections between such equations and complex geometry
as well as spectral theory were discovered and much of the efforts were aimed at representing classes of
solutions (referred to as finite band solutions) by the means of theta functions, leading to the discovery of 
finite dimensional invariant tori. See e.g. \cite{Lax,DN} as well as the books \cite{NMPZ,BBEIM,GH} and the 
references therein. Further developments of these connections allowed to treat more general classes of solutions. 
In particular, it was established that many integrable PDEs admit invariant tori with infinitely many degrees of freedom.
See e.g. \cite{McKT,FMcL,KP1,GK1,KT1} and the references therein.

\medskip

\noindent{\em Organization of the paper:} 
In Section \ref{sec:setup} we review various results on ZS operators and related topics
which are used in the paper. 
In Section \ref{sec:actions_in_U_tn} we construct a tubular neighborhood $U_{\rm tn}$ of an appropriate path, 
connecting a potential $\psi^{(0)}\in\mathcal{T}$ (cf. \eqref{eq:T}) near zero with a given potential $\psi\in\mathcal{T}$ and
define the actions in $U_{\rm tn}$. In Section \ref{sec:angles_in_U_tn} we define the angles in $U_{\rm tn}$.
In Section \ref{sec:actions_and_angles_in_U_iso} we introduce actions and angles in a tubular neighborhood in $i L^2_r$
of the isospectral set $\iso_o(\psi)$ of $\psi$. 
In Section \ref{sec:birkhoff_map} we define the pre-Birkhoff map and study its local properties
whereas in Section \ref{sec:proofs} we prove Theorem \ref{th:main}.

\medskip 

\noindent{\em Acknowledgment:} The authors gratefully acknowledge the support and hospitality of
the FIM at ETH Zurich and the Mathematics Departments of the Northeastern University and the University of Zurich.

\section{Setup}\label{sec:setup}
In this Section we review results needed throughout the paper. In particular
we recall the spectral properties of ZS operators and the results on Birkhoff
coordinates in a neighborhood of the zero potential (\cite{KLTZ}).

\medskip

For $\varphi = (\varphi _1,\varphi _2) \in L^2_c$ and $\lambda \in \C$,
let $M = M(x,\lambda ,\varphi )$, $x \in {\mathbb R}$, be the fundamental solution of 
$L(\varphi )M = \lambda M$ satisfying the initial condition 
$M(0,\lambda,\varphi )=\Id$, where $\Id$ is the identity $2\times 2$ matrix. 
It is convenient to write
\begin{equation}\label{eq:M}
 M := \begin{pmatrix} m_1 & m_2 \\ m_3 & m_4 \end{pmatrix}, \quad 
 M_1 :=\begin{pmatrix} m_1 \\ m_3 \end{pmatrix}, \quad 
 M_2 := \begin{pmatrix}m_2 \\ m_4 \end{pmatrix} .
\end{equation}
The fundamental solution $M(x,\lambda ,\varphi )$ is a continuous function on
${\mathbb R} \times \C \times L^2_c$, for any given $x \in {\mathbb R}$
it is analytic in $(\lambda,\varphi)\in\C \times L^2_c$, and for any
given $(\lambda , \varphi )$ in $\C \times L^2_c$, 
$M(\cdot,\lambda,\varphi)\in H^1([0,1],{\rm Mat}_{2\times 2}(\C))$ -- see \cite{GK1}. 
For $\varphi = 0$, $M$ is given by the diagonal $2\times 2$ matrix 
$E_\lambda (x) = \mbox{diag}(e^{- i\lambda x},e^{ i\lambda x}$).

\medskip

\noindent {\em Periodic spectrum:} Recall that a complex number $\lambda$ is said to be
a {\em periodic eigenvalue} of $L(\varphi )$ iff there exists a nonzero solution of 
$L(\varphi )f = \lambda f$ with $f(1) = \pm f(0)$. As $f(1) = M(1, \lambda )f(0)$ it means that $1$ or
$-1$ is an eigenvalue of the Floquet matrix $M(1,\lambda )$. Denote by
$\Delta (\lambda , \varphi )$ the discriminant of $L(\varphi )$
\[ 
\Delta (\lambda , \varphi ) := m_1(1, \lambda , \varphi ) + m_4(1,\lambda , \varphi ) 
\]
and note that by the discussion above $\Delta : \C\times L^2_c\to \C$ is analytic. 
It follows easily from the Wronskian identity that $\lambda$ is a periodic eigenvalue of 
$L(\varphi )$ iff $\Delta(\lambda )\in\{2,-2\}$. 
Hence the periodic spectrum of $L(\varphi )$ coincides with the zero set of the entire function 
\begin{equation}\label{eq:chi_p}
\chi_p(\lambda,\varphi ):=\Delta^2(\lambda,\varphi) - 4. 
\end{equation}
In fact, by \cite[Lemma 2.3]{KLT2}, the algebraic multiplicity of a periodic eigenvalue coincides with 
its multiplicity as a root of $\chi _p(\cdot,\varphi )$. We say that two complex numbers $a$ and $b$ are
{\em lexicographically ordered}, $a \preccurlyeq b$, if $[\re(a) < \re(b)]$ or
$[\re(a) = \re(b) \mbox { and } \im(a) \leq \im(b)]$. Furthermore, for any $n\in\Z$ and $R\in\Z_{\ge 0}$ 
introduce the disks 
\[
D_n := \{ \lambda \in \C,|\, |\lambda - n\pi | < \pi /6\}\, \text{and}\, 
B_R := \{ \lambda \in \C\,|\, |\lambda | < R \pi + \pi/6 \}.
\]
For a proof of the following well known Lemma see e.g. \cite{GK1}.

\begin{Lem}\label{lem:counting_lemma}
For any $\psi\in L^2_c$ there exist an open neighborhood $V_\psi$ of $\psi$ in $L^2_c$ and  $R_p\in\Z_{\ge 0}$ 
so that for any $\varphi\in V_\psi$ the following properties hold:
\begin{itemize}
\item[(i)]  The periodic spectrum $\spec_p L(\psi)$ is discrete. The set of eigenvalues counted with their
algebraic multiplicities consists of two sequences of complex numbers $(\lambda_n^+)_{|n|>R_p}$ and 
$(\lambda_n^-)_{|n|>R_p}$ with $\lambda_n^+,\lambda_n^-\in D_n$, 
$\lambda^-_n\preccurlyeq\lambda^+_n$, and a set $\Lambda_{R_p}\equiv\Lambda_{R_p}(\varphi)$ of 
$4 R_p+2$ additional eigenvalues that lie in the disk $B_{R_p}$. For any $|n|>R_p$, 
$\Delta(\lambda_n^\pm,\varphi)=(-1)^n 2$ and
\[
\lambda_n^\pm=n\pi+\ell_n^2
\]
where the remainder $(\lambda_n^\pm-n\pi)_{|n|>R_p}$ is bounded in $\ell^2$ locally uniformly in 
$\varphi\in V_\psi$. 
(Later we will list the eigenvalues in $B_{R_p}$ in a way convenient for our purposes.)

\item[(ii)] The set of roots of the entire function 
$\lambda\mapsto\dot\Delta(\lambda)\equiv\p_\lambda\Delta(\lambda,\varphi)$,
when counted with multiplicities, consists of a sequence
$(\dot\lambda_n)_{|n|>R_p}$, so that $\dot\lambda_n\in D_n$, and a set 
$\dot\Lambda_{R_p}\equiv\dot\Lambda_{R_p}(\varphi)$ of $2 R_p+1$ additional roots 
that lie in the disk $B_{R_p}$. For $|n|>R_p$, one has
\[
\dot\lambda_n=\frac{\lambda_n^++\lambda_n^-}{2}+(\lambda_n^+-\lambda_n^-)^2\ell^2_n
\]
where the remainder $\ell^2_n$ is bounded in $\ell^2$ locally uniformly in $\varphi\in V_\psi$.
The roots in $\dot\Lambda_{R_p}$ are listed in lexicographic order and with their multiplicities 
\[
\dot\lambda_{-R_p}\preccurlyeq\cdots\preccurlyeq
\dot\lambda_{k}\preccurlyeq\dot\lambda_{k+1}
\preccurlyeq\cdots\preccurlyeq\dot\lambda_{R_p},
\quad -R_p\le k\le R_p-1.
\]
\end{itemize}
\end{Lem}

For potentials $\varphi $ in $ i L^2_r$, the periodic spectrum of $L(\varphi )$ has additional properties.
By \cite[Proposition 2.6]{KLT2} the following holds.

\begin{Lem}\label{lem:spectrum_symmetries} 
For any given $\varphi\in i L^2_r$ any real periodic eigenvalue of $L(\varphi )$ has geometric multiplicity 
two and even algebraic multiplicity. For any periodic eigenvalue in $\C \backslash {\mathbb R}$, 
its complex conjugate $\overline \lambda$ is also a periodic eigenvalue and has the same
algebraic and geometric multiplicity as $\lambda $.
The periodic eigenvalues $\lambda^+_n$ and $\lambda^-_n$, $|n| > R_p$, given by 
Lemma \ref{lem:counting_lemma}, satisfy $\im(\lambda_n^+)\ge 0$ and 
\[ 
\lambda^-_n=\overline{\lambda ^+_n}\quad \forall |n| > R_p\,.
\]
\end{Lem}

\medskip

\noindent{\em Discriminant:} The following properties of $\Delta $ and $\dot\Delta $ are 
well known -- see e.g. \cite[Section 5] {GK1}. To state them introduce
$\pi _n:=n\pi \,\,\, \forall n \in \Z \backslash \{ 0 \}$ and $\pi _0 := 1$.

\begin{Lem}\label{lem:product1} 
For $\varphi\in L^2_c$ arbitrary, let $R_p \in \Z_{\ge 0}$ be as in 
Lemma \ref{lem:counting_lemma}.
\begin{itemize}

\item[(i)] The function 
$\lambda \mapsto \chi _p(\lambda ) = \Delta (\lambda ,\varphi )^2 - 4$ 
is entire and admits the product representation
\[ 
\chi _p(\lambda ) =- 4\Big(\prod _{|n|\leq R_p}\frac{1}{\pi _n^2}\Big)
\cdot\chi _{R_p}(\lambda)
\cdot\prod_{|n|>R_p} \frac{(\lambda ^+_n-\lambda )(\lambda ^-_n-\lambda )}{\pi ^2_n} 
\]
where $\chi _{R_p}(\lambda )\equiv\chi _{R_p}(\lambda,\varphi)$ denotes the polynomial of degree 
$4R_p + 2$ given by
\[ 
\chi_{R_p}(\lambda):=\prod _{\eta\in\Lambda_{R_p}} (\eta-\lambda) .
\]

\item[(ii)] The function $\lambda\mapsto\dot\Delta(\lambda )\equiv\dot\Delta(\lambda,\varphi)$ is 
entire and admits the product representation
\[ 
\dot\Delta(\lambda)=2\Big(\prod _{|n|\leq R_p}\frac{1}{\pi _n}\Big)
\cdot\prod_{\eta\in\dot\Lambda_{R_p}} (\eta-\lambda)
\cdot\prod _{|n|>R_p}\frac{\dot\lambda _n-\lambda }{\pi _n} .
\]

\item[(iii)] For any $\varphi\in i L^2_r$ and $\lambda\in\C$
\[ 
\Delta(\overline\lambda,\varphi)=\overline{\Delta (\lambda,\varphi )}, \quad 
\dot\Delta(\overline\lambda,\varphi )=\overline{\dot\Delta(\lambda,\varphi)} .
\]
In particular, the zero set of $\dot\Delta(\cdot,\varphi )$ is invariant under complex conjugation and 
thus $\dot\lambda_n$ is a simple real root for any $|n|>R_p$.
\end{itemize}
\end{Lem}

\medskip

\noindent The spectrum of $L(\varphi)$, $\varphi \in L^2_c$, when considered as an unbounded operator on 
$L^2({\mathbb R},\C)\times L^2 ({\mathbb R}, \C)$ is given by
\[ 
\spec_{\R} L(\varphi )=\big\{\lambda \in \C\,|\,\Delta (\lambda)\in[-2, 2]\big\}
\]
-- see e.g. \cite{KLT1}. Now consider the case $\varphi\in i L^2_r$.
By Lemma \ref{lem:product1} (iii), $\Delta (\lambda )$ is real for $\lambda \in{\mathbb R}$ and 
by Lemma \ref{lem:spectrum_symmetries} and the Wronskian identity one concludes (see e.g. \cite{KLT1}) that 
\begin{equation}\label{eq:real_line}
\Delta(\lambda,\varphi)\in [-2,2]\quad\forall\lambda\in\R\,\,\,\forall\varphi\in i L^2_r. 
\end{equation}

\medskip

\noindent{\em Dirichlet spectrum:} Denote by $\spec _D L(\varphi )$ the Dirichlet
spectrum of the operator $L(\varphi )$, i.e. the spectrum of the operator
$L(\varphi )$ considered with domain
\[ 
\big\{ f = (f_1, f_2) \in H^1([0,1], \C)^2\,\big|\,f_1(0) = f_2(0),\
f_1(1) = f_2(1)\big\} .
\]
(When the operator $L(\varphi)$ is written as an AKNS operator, the above boundary conditions 
become the standard Dirichlet boundary conditions -- see e.g. \cite{GK1}.)   
The Dirichlet spectrum is discrete and the eigenvalues satisfy the following
Counting Lemma -- see e.g. \cite {GK1}.
\begin{Lem}\label{lem:dirichlet_spectrum}
For any $\psi \in L^2_c$ the Dirichlet spectrum $\spec_D L(\psi)$ of $L(\psi)$ is discrete. Moreover, 
there exist $R_D\in\Z_{\ge 0}$ and an open neighborhood $V_\psi$ of $\psi$ in $L^2_c$ so that for any 
$\varphi\in V_\psi$, the set of Dirichlet eigenvalues, when counted with their multiplicities, 
consists of a sequence $(\mu_n)_{|n|>R_D}$ with 
$\mu_n\in D_n$  and a set of $2 R_D+1$ additional Dirichlet eigenvalues that lie in the disk $B_{R_D}$. 
These additional Dirichlet eigenvalues are listed in lexicographic order and with multiplicities
$\mu_{-R_D}\preccurlyeq\cdots\preccurlyeq\mu_{k}\preccurlyeq\mu_{k+1}\preccurlyeq\cdots\preccurlyeq\mu_{R_D}$, 
$-R_D\le k\le R_D$. For $|n|>R_D$, one has
\[
\mu_n=n\pi+\ell^2_n
\]
where the remainder $\ell^2_n$ is bounded in $\ell^2$ locally uniformly in $\varphi\in V_\psi$.
If $\lambda $ is a periodic eigenvalue of $L(\varphi )$ of geometric multiplicity two then $\lambda 
$ is also a Dirichlet eigenvalue.
\end{Lem}

Note that for any given $\varphi\in L^2_c$ the Dirichlet spectrum of $L(\varphi)$ coincides (with multiplicities) with the 
zeroes of the entire function $\chi_D(\cdot,\varphi) :\C\to\C$ where
\[
\chi_D : \C\times L^2_c\to\C,\quad(\lambda,\varphi)\mapsto\chi_D(\lambda,\varphi),
\]
is analytic and
\begin{equation}\label{eq:chi_D}
2 i\chi_D(\lambda,\varphi):=\big(m_4+m_3-m_2-m_1\big)\big|_{(1,\lambda,\varphi)}
\end{equation}
(see \cite[Theorem 5.1]{GK1} and the deformation 
argument in \cite[Appendix C]{KLT1}).

\medskip

\noindent{\em Birkhoff normal form:} We review results from \cite{KLTZ}
where Birkhoff coordinates near $\psi = 0$ were constructed.
It follows from Lemma \ref{lem:counting_lemma} and Lemma \ref{lem:dirichlet_spectrum} 
that there exists an open ball ${\mathcal U}_0$ centered at zero in $L^2_c$ so that for any $\varphi$ in 
${\mathcal U}_0$, the periodic eigenvalues of $L(\varphi )$ are given by two sequences $\lambda ^\pm _n$, 
$n \in\Z$, so that for any $n\in\Z$, $\lambda^\pm_n\in D_n$ and $\Delta (\lambda ^\pm _n) = 2(-1)^n$. 
In addition, the Dirichlet eigenvalues and the roots of $\dot \Delta$ are given by two sequences 
$\mu_n$, $n\in\Z$, and $\dot\lambda_n$, $n\in\Z$, which are all simple and satisfy 
$\mu_n,\dot \lambda _n\in D_n$ for any $n\in\Z$. 
By shrinking the ball $\U_0$ in $L^2_c$ if necessary so that it is contained in the domain of 
definition of the Birkhoff map constructed in \cite[Theorem 1.1]{KLTZ} we obtain the following

\begin{Th}\label{th:main_near_zero} 
The claims of Theorem \ref{th:main} hold on the open ball $\W_0\equiv\U_0\cap i L^2_r$.
The diffeomorphism $\Phi : \W_0\to\Phi(\W_0)\subseteq i\h^0_r$ is the restriction to $\U_0$ of the Birkhoff map
constructed in \cite[Theorem 1.1]{KLTZ}.
\end{Th}

\section{Actions on the neighborhood $U_{\rm tn}$}\label{sec:actions_in_U_tn}
By Theorem \ref{th:main_near_zero}, actions have been constructed for potentials in the open 
ball $\U_0$ centered at zero in $L^2_c$. Our goal is this Section is to show that they
analytically extend along a suitably chosen path to an open neighborhood of any given
potential $\psi^{(1)}$ with simple periodic spectrum. First we need to make some preliminary considerations.
For a given  $R\ge\Z_{\ge 0}$ introduce the set $\mathcal{T}^R\subseteq L^2_c$, defined as follows.

\begin{Def}\label{def:T^R}
An element $\varphi\in L^2_c$ lies in $\mathcal{T}^R$ if the following conditions on the periodic spectrum
$\spec_p L(\varphi)$ and the Dirichlet spectrum $\spec_D L(\varphi)$ (both counted with multiplicities) hold:
\begin{itemize}
\item[(R1)] For any $|n|> R$, the disk $D_n$ contains precisely two periodic eigenvalues. 
The set  $\Lambda_R(\varphi)$ of the remaining periodic eigenvalues consists of $4R+2$ eigenvalues 
which are simple and contained in the disk $B_R$, $\Lambda_R(\varphi)\subseteq B_R$.
\item[(R2)] For any $|n|> R$, the disk $D_n$ contains precisely one Dirichlet eigenvalue denoted by 
$\mu_n\equiv\mu_n(\varphi)$. There are $2R+1$ remaining Dirichlet eigenvalues which are contained in 
the disk $B_R$. These remaining eigenvalues are listed in lexicographic order with multiplicities
$\mu_{-R}\preccurlyeq\cdots\preccurlyeq
\mu_{k}\preccurlyeq\mu_{k+1}
\preccurlyeq\cdots\preccurlyeq\mu_R$, $-R\le k\le R-1$.
\end{itemize}
\end{Def}

Note that by Lemma \ref{lem:counting_lemma} and Lemma \ref{lem:dirichlet_spectrum}, the set 
$\mathcal{T}^R$ is open in $L^2_c$. 
For the given potential $\psi^{(1)}\in i L^2_r$ with $\spec_p L(\psi^{(1)})$ simple we choose $R\in\Z_{\ge 0}$ as follows:
Denote by $\ell$ the line segment in $i L^2_r$ connecting $\psi^{(0)}$ with $\psi^{(1)}$.
By replacing $\psi^{(1)}$, if necessary, by the first intersection point of $\ell$ with $\iso_o(\psi^{(1)})$ we can assume without 
loss of generality that $\ell$ intersects $\iso_o(\psi)$ only at $\psi^{(1)}$. 
For any $\varphi\in i L^2_r$ we choose an open ball $U_\varphi$ of $\varphi$ in $L^2_c$ and 
$R_\varphi\in\Z_{\ge 0}$ so that the statements of Lemma \ref{lem:counting_lemma}
and Lemma \ref{lem:dirichlet_spectrum} hold with $R_p$ and $R_D$ replaced by $R_\varphi$. 
In view of the compactness of $\ell$, we can find $\varphi_j$, $1< j\le K$, all in $\ell$,
so that $\ell\subseteq\bigcup_{1< j\le K}U_{\varphi_j}$. Now, define 
\begin{equation}\label{eq:R}
R:=\max_{1\le j\le K} R_{\varphi_j}.
\end{equation}

\medskip

Further, by using Theorem \ref{prop:general_position} and by arguing as in the proof of \cite[Corollary 3.3]{KLT2}, one 
can construct a simple (i.e. without self intersections) continuous path 
\begin{equation*}\label{eq:gamma}
\gamma : [0,1]\mapsto\Big(\bigcup_{1< j\le K}U_{\varphi_j}\Big)\cap i L^2_r,\quad s\mapsto\psi^{(s)},
\end{equation*}
connecting $\psi^{(0)}$ with $\psi^{(1)}$, so that $\gamma\subseteq\mathcal{T}^R\cap i L^2_r$. 
In particular, we see that the statements of Lemma \ref{lem:counting_lemma} and Lemma \ref{lem:dirichlet_spectrum} 
still hold uniformly on $\gamma$ with $R_p$ and $R_D$ replaced by $R$. 
In addition, as $\gamma\subseteq\mathcal{T}^R\cap i L^2_r$, for any $\varphi\in\gamma$ the periodic eigenvalues of
$L(\varphi)$ inside $B_R$ are simple and non-real (see Lemma \ref{lem:spectrum_symmetries}).
This together with the compactness of $\gamma$ and the openness of $\mathcal{T}^R$ in $L^2_c$ implies that
there exists a connected open tubular neighborhood $U_{\rm tn}$ of $\gamma$ in $L^2_c$ so that the following holds:

\medskip

\begin{itemize}

\item[(T1)] The statements of Lemma \ref{lem:counting_lemma} and Lemma \ref{lem:dirichlet_spectrum} hold  
uniformly in $\varphi\in U_{\rm tn}$ with $R_p$ and $R_D$ replaced by $R$.

\item[(T2)] The set $U_{\rm tn}\cap i L^2_r$ is connected and for any $\varphi\in U_{\rm tn}$ the periodic eigenvalues of 
$L(\varphi)$ in the disk $B_R$ are simple, non-real, and have the symmetries of Lemma \ref{lem:spectrum_symmetries}.

\end{itemize}

\medskip

It follows from the construction of the neighborhood $\U_0$ of zero in $L^2_c$
(see the discussion ahead of Theorem \ref{th:main_near_zero}) and the property (T2) that for any 
$\varphi\in U_{\rm tn}\cap\U_0$ and for any $n\in\Z$, the disk $D_n$ contains precisely 
two periodic eigenvalues of $L(\varphi)$. For any $\varphi\in U_{\rm tn}\cap\U_0$ we list the periodic
eigenvalues as follows: for $|n|\le R$,
\[
\lambda_n^+,\lambda_n^-\in D_n\quad\text{with}
\quad \pm\im(\lambda_n^\pm)>0,
\]
and for $|n|> R$,
\[
\lambda_n^+,\lambda_n^-\in D_n\quad\text{with}
\quad\lambda_n^-\preccurlyeq\lambda_n^+.
\]
Then for any $|n|\le R$, in view of their simplicity, the periodic eigenvalues $\lambda_n^+$ and $\lambda_n^-$ of 
$L(\varphi)$ considered as functions of the potential, $\lambda_n^+,\lambda_n^- : U_{\rm tn}\cap\U_0\to\C$, are 
analytic. Since by (T2) the simplicity of these eigenvalues holds on the entire tubular neighborhood
$U_{\rm tn}$, the analytic maps above 
extend to analytic maps
\[
\lambda_n^+,\lambda_n^- : U_{\rm tn}\to\C\quad\forall |n|\le R.
\]
(This is in sharp contrast to the eigenvalues outside the disk $B_R$ which,
at least for $|n|$ sufficiently large, are not even continuous in view of the lexicographic ordering.)
We note that on $U_{\rm tn}\cap i L^2_r$,
\[
\lambda_n^-=\overline{\lambda_n^+}\quad\text{where}\quad\im(\lambda_n^+)>0\quad\forall |n|\le R.
\]
An important ingredient for the construction of the actions on $U_{\rm tn}$ is the choice
of pairwise disjoint simple continuous paths $G_n\subseteq\C$, $n\in\Z$, also referred to as {\em cuts}, that connect
$\lambda_n^-$ with $\lambda_n^+$, and continuous contours $\Gamma_n$ around $G_n$.
In a first step we define the cuts $G_n$ along the path $\gamma : [0,1]\to U_{\rm tn}\cap i L^2_r$,
$s\mapsto\psi^{(s)}$.
To simplify notation, introduce $\lambda_n^\pm(s):=\lambda_n^\pm(\psi^{(s)})$.
We note that $\forall s\in[0,1]$,
\[
\lambda_n^-(s)=\overline{\lambda_n^+(s)}\quad\text{and}\quad\im(\lambda_n^+(s))>0\quad\forall |n|\le R
\]
while
\[
\lambda_n^-(s)=\overline{\lambda_n^+(s)}\quad\text{and}\quad\im(\lambda_n^+(s))\ge 0\quad\forall |n|> R.
\]
For any $s\in[0,1]$ and $|n|>R$, we define the cuts $G_n(s):=G_n(\psi_n^{(s)})$ to be the vertical 
line segments $[\lambda_n^-(s),\lambda_n^+(s)]\subseteq D_n$, parametrized by
\[
G_n : [0,1]\times[-1,1]\to\C,\quad
(s,t)\mapsto\tau_n(s)+t\big(\lambda_n^+(s)-\lambda_n^-(s)\big)/2,
\]
where for any $n\in\Z$, $\tau_n(s):=(\lambda_n^-(s)+\lambda_n^+(s))/2$.
Note that $G_n(s,-t)=\overline{G_n(s,t)}$ for any $t\in[-1,1]$.
For $|n|\le R$ and $s$ sufficiently small so that $\psi^{(s)}\in\W_0\equiv\U_0\cap i L^2_r$, we define
$G_n(s,t)$ in a similar fashion as in the case $|n|>R$ and then show by a somewhat
lengthly but straightforward argument that the cuts $G_n(s)$ can be chosen so that they
depend continuously on $s\in[0,1]$. More precisely, the following holds:

\begin{Lem}\label{lem:deformation_G_n}
There exist continuous functions $G_n$,  $|n|\le R$, with values in the disk $B_R$,
\[
G_n : [0,1]\times[-1,1]\to B_R,\quad (s,t)\mapsto G_n(s,t),
\]
so that for any $s\in[0,1]$ the following properties hold:
\begin{itemize}
\item[(i)] $G_n(s)\equiv G_n(s,\cdot)$ is a simple $C^1$-smooth path such that
for any $t\in[-1,1]$,
\[
G_n(s,-1)=\lambda_n^-(s),\quad G_n(s,1)=\lambda_n^+(s),
\quad G_n(s,-t)=\overline{G_n(s,t)}.
\]

\item[(ii)] The paths $G_n(s)$ and $G_k(s)$ do not intersect for any integer numbers $n$ and $k$, $n\ne k$, 
such that $|n|,|k|\le R$.

\item[(iii)] There exist $0<\rho_0<1$ and $\delta_0>0$ such that 
\[
G_n(s,t)=\tau_n(s)+t\big(\lambda_n^+(s)-\lambda_n^-(s)\big)/2\quad\forall t\in[-\rho_0,\rho_0]
\]
where $\tau_n(s):=\big(\lambda_n^+(s)+\lambda_n^-(s)\big)/2$ and
\[
\im\big(G_n(s,t)\big)\ge\delta_0\quad\forall t\in[\rho_0,1].
\]
In particular, the path $G_n(s)$ intersects the real line at the unique point 
$\tau_n(s)=G_n(s,0)$.
\end{itemize}
\end{Lem}

In the next step, for any $s\in[0,1]$ and $n\in\Z$ we will choose a counterclockwise oriented simple 
$C^1$-smooth contour $\Gamma_n(s)$ in $\C$ around the cut $G_n(s)$ so that 
$\Gamma_n(s)$ is invariant under complex conjugation, $\overline{\Gamma_n(s)}=\Gamma_n(s)$, 
$\Gamma_n(s)\cap\R$ consists of two points, and $\Gamma_n(s)\cap\Gamma_k(s)=\emptyset$ for $n\ne k$.
In the case $|n|>R$ we choose $\Gamma_n(s)$ to be the boundary $\Gamma_n$ of the disk 
$\{\lambda\in\C\,|\,|\lambda-n\pi|<\pi/4\}$ whereas for $|n|\le R$, the contours $\Gamma_n(s)$ are
chosen in the disk $B_R$.
In a similar way, for any $s\in[0,1]$ and $n\in\Z$ we choose a counterclockwise oriented simple 
$C^1$-smooth contour $\Gamma_n'(s)$ around $G_n(s)$ so that $\Gamma_n'(s)$ lies in the interior 
domain of the contour $\Gamma_n(s)$, $\overline{\Gamma_n'(s)}=\Gamma_n'(s)$, $\Gamma_n'(s)\cap\R$ consists of
two points, and $\Gamma_n'(s)\cap\Gamma_k'(s)=\emptyset$ for $n\ne k$. In the case $|n|>R$ we choose
$\Gamma_n'(s)$ to be the boundary of the disk $D_n$ whereas for $|n|\le R$, the contours $\Gamma_n'(s)$
are chosen so that $\inf\limits_{s\in[0,1], |n|\le R}\mathop{\rm dist}\big(\Gamma_n'(s),\Gamma_n(s)\big)>0$.
For any $s\in[0,1]$  and $n\in\Z$ denote by $D_n'(s)$ the interior domain of the contour $\Gamma_n'(s)$ and 
by $D_n(s)$ the interior domain of $\Gamma_n(s)$. The domains $D_n'(s)$ and $D_n(s)$ are topologically
open disks with the property that $G_n(s)\subseteq D_n'(s)\subseteq D_n(s)$.
Note that by definition $D_n(s)=D_n$ for any $s\in[0,1]$ and $|n|>R$ whereas for $|n|\le R$ this 
does not necessarily hold.

The contours $\Gamma_n'(s)$, constructed above for any $s\in[0,1]$ and $n\in\Z$, are now used to 
choose open neighborhoods $U_s\subseteq U_{\rm tn}$ of $\psi^{(s)}$ in $L^2_c$ and cuts 
$G_n(s,\varphi)\subseteq D_n'(s)$ for any $\varphi\in U_s$ (and not just for $\varphi$ in $\gamma$). 
In the case $|n|\le R$ we proceed as follows: for any given $s\in[0,1]$ consider the potential $\psi^{(s)}$.
Then we choose an open ball $U_s\subseteq U_{\rm tn}$ in $L^2_c$ centered at $\psi^{(s)}$
so that for any $\varphi\in U_s$ and $k\in\Z$, the periodic eigenvalues $\lambda_k^+(\varphi)$ and
$\lambda_k^-(\varphi)$ are in the interior domain $D_k'(s)$ of $\Gamma_k'(s)$.
Then for any $|n|\le R$ we choose a $C^1$-smooth simple path $G_n(s,\varphi)$ in the interior domain $D_n'(s)$ 
of $\Gamma_n'(s)$, connecting $\lambda_n^-(\varphi)$ with $\lambda_n^+(\varphi)$ so that $G_n(s,\varphi)$ 
intersects the real axis at a unique point $\varkappa_n(s,\varphi)$. 
In the case when $\varphi\in U_s\cap i L^2_r$, in addition to the properties described above, $G_n(s,\varphi)$ is
chosen so that $\overline{G_n(s,\varphi)}=G_n(s,\varphi)$. 
It is convenient to define  $G_n(s,\varphi)$ for $\varphi=\psi^{(s)}$ by $G_n(s,\psi^{(s)}):=G_n(s)$. 
Note that then $\varkappa_n(s,\psi^{(s)})=\tau_n(s)$.
By construction, the cuts $G_n(s,\varphi)$, $|n|\le R$, do not necessarily depend continuously on $\varphi\in U_s$. 
In the case when $|n|>R$ we choose $G_n(s,\varphi)$ to be the line segment 
$G_n(\varphi):=[\lambda_n^-(\varphi),\lambda_n^+(\varphi)]$ for any $\varphi\in U_{\rm tn}$. 

\medskip

The contours $\Gamma_n(s)$, $0\le s\le 1$, and the cuts $G_n(s,\varphi)$, $\varphi\in U_s$, $n\in\Z$,
are a key ingredient not only for the construction of the actions but also for the construction of a family
of one forms used in the subsequent section to define the angles. Let us begin with the one forms.
To obtain such a family of one forms, we apply Theorem 1.3 in \cite{KT2}: shrinking the ball $U_s$ if necessary,
it follows that there exist analytic functions
\[
\zeta_n^{(s)} : \C\times U_s\to\C,\quad n\in\Z,
\]
and an integer $R_s\ge R$, used to describe the location of the zeros of these functions,
so that for any $\varphi\in U_s$ and $n\in\Z$, 
\begin{equation}\label{eq:normalization_s}
\frac{1}{2\pi}\int_{\Gamma_m(s)}\frac{\zeta_n^{(s)}(\lambda,\varphi)}{\sqrt[c]{\Delta^2(\lambda,\varphi)-4}}
\,d\lambda=\delta_{nm},\quad m\in\Z.
\end{equation}
The {\em canonical root} appearing in the denominator of the integrand in \eqref{eq:normalization_s} is defined by the infinite product
\begin{equation}\label{eq:canonical_root}
\sqrt[c]{\Delta^2(\lambda,\varphi)-4}:=2 i \prod_{k\in\Z}
\frac{\sqrt[\rm st]{(\lambda_k^+(\varphi)-\lambda)(\lambda_k^-(\varphi)-\lambda)}}{\pi_k}
\end{equation}
and the {\em standard root} $\sqrt[\rm st]{(\lambda_k^+(\varphi)-\lambda)(\lambda_k^-(\varphi)-\lambda)}$ 
is defined as the unique holomorphic function on $\C\setminus G_k(s,\varphi)$ satisfying the asymptotic relation
\begin{equation}\label{eq:standard_root}
\sqrt[\rm st]{(\lambda_k^+(\varphi)-\lambda)(\lambda_k^-(\varphi)-\lambda)}\sim -\lambda
\quad\text{as}\quad |\lambda|\to\infty.
\end{equation}
Note that by the asymptotic estimate in Lemma \ref{lem:counting_lemma} (i), for any $\varphi\in U_s$ the canonical root 
$\sqrt[c]{\Delta^2(\lambda,\varphi)-4}$ is a holomorphic function of $\lambda$ in the domain
$\C\setminus\big(\bigsqcup_{k\in\Z}G_k(s,\varphi)\big)$.
In addition, the map
\begin{equation}\label{eq:canonical_root_analytic}
\Big(\C\setminus\big(\bigsqcup_{k\in\Z}\overline{D_k'(s)}\big)\Big)\times U_s\to\C, 
\quad(\lambda,\varphi)\mapsto\sqrt[c]{\Delta^2(\lambda,\varphi)-4},
\end{equation}
is analytic and its image does not contain zero.

\begin{Rem}
The infinite product in \eqref{eq:canonical_root} is understood as the limit
\[
\lim_{N\to\infty} 2 i\prod_{|k|\le N}\frac{\sqrt[\rm st]{(\lambda^+_k(\varphi)-\lambda)(\lambda^-_k(\varphi)-\lambda)}}{\pi_k}.
\]
In order to see that this limit exists locally uniformly in $\varphi\in U_s$ and 
$\lambda\in\C\setminus\big(\bigsqcup_{k\in\Z}\overline{D_k'(s)}\big)$ one combines
the terms corresponding to $k$ and $-k$ for $1\le k\le N$ and notices that in view of Lemma \ref{lem:counting_lemma} (i),
\[
\frac{(\lambda^+_k(\varphi)-\lambda)(\lambda^+_{-k}(\varphi)-\lambda)}{(-\pi_k^2)}\cdot
\frac{(\lambda^-_k(\varphi)-\lambda)(\lambda^-_{-k}(\varphi)-\lambda)}{(-\pi_k^2)}=1+\frac{a_k}{k},
\]
where the remainder $(a_k)_{k\in\Z}$ is bounded in $\ell^2$ locally uniformly in $\varphi\in U_s$ and
$\lambda\in\C\setminus\big(\bigsqcup_{k\in\Z}\overline{D_k'(s)}\big)$ (cf. \cite{GK1}).
\end{Rem}

According to Theorem 1.3 in \cite{KT2}, there exists $R_s>R$ so that for any $\varphi\in U_s$ and $n\in\Z$ the zeros 
of the entire function $\zeta_n^{(s)}(\cdot,\varphi)$, counted with their multiplicities and listed {\em lexicographically},
$\big\{\sigma^n_k(\varphi)\,\big|\,k\in\Z\setminus\{n\}\big\}$ have the following properties:

\begin{itemize}

\item[(D1)] For any $|k|>R_s$, $k\ne n$, $\sigma^n_k(\varphi)$ is the only zero of $\zeta_n^{(s)}(\cdot,\varphi)$ in
the disk $D_k$ and the map $\sigma^n_k : U_s\to D_k$ is analytic. Furthermore, for any $|k|\le R_s$,
$k\ne n$, $\sigma^n_k(\varphi)\in B_{R_s}$.

\item[(D2)] For any $|k|>R_s$, $k\ne n$, we have that 
\[
\sigma^n_k(\varphi)=\tau_k(\varphi)+\gamma_k^2(\varphi)\ell_k^2,\quad\gamma_k(\varphi):=\lambda_k^+(\varphi)-\lambda_k^-(\varphi),
\]
uniformly in $n\in\Z$ and locally uniformly in $\varphi\in U_s$.  

\item[(D3)] The entire function $\zeta_n^{(s)}(\cdot,\varphi)$ admits the product representation
\[
\zeta_n^{(s)}(\lambda,\varphi)=-\frac{2}{\pi_n}\prod_{k\ne n}\frac{\sigma^n_k(\varphi)-\lambda}{\pi_k}.
\]
Moreover, if $\lambda_k^+(\varphi)=\lambda_k^-(\varphi)=\tau_k(\varphi)$ for some $k\ne n$ then $\tau_k(\varphi)$ is a zero of 
$\zeta_n^{(s)}(\cdot,\varphi)$.
\end{itemize}

Finally, we use the compactness of the path $\gamma$ to find finitely many numbers
$s_1<...<s_N$ in the interval $[0,1]$ so that $\big\{U_{s_1},...,U_{s_N}\big\}$ is an open cover 
of $\gamma$ in $L^2_c$. We can assume without loss of generality that $s_1=0$ and $s_N=1$.
We now shrink $U_{\rm tn}$ and set 
\begin{equation}\label{eq:U_tn}
U_{\rm tn}:=\bigcup\limits_{1\le j\le N}U_{s_j},\quad R':=\max\limits_{1\le j\le N}R_{s_j}.
\end{equation}
In the sequel we will always assume that $U_{s_1}\subseteq\U_0$ where $\U_0$ is the open ball
in $L^2_c$ centered at zero introduced at the end of Section \ref{sec:setup}.
By the construction above, $U_{\rm tn}$ is a connected tubular neighborhood of $\gamma$ so that 
the properties (T1) and (T2) hold. 

\begin{Lem}\label{lem:important}
For any $1\le k<l\le N$, $\varphi\in U_{s_k}\cap U_{s_l}$ and $m\in\Z$ the contours $\Gamma_m(s_k)$ and 
$\Gamma_m(s_l)$ are homologous within the resolvent set of $L(\varphi)$. 
\end{Lem}

The statement of this Lemma holds since by Lemma \ref{lem:deformation_G_n}, for any $m\in\Z$ the cut $G_m(s_l)$
is obtained from $G_m(s_k)$ by a continuous deformation $\big\{G_m(s)\,\big|s\in[s_k,s_l]\big\}$ satisfying the 
properties listed in Lemma \ref{lem:deformation_G_n}.

Lemma \ref{lem:important} implies the following. If $\varphi\in U_{s_k}\cap U_{s_l}$ for some $1\le k<l\le N$, then 
in view of the normalization condition \eqref{eq:normalization_s}, we conclude that for any $n\in\Z$,
\[
\int_{\Gamma_m(s_k)}\frac{\zeta_n^{(s_k)}(\lambda,\varphi)}{\sqrt[c]{\Delta^2(\lambda,\varphi)-4}}
\,d\lambda
=
\int_{\Gamma_m(s_l)}\frac{\zeta_n^{(s_l)}(\lambda,\varphi)}{\sqrt[c]{\Delta^2(\lambda,\varphi)-4}}
\,d\lambda,\quad m\in\Z.
\]
This together with Lemma \ref{lem:important} and the definition of the canonical root \eqref{eq:canonical_root}
shows that
\[
\int_{\Gamma_m(s_l)}
\frac{\zeta_n^{(s_k)}(\lambda,\varphi)-\zeta_n^{(s_l)}(\lambda,\varphi)}{\sqrt[c]{\Delta^2(\lambda,\varphi)-4}}
\,d\lambda=0,\quad m\in\Z.
\]
Now we can apply \cite[Proposition 5.2]{KT2} to conclude that for any $n\in\Z$,
\[
\zeta_n^{(s_k)}(\cdot,\varphi)=\zeta_n^{(s_l)}(\cdot,\varphi).
\]

\begin{Rem}\label{rem:uniqueness_differentials}
An important point in the above argument is that we apply \cite[Proposition 5.2]{KT2} to the difference of the
forms $\omega_n^{(s_k)}:=\frac{\zeta_n^{(s_k)}(\lambda,\varphi)\,d\lambda}{\sqrt{\Delta^2(\lambda,\varphi)-4}}$ 
and $\omega_n^{(s_l)}:=\frac{\zeta_n^{(s_l)}(\lambda,\varphi)\,d\lambda}{\sqrt{\Delta^2(\lambda,\varphi)-4}}$.
More specifically, by construction (see the ansatz (16) in \cite{KT2}), the two forms have the same ``leading'' term 
$\Omega_n$ which cancels when we take their difference. This allows us to apply \cite[Proposition 5.2]{KT2} to the difference 
of the forms and then conclude that they coincide.
\end{Rem}

The above allows us to define $\zeta_n(\lambda,\varphi)$ for $(\lambda,\varphi)\in\C\times U_{\rm tn}$ by
setting
\[
\zeta_n\big|_{\C\times U_{s_k}}:=\zeta_n^{(s_k)}.
\]
In this way we proved

\begin{Prop}\label{lem:differentials_on_U_tn}
For any $n\in\Z$, the analytic function 
\[
\zeta_n :\C\times U_{\rm tn}\to\C
\]
satisfies the above properties {\rm (D1)--(D3)} with $\zeta_n^{(s)}$ replaced by $\zeta_n$ and $R_s$ replaced by $R'$ 
uniformly on $U_{\rm tn}$. In addition, for any $\varphi\in U_{\rm tn}$ so that $\varphi\in U_{s_k}$ for some
$1\le k\le N$, the analytic function $\zeta_n$ satisfies the normalization 
conditions
\[
\frac{1}{2\pi}\int_{\Gamma_m(s_k)}\frac{\zeta_n(\lambda,\varphi)}{\sqrt[c]{\Delta^2(\lambda,\varphi)-4}}
\,d\lambda=\delta_{nm},\quad m\in\Z.
\]
\end{Prop}

\medskip

We now turn to the construction of the actions on the connected tubular neighborhood 
$U_{\rm tn}=\bigcup_{1\le j\le N}U_{s_j}$ of the path $\gamma$ defined in \eqref{eq:U_tn}.
Recall that $s_1=0$ and $U_{s_1}\subseteq\U_0$ where $\U_0$ is the open ball in $L^2_c$ centered at zero, 
introduced at the end of Section \ref{sec:setup}. For any $1\le j\le N$ and $\varphi\in U_{s_j}$ define 
the (prospective) actions
\begin{equation}\label{eq:actions}
I_n^{(j)}(\varphi):=\frac{1}{\pi}\int_{\Gamma_n(s_j)}\frac{\lambda\dot\Delta(\lambda,\varphi)}
{\sqrt[c]{\Delta^2(\lambda,\varphi)-4}}\,d\lambda,\quad n\in\Z.
\end{equation}
This definition is motivated by \cite{KLTZ} where actions were defined by formula \eqref{eq:actions}
for potentials $\varphi$ in the ball $\U_0$ (cf. Theorem \ref{th:main_near_zero}). 
Note that the contours of integration $\Gamma_n(s_j)$, $n\in\Z$, appearing in \eqref{eq:actions} 
are independent of $\varphi\in U_{s_j}$ and are contained in the set 
$\C\setminus\big(\bigsqcup_{k\in\Z}\overline{D_k'(s_j)}\big)$.
In addition, the mapping $\dot\Delta : \C\times L^2_c\to\C$ is analytic and the mapping 
\eqref{eq:canonical_root_analytic} is analytic and does not have zeros (cf. Section \ref{sec:setup}). 
This shows that for any $n\in\Z$ and $1\le j\le N$,
\[
I_n^{(j)} : U_{s_j}\to\C
\] 
is analytic. If $\varphi\in U_{s_k}\cap U_{s_l}$ for some $1\le k<l\le N$, then for any $n\in\Z$ 
the contours $\Gamma_n(s_k)$ and $\Gamma_n(s_l)$ are homologous in the resolvent set
of $L(\varphi)$ -- see Lemma \ref{lem:important}. This together with the definition of the canonical root
\eqref{eq:canonical_root} shows that for any $n\in\Z$,
\[
I_n^{(k)}(\varphi)=I_n^{(l)}(\varphi).
\]
Hence, 
\begin{equation}\label{eq:I_n}
I_n : U_{\rm tn}\to\C,\quad I_n\big|_{U_{s_j}}:=I_n^{(j)},\quad 1\le j\le N,
\end{equation}
is well defined. In this way we proved

\begin{Prop}\label{lem:actions}
For any $n\in\Z$, the function $I_n : U_{\rm tn}\to\C$, defined by \eqref{eq:I_n}, is analytic.
On $U_{s_1}\cap\U_0$ the function $I_n$ coincides with the $n$-th action
variable constructed in \cite{KLTZ}.
\end{Prop}

\begin{Rem}
Alternatively, one can analytically extend the actions by arguments similar to the ones
used in the proof of Proposition \ref{prop:beta^n_tn} to analytically extend the angles. Here we
use instead a deformation of the cuts (cf. Lemma \ref{lem:deformation_G_n}) and 
a subsequent deformation of the contours $\Gamma_n(s)$ which provide a geometrically simple 
approach.
\end{Rem}

In view of Theorem \ref{th:main_near_zero}, for any $n\in\Z$, the action
$I_n|_{U_{s_1}\cap\W_0}$ is real valued and for any $m,n\in\Z$,
\[
\big\{I_m|_{U_{s_1}\cap\U_0},I_n|_{U_{s_1}\cap\U_0}\big\}=0.
\]
We have the following

\begin{Coro}\label{coro:actions_poisson_relations}
For any $m,n\in\Z$, $\{I_m,I_n\}=0$ on $U_{\rm tn}$. Moreover, the action 
$I_n : U_{\rm tn}\to\C$ is real-valued when restricted to $U_{\rm tn}\cap i L^2_r$.
\end{Coro}

\begin{proof}[Proof of Corollary \ref{coro:actions_poisson_relations}]
Note that the analyticity of the action $I_n : U_{\rm tn}\to\C$ implies that for any $j=1,2$ the $L^2$-gradient
$\partial_j  I_n :  U_{\rm tn}\to L^2_c$ is analytic. By the definition \eqref{eq:poisson_bracket} of the Poisson bracket
we conclude that $\{I_n,I_m\} : U_{\rm tn}\to\C$ is analytic for any $m,n\in\Z$.
Since $U_{\rm tn}$ is connected and $\{I_n,I_m\}|_{U_{s_1}}=0$ by the considerations above the first statement
of the Corollary follows. The second statement follows from Proposition \ref{lem:actions}, the fact that
$I_n : U_{\rm tn}\to\C$ is real-valued when restricted to $U_{s_1}\cap\W_0$, and 
Lemma \ref{lem:real-analyticity} below.
\end{proof}

In the proof of Corollary \ref{coro:actions_poisson_relations} we used the following result about real analytic functions.

\begin{Lem}\label{lem:real-analyticity}
Let $X_r$ be a real subspace inside the complex Banach space $X_c=X_r\otimes\C$, 
$U$ a connected open set in $X_c$ such that $U\cap X_r$ is connected, and $f : U\to\C$ an analytic function. 
Assume that there exists an open ball $B(x_0)$ in $X_c$ centered at $x_0\in U\cap X_r$ such that 
$B(x_0)\subseteq U$ and $f|_{B(x_0)\cap X_r}$ is real-valued. Then $f : U\to\C$ is real-valued on $U\cap X_r$.
\end{Lem}

The Lemma follows easily from standard arguments involving Taylor's series expansions of $f$.

\section{Angles on the neighborhood $U_{\rm tn}$}\label{sec:angles_in_U_tn}
In this Section we analytically extend the angles, constructed in \cite{KLTZ} inside $U_{s_1}$, along the tubular 
neighborhood $U_{\rm tn}$ defined by \eqref{eq:U_tn}. First we need some preparation.
As by construction $R'\ge R$ where $R'$ and $R$ are defined by \eqref{eq:U_tn} and, respectively, \eqref{eq:R}, 
there is a (possibly empty) set of indices $R<|k|\le R'$ so that the statements of Lemma \ref{lem:counting_lemma} and 
Lemma \ref{lem:dirichlet_spectrum}, hold uniformly in $\varphi\in U_{\rm tn}$ with $R_p$ and $R_D$ replaced by $R'$ 
but in contrast to (T2), there could be double periodic eigenvalues of $L(\varphi)$ inside the disk $B_{R'}$. 
More specifically, double periodic eigenvalues in $B_{R'}$ can only appear in the union of disks $\bigcup_{R<|k|\le R'}D_k$, 
since by construction, $\lambda_k^+(\varphi),\lambda_k^-(\varphi)\in D_k\subseteq B_{R'}$ for any 
$\varphi\in U_{\rm tn}$ and $R<|k|\le R'$. Next we argue as in the proof of Corollary 3.3 in \cite{KLT2} to construct
a continuous path $:{\tilde\gamma} : [0,1]\to U_{\rm tn}\cap i L^2_r$ so  that $\tilde\gamma(0)=\psi^{(0)}$,
$\tilde\gamma(1)=\psi^{(1)}$, and for any potential $\varphi\in{\tilde\gamma}\big([0,1]\big)$ the operator $L(\varphi)$ 
has only simple (and hence non-real) periodic eigenvalues in the disk $B_{R'}$. In view of the compactness of 
$\tilde\gamma$ we can find a connected open neighborhood $V_{\rm tn}$ of $\tilde\gamma$ in $L^2_c$ so 
that $V_{\rm tn}\subseteq U_{\rm tn}$, $V_{\rm tn}\cap i L^2_r$ is connected, and for any $\varphi\in V_{\rm tn}$,
the operator $L(\varphi)$ has only simple, non-real periodic eigenvalues in the disk $B_{R'}$.
To simplify notation, in the sequel we denote $V_{\rm tn}$ by $U_{\rm tn}$ and $R'$ by $R$. In this way, we obtain

\begin{Lem}\label{lem:U_tn}
The neighborhood $U_{\rm tn}$ is connected, contains the potentials $\psi^{(0)}$ and $\psi^{(1)}$, 
and satisfies the properties (T1) and (T2), and Proposition \ref{lem:differentials_on_U_tn} with $R'$ replaced by $R$.
\end{Lem}

Now, we proceed with the construction of the angles.
For any given $\varphi\in U_{\rm tn}$ consider the affine curve
\[
\CC_\varphi:=\big\{(\lambda,w)\in\C^2\,\big|\,w^2=\Delta^2(\lambda,\varphi)-4\big\}
\]
and the projection $\pi_1 : \CC_\varphi\to\C$, $(\lambda,w)\mapsto\lambda$.
In fact, in what follows we work only with the following subsets of $\CC_\varphi$:
\begin{equation}\label{eq:riemann_surface}
\CC_{\varphi,R}:=\pi_1^{-1}\big(B_R\big)
\end{equation}
and
\begin{equation}\label{eq:k-th_handle}
\DD_{\varphi,k}:=\pi_1^{-1}\big(D_k\big),\quad|k|>R.
\end{equation}
By Lemma \ref{lem:U_tn}, for any $\varphi\in U_{\rm tn}$ there are precisely $4R+2$ periodic eigenvalues of 
$L(\varphi)$ inside the disk $B_R$ and they are all simple. This implies that $\CC_{\varphi,R}$ is an open 
Riemann surface with $2 R+1$ handles whose boundary in $\CC_\varphi$ is a disjoint union of two circles.
The same Lemma also implies that for any $\varphi\in U_{\rm tn}$ and $|k|>R$,
\[
\lambda_k^+(\varphi),\,\,\lambda_k^-(\varphi),\,\, \mu_k(\varphi)\in D_k,
\]
and there are no other periodic or Dirichlet eigenvalues of $L(\varphi)$ inside $D_k$.
While for any $|k|>R$, the Dirichlet eigenvalue $\mu_k$ is simple and hence depends analytically on 
$\varphi\in U_{\rm tn}$, the periodic eigenvalues $\lambda_k^+$ and $\lambda_k^-$ are not necessarily simple.
In particular, we see that $\DD_{\varphi,k}$ is either a Riemann surface diffeomorphic to $(0,1)\times\T$ or, 
when $\lambda_k^+=\lambda_k^-$, a transversal intersection in $\C^2$ of two complex disks at their centers.
Now, let $\varphi\in U_{\rm tn}$ and assume that $\varphi\in U_{s_l}$ for some $1\le l\le N$ (cf. \eqref{eq:U_tn}).
For any $k\in\Z$ consider the contour $\Gamma_k(s_l)$. By Lemma \ref{lem:important}, 
the homology class of the cycle $\Gamma_k(s_l)$ within the resolvent set of $L(\varphi)$
is independent of the choice of $1\le l\le N$ with the property that $\varphi\in U_{s_l}$.
Denote by $a_k$ the homology class in $\CC_\varphi$ of the component of $\pi_1^{-1}\big(\Gamma_k(s_l)\big)$ 
that lies on the canonical sheet  
\[
\CC_\varphi^c:=\Big\{(\lambda,w)\,\Big|\,\lambda\in\C\setminus\Big(\bigsqcup_{k\in\Z}G_k(s_l,\varphi)\Big),
w=\sqrt[c]{\Delta^2(\lambda,\varphi)-4}\Big\}
\] 
of the curve $\CC_\varphi$ where the canonical root $\sqrt[c]{\Delta^2(\lambda,\varphi)-4}$ is defined by
\eqref{eq:canonical_root}. By the discussion above, for any $\varphi\in U_{\rm tn}$ and $k\in\Z$ the class $a_k$ is 
independent of the choice of $1\le l\le N$ with the property that $\varphi\in U_{s_l}$. 
In what follows we will not distinguish between the class $a_k$ and a given 
$C^1$-smooth representative of $a_k$, which we call an {\em $a_k$-cycle}. 
In a similar way for any $1\le |k|\le R$ we define the {\em $b_k$-cycle}.
More specifically, given any $1\le |k|\le R$ and $\varphi\in U_{\rm tn}$ so that $\varphi\in U_{s_l}$
for some $1\le l\le N$, consider the intersection point $\varkappa_k(s_l,\varphi)$ of the cut $G_k(s_l,\varphi)$ 
with the real axis. Denote by $b_k$ the homology class in $\CC_\varphi$ of the cycle 
$\pi_1^{-1}\big([\varkappa_{k-1}(s_l,\varphi),\varkappa_k(s_l,\varphi)]\big)$ if $1\le k\le R$ and
the cycle $\pi_1^{-1}\big([\varkappa_k(s_l,\varphi),\varkappa_{k+1}(s_l,\varphi)]\big)$ if $-R\le k\le-1$
oriented so that the intersection index $a_k\circ b_k $ of $a_k$ with $b_k$ is equal to one. 
It is not hard to see that for any $1\le k\le R$ the class $b_k$ is independent of the choice of $1\le l\le N$ with the
property that $\varphi\in U_{s_l}$. Moreover, the first homology group
\begin{equation}\label{eq:homology_group}
H_1(\CC_{\varphi,R},\Z)=\mathop{\rm span}\big\langle a_0, a_k, b_k, 1\le |k|\le R\big\rangle_\Z\cong\Z^{4R+1}.
\end{equation}
In view of Lemma \ref{lem:U_tn} and Proposition \ref{lem:differentials_on_U_tn}, for any $n\in\Z$
the analytic functions $\zeta_n : \C\times U_{\rm tn}\to\C$ are well defined and satisfy the normalization conditions
\begin{equation}\label{eq:a-normalization}
\frac{1}{2\pi}\int_{a_k}\frac{\zeta_n(\lambda,\varphi)}{\sqrt{\chi_p(\lambda,\varphi)}}\,d\lambda=
\delta_{nk},\quad k\in\Z,
\end{equation}
where by \eqref{eq:chi_p}, $\chi_p(\lambda,\varphi)=\Delta^2(\lambda,\varphi)-4$.
For any $\varphi\in U_{\rm tn}$, $n\in\Z$, and $1\le |k|\le R$, denote by $p_{nk}$ the $b_k$-period
\[
p_{nk}\equiv p_{nk}(\varphi):=\int_{b_k}\frac{\zeta_n(\lambda,\varphi)}{\sqrt{\chi_p(\lambda,\varphi)}}\,d\lambda.
\]
(Note that $p_{nk}(\varphi)$ is well defined since $b_k$ is independent of the choice of $1\le l\le N$ 
with the property that $\varphi\in U_{s_l}$.)
Recall from Lemma \ref{lem:dirichlet_spectrum} that for any $\varphi\in U_{\rm tn}$ we list the Dirichlet eigenvalues
in lexicographic order and with multiplicities $\mu_k=\mu_k(\varphi)$, $k\in\Z$. For any $k\in\Z$ we define
the {\em Dirichlet divisor}
\begin{equation}\label{eq:dirichlet_divisor}
\mu_k^*(\varphi):=\Big(\mu_k(\varphi),\grave{m}_2\big(\mu_k(\varphi),\varphi\big)+
\grave{m}_3\big(\mu_k(\varphi),\varphi\big)\Big)
\end{equation}
where for any $\lambda\in\C$
\[
\begin{pmatrix}
\grave{m}_1(\lambda,\varphi)&\grave{m}_2(\lambda,\varphi)\\
\grave{m}_3(\lambda,\varphi)&\grave{m}_4(\lambda,\varphi)
\end{pmatrix}
:=M(1,\lambda,\varphi)
\]
and $M(x,\lambda,\varphi)$ is the fundamental solution \eqref{eq:M}.
Note that $\mu_k^*(\varphi)$ lies on the curve $\CC_\varphi$ since
by \cite[Lemma 6.6]{GK1},
\[
\big(\grave{m}_2(\mu_k)+\grave{m}_3(\mu_k)\big)^2=
\big(\grave{m}_1(\mu_k)+\grave{m}_4(\mu_k)\big)^2-4=
\Delta(\mu_k)^2-4.
\]
As the Floquet matrix $M(1,\lambda,\varphi)\in\Mat$ is analytic in 
$(\lambda,\varphi)\in\C\times L^2_c$, we conclude from \eqref{eq:dirichlet_divisor} and the discussion above 
that for any $|k|>R$, the mapping $U_{\rm tn}\to \C^2$, $\varphi\mapsto\mu_k^*(\varphi)$,
is analytic. 

\medskip

With these preparations we are ready to define for any $\varphi\in U_{\rm tn}$ and $n\in\Z$ the following 
{\em multivalued} functions,
\begin{equation}\label{4.1} 
\beta^n(\varphi):=\sum_{|k| \leq R}\int^{\mu^*_k(\varphi)}_{\lambda^-_k(\varphi)}
\frac{\zeta_n(\lambda,\varphi)}{\sqrt{\chi_p(\lambda,\varphi)}}\,d\lambda
\end{equation}
and, for any $|k| > R$
\begin{equation}\label{4.2} 
\beta^n_k(\varphi):=\int^{\mu^*_k(\varphi)}_{\lambda^-_k(\varphi)} 
\frac{\zeta_n(\lambda,\varphi)}{\sqrt{\chi_p(\lambda,\varphi)}}\,d\lambda.
\end{equation}
Let us discuss the definition of these path integrals in more detail. In \eqref{4.1} the paths of integration are chosen in 
$\CC_{\varphi,R}$. The integrals in \eqref{4.1} depend on the choice of the path but only up to integer 
linear combinations of the periods of the one form $\frac{\zeta _n(\lambda)}{\sqrt{\chi _p(\lambda )}}\,d\lambda$ 
with respect to the basis of cycles $(a_k)_{|k|\leq R}$ and $(b_k)_{1\le |k|\le R}$ on $\CC_{\varphi,R}$.
More specifically, if $|n|>R$, then since $\int_{a_k}\frac{\zeta _n(\lambda)}{\sqrt{\chi _p(\lambda )}}\,d\lambda=0$
for any $|k|\le R$, the quantity $\beta^n(\varphi)$ is defined modulo the lattice
\begin{equation}\label{4.3} 
\mathcal{L}_n\equiv \mathcal{L}_n(\varphi ):= 
\Big\{ \sum_{1\le |k| \leq R} m_k\,p_{nk}(\varphi)\,\Big|\, m_k\in\Z,\,
1\le |k|\leq R\Big\}
\end{equation}
whereas, if $|n|\le R$, it is defined modulo $\mathop{\rm span}\langle 2\pi,\mathcal{L}_n\rangle_\Z$
since $\frac{1}{2\pi}\int_{a_k}\frac{\zeta _n(\lambda)}{\sqrt{\chi _p(\lambda )}}\,d\lambda=\delta_{nk}$. 
In \eqref{4.2}, for $|k| > R$, the path of integration is chosen to be in 
${\mathcal D}_{\varphi,k}$. If $k\ne n$, the integral $\beta^n_k(\varphi)$ is independent of the path whereas
for $k=n$ with $|n|>R$, the integral $\beta^n_n(\varphi)$ is defined modulo $2\pi$ on $U_{\rm tn}\backslash\mathcal{Z}_n$ 
where
\begin{equation}\label{eq:Z_n}
\mathcal{Z}_n :=\big\{\varphi\in U_{\rm tn}\,\big|\,\gamma ^2_n=0\big\}.
\end{equation}
Since for any $|n|>R$, $\gamma_n^2$ is analytic on $U_{\rm tn}$ (cf. \cite[Lemma 12.4]{GK1}),
$\mathcal{Z}_n$ is an analytic subvariety of $U_{\rm tn}$. Furthermore, by the proof of Lemma 7.9 and
Proposition 7.10 in \cite{GK1}, $\mathcal{Z}_n\cap i L^2_r$ is a real analytic submanifold of $U_{\rm tn}\cap i L^2_r$
of real codimension two. Defining $\gamma_n:=\lambda^+_n-\lambda^-_n$ for $|n|\le R$ it is clear that
$\mathcal{Z}_n :=\big\{\varphi\in U_{\rm tn}\,\big|\,\gamma ^2_n=0\big\}=\emptyset$
for $|n|\le R$. Note that the integrals in \eqref{4.1} and \eqref{4.2} exist since whenever
$\lambda ^+_k \not=\lambda^-_k$, the integrands have a singularity of the form $(\lambda-\lambda^\pm _k)^{-1/2}$
for $\lambda$ near $\lambda ^\pm _k$, and hence are integrable. 
If $\lambda ^+_k = \lambda ^-_k$ (and hence by the construction of $U_{\rm tn}$ necessarily $|k|>R$),
the singularity of the integrand in \eqref{4.2} is removable since, by Lemma \ref{lem:U_tn} and 
Lemma \ref{lem:differentials_on_U_tn} (see property (D3)), the root $\sigma^n_k$ of $\zeta_n(\lambda)$ in $D_k$ 
then coincides with $\tau_k$ which, in view of \eqref{eq:canonical_root}, is a zero of the denominator since
$\sqrt[\rm st]{(\lambda_k^+-\lambda)(\lambda_k^--\lambda)}=\tau_k-\lambda$. For convenience, we introduce
\begin{equation}\label{eq:w_k}
w_k(\lambda):=\sqrt[\rm st]{(\lambda_k^+-\lambda)(\lambda_k^--\lambda)}.
\end{equation}

\medskip

The aim of this Section is to show that 
\begin{equation}\label{eq:angles_U_tn}
\theta_n(\varphi):={\tilde\beta}^n(\varphi)+\sum_{|k|>R}\beta^n_k(\varphi),\quad
\varphi\in U_{\rm tn}\setminus\mathcal{Z}_n,\quad n\in\Z,
\end{equation}
are bona fide angle variables, conjugate to the action variables introduced in Section \ref{sec:actions_in_U_tn}.
In particular the series in \eqref{eq:angles_U_tn} converges.
The definition \eqref{eq:angles_U_tn} is motivated by \cite{KLTZ} where the angle $\theta_n$
was defined by a formula of the type as in \eqref{eq:angles_U_tn} for potentials in $U_{s_1}\setminus\mathcal{Z}_n$
(cf. Theorem \ref{th:main_near_zero}). 
Here ${\tilde\beta}^n(\varphi)$ denotes an analytic branch of the multivalued function $\beta^n$,
defined by \eqref{4.1}, which is well defined modulo $2\pi$ and obtained by analytic extension
of the corresponding function defined on $U_{s_1}\setminus\mathcal{Z}_n$ in \cite{KLTZ}.
We begin by studying the integrals $\beta^n_k(\varphi)$, $|k|>R$.
First we establish the following estimates.

\begin{Lem}\label{lem:beta^n_k-asymptotics} 
For any $n\in\Z$ and $|k|>R$, $k\ne n$,
\[ 
\beta^n_k=O\left(\frac{|\gamma_k|+|\mu _k-\tau _k|}{|k-n|}\right)
\]
locally uniformly in $\varphi\in U_{\rm tn}$ and uniformly in $n\in\Z$.
\end{Lem}

\begin{proof}[Proof of Lemma \ref{lem:beta^n_k-asymptotics}] 
We follow the arguments of the proof of Lemma 5.1 in \cite{GK1}.
It follows from \eqref{4.2}, the normalization 
condition \eqref{eq:a-normalization}, and the discussion above that $\beta^n_k = 0$ for 
any $\varphi\in U_{\rm tn}$, $n\in\Z$, and $|k|>R$ with $k\ne n$, such that $\mu_k\in\{\lambda^+_k,\lambda^-_k\}$. 
Moreover, in view of the normalization condition \eqref{eq:a-normalization}, 
the value of $\beta^n_k$ with $|k|>R$, $k\ne n$, will not change if we replace in formula \eqref{4.2} the eigenvalue
$\lambda^-_k$ by $\lambda ^+_k$. 
Hence it is sufficient to prove the claimed estimate only for those $\varphi\in U_{\rm tn}$, $n\in\Z$, and $|k|>R$ with 
$k\ne n$, for which
\begin{equation*}
\mu_k\ne\{\lambda^+_k,\lambda^-_k\}\quad\text{and}
\quad |\mu_k-\lambda^-_k|\le|\mu_k-\lambda^+_k|.
\end{equation*}
Using the definition \eqref{eq:canonical_root} of the canonical root and (D3)
we write for $\lambda\in D_k$,
\begin{equation}\label{4.6} 
\frac{\zeta_n(\lambda )}{\sqrt[c]{\chi _p(\lambda )}} =
\frac{\sigma ^n_k - \lambda}{w_k(\lambda )}\,\zeta ^n_k(\lambda )
\quad \mbox { with } \quad
\zeta ^n_k(\lambda ) := \frac{i}{w_n(\lambda )}
\prod _{r \not= k,n} \frac{\sigma ^n_r - \lambda }{w_r(\lambda )}.
\end{equation}
Arguing as in \cite[Corollary 12.7]{GK1} and in view of the asymptotics for $(\sigma^n_m)_{m\ne n}$ 
given by Lemma \ref{lem:U_tn} and Proposition \ref{lem:differentials_on_U_tn} (property (D2)) we get
\begin{equation}\label{eq:zeta^n_k}
\zeta^n_k(\lambda)=O\Big(\frac{1}{|n - k|}\Big),\quad k\ne n,\,\,\lambda\in D_k,
\end{equation}
locally uniformly on $U_{\rm tn}$.
To estimate $\beta^n_k$, we parametrize the interval $[\lambda_k^-,\mu_k]\subseteq\C$ in formula \eqref{4.2},  
$t\mapsto\lambda(t):=\lambda^-_k+t d_k$, $t\in[0,1]$, where $d_k:=\mu_k-\lambda^-_k$ and $d_k\ne 0$,  
and then use \eqref{4.2}, \eqref{4.6}, and \eqref{eq:zeta^n_k}, to get
\begin{eqnarray}
      |\beta^n_k|&=&\Big| \int ^{\mu _k}_{\lambda ^-_k}\frac{\sigma ^n_k -\lambda }{w_k(\lambda )}\,
      \zeta ^n_k(\lambda )\,d\lambda\Big|\nonumber\\
                        &=&O\Big( \frac{1}{|n-k|}\Big)\!\int ^1_0\!\Big|\frac{\sigma^n_k-\lambda (t)}{t d_k}\Big|^{1/2} 
                         \Big|\frac{\sigma^n_k-\lambda (t)}{\lambda ^+_k-\lambda (t)}\Big|^{1/2}\,|d_k|\,dt.\label{eq:beta^n_k}
\end{eqnarray}
Further, we have for any $0\le t\le 1$,
\[ 
\frac{\sigma ^n_k - \lambda (t)}{\lambda ^+_k - \lambda (t)} = 1 +
\frac{\sigma ^n_k - \lambda ^+_k}{\lambda ^+_k - \lambda (t)}.
\]
Using that $|\mu _k - \lambda ^-_k| \leq |\mu _k - \lambda ^+_k|$ one easily sees that
$|\lambda ^+_k - \lambda (t)| \geq |\gamma _k| / 2$ for any $0 \leq t \leq 1$.
Then, by Lemma \ref{lem:U_tn} and Proposition \ref{lem:differentials_on_U_tn} (property (D2))
and the triangle inequality,
$|\sigma^n_k-\lambda ^+_k|\le|\sigma^n_k-\tau _k|+|\gamma _k|/2=O(\gamma _k)$.
This implies that
\begin{equation}\label{4.7} 
\Big|\frac{\sigma^n_k-\lambda (t)}{\lambda ^+_k-\lambda (t)}\Big| = O(1), \quad t\in[0,1],
\end{equation}
locally uniformly on $U_{\rm tn}$. 
On the other hand,
   \begin{equation}\label{4.8} 
         \Big|\frac{\sigma^n_k-\lambda(t)}{t d_k}\Big|^{1/2}= 
                   \frac{\big(|\sigma^n_k-\lambda^-_k|+|d_k |\big)^{1/2}}{\sqrt{t}\,|d_k|^{1/2}}= 
          O\left(\frac{\big(|\gamma_k|+|d_k|\big)^{1/2}}{\sqrt{t}\,|d_k|^{1/2}}\right).
   \end{equation}
Combining,  \eqref{eq:beta^n_k} with \eqref{4.7} and \eqref{4.8} we finally obtain for $|k|>R$, $k\ne n$,
\begin{equation}\label{eq:|beta|}
|\beta ^n_k|=\Big\arrowvert\int ^{\mu _k}_{\lambda ^-_k} 
\frac{\sigma^n_k-\lambda }{w_k(\lambda)}\,\zeta^n_k(\lambda )\,d\lambda
\Big\arrowvert=O\left(\frac{\big(|\gamma_k|+|d_k|\big)^{1/2}
|d_k|^{1/2}}{|n-k|}\right).
\end{equation}
The claimed estimate then follows by the Cauchy-Schwarz inequality.
Going through the arguments of the proof one sees that the estimates for
$\beta ^n_k$ hold uniformly in $n\in\Z$ and locally uniformly
on $U_{\rm tn}$.
\end{proof}

\medskip

The next result claims that $\beta^n_k$, $|k|>R$, are analytic. More precisely, the following holds.

\begin{Lem}\label{lem:beta^n_k-analytic} 
Let $n \in \Z$ be arbitrary.
\begin{itemize}
\item[(i)] For any $|k| > R$ with $k\ne n$, $\beta ^n_k$ is analytic on
$U_{\rm tn}$.
\item[(ii)] For $|n| > R$, $\beta^n_n$ is defined modulo $2\pi$. It is analytic on $U_{\rm tn} \backslash{\mathcal Z}_n$ when 
considered modulo $\pi$.
\end{itemize}
\end{Lem}

\begin{proof}[Proof of Lemma \ref{lem:beta^n_k-analytic}] 
We follow the arguments of the proof of Lemma 15.2 in \cite{GK1}.
(i) Fix $k \not= n$ with $|k| > R$. In addition to the analytic
subvariety ${\mathcal Z}_k$ introduced in \eqref{eq:Z_n} we also consider
   \[ {\mathcal E}_k := \big\{ \varphi \in U_{\rm tn}\,\big|\,\mu_k\in 
   \{\lambda ^\pm_k\}\big\} = 
   \big\{ \varphi \in  U_{\rm tn}\,\big|\, \Delta (\mu _k) = 2(-1)^k \big\}
   \]
which clearly is also an analytic subvariety of $U_{\rm tn}$. We prove that
$\beta ^n_k$ is analytic on $U_{\rm tn} \backslash ({\mathcal Z}_k \cup{\mathcal E}_k)$
when taken modulo $\pi$, extends continuously to $U_{\rm tn}$ and has weakly
analytic restrictions to ${\mathcal Z}_k$ and ${\mathcal E}_k$. It then
follows by \cite[Theorem A.6]{GK1} that $\beta ^n_k$ is analytic on
$U_{\rm tn}$. To prove that $\beta ^n_k$ is analytic on 
$U_{\rm tn}\backslash ({\mathcal Z}_k \cup {\mathcal E}_k)$ it suffices to prove its
differentiability. Note that $\lambda ^\pm _k$ are simple eigenvalues on
$U_{\rm tn}\backslash{\mathcal Z}_k$, but as they are listed in
lexicographic order, they are not necessarily continuous. For any given
$\varphi\in U_{\tt tn}\backslash ({\mathcal Z}_k \cup {\mathcal E}_k)$, according to
\cite[Proposition 7.5]{GK1}, in a neighborhood of $\varphi$ there exist two analytic
functions $\varrho ^\pm _k$ such that $\{ \lambda ^+_k, \lambda ^-_k\} =
\{ \varrho ^+_k, \varrho ^-_k \} $. Choose $\varrho ^\pm _k$ so that
$\mathop{\rm dist}\big([ \varrho ^-_k, \mu _k], \varrho ^+_k\big)\ge\frac{1}{3} |\gamma_k|$. 
In view of the normalization condition in Proposition \ref{lem:differentials_on_U_tn} we 
can write
   \[ \beta ^n_k = \int ^{\mu _k}_{\varrho ^-_k}\frac{\zeta _n(\lambda )}
      {\sqrt[\ast ]{\chi _p(\lambda )}}\,d\lambda
   \]
where the integral is taken along any path from $\varrho ^-_k$ to $\mu _k$
inside $D_k$ that besides its end point(s) does not contain any point
of $G_k$. The sign of the $\ast $-root along such a path is the one
determined by
   \[ \sqrt[\ast ]{\chi _p(\mu _k)} = \grave{m} _2(\mu _k) + \grave
      {m} _3(\mu _k) .
   \]
As in \eqref{4.6} write
   \[ \frac{\zeta_n(\lambda)}{\sqrt[c]{\chi _p(\lambda )}} = 
      \frac{\sigma ^n_k-\lambda }{w_k(\lambda )}\,\zeta ^n_k(\lambda )
   \]
and let $d_k := \mu _k - \varrho^-_k$. With the substitution
$\lambda(t)=\varrho^-_k+t d_k$ one has $w_k(\lambda )^2 = t d_k(\lambda(t) - \varrho^+_k)$ and 
as by assumption, $|\lambda(t) - \varrho ^+_k|\ge|\gamma _k / 3|$ for $0 \le t\le 1$ and $\psi$
in a neighborhood of $\varphi$ the argument of $\lambda(t)-\varrho_k^+$ is contained in an interval 
of length strictly smaller than $\pi$.
Hence the square root $\sqrt{\lambda (t) - \varrho ^+_k}$ can be
chosen to be continuous in $t$ and analytic near $\varphi $. With the
appropriate choice of the root $\sqrt{d_k}$ it then follows that
   \[ \beta ^n_k = \int ^1_0 \frac{1}{\sqrt{t}} \frac{\sigma ^n_k - \lambda }
      {\sqrt{\lambda - \varrho ^+_k}}\,\zeta ^n_k(\lambda ) \sqrt{d_k}\,dt
   \]
is differentiable at $\varphi $. Next let us show that $\beta ^n_k$ is continuous on 
$U_{\rm tn}$. By the previous considerations, $\beta ^n_k$ is continuous in all points of
$U_{\rm tn}\backslash ({\mathcal Z}_k \cup {\mathcal E}_k)$. By \eqref{eq:|beta|} and 
$\beta^n_k\big\arrowvert _{{\mathcal E}_k} = 0$, it follows that $\beta ^n_k$ is
continuous at points of ${\mathcal E}_k$. It thus remains to prove that
$\beta ^n_k$ is continuous in the points of ${\mathcal Z}_k \backslash
{\mathcal E}_k$. First we show that $\beta ^n_k \big\arrowvert _{{\mathcal
Z}_k \backslash {\mathcal E}_k}$ is continuous. Indeed, on ${\mathcal Z}_k$, 
$\lambda ^-_k = \tau _k$ and $(\sigma ^n_k - \lambda ) / w_k(\lambda )=1$ hence
   \[ \beta ^n_k \big\arrowvert _{{\mathcal Z}_k \backslash {\mathcal E}
      _k} = \int ^{\mu _k}_{\tau _k} \zeta ^n_k (\lambda )\,d\lambda
      \big\arrowvert _{{\mathcal Z}_k \backslash {\mathcal E}_k}
   \]
and it follows that $\beta ^n_k \big\arrowvert _{{\mathcal Z}_k \backslash
{\mathcal E}_k}$ is continuous. As ${\mathcal E}_k$ is closed in
$U_{\rm tn}$, it then remains to show that for any sequence 
$(\varphi^{(j)})_{j \geq 1} \subseteq U_{\rm tn}\backslash ({\mathcal Z}_k\cup {\mathcal E}_k)$ 
converging to an element $\varphi \in {\mathcal Z}_k
\backslash {\mathcal E}_k$ one has
   \[ \beta ^n_k(\varphi ^{(j)}) \underset {j \to \infty }
      {\longrightarrow } \beta ^n_k(\varphi ) .
   \]
Without loss of generality we may assume that $\inf\limits_j\big|(\mu _k - \tau _k)(\varphi ^{(j)})\big| > 0$,
   \[ \big|\lambda ^+_k(\varphi ^{(j)}) - \mu _k(\varphi ^{(j)})\big| \ge
      \big|\lambda ^-_k(\varphi ^{(j)}) - \mu _k(\varphi ^{(j)})\big|
   \]
(otherwise go to a subsequence of $\varphi ^{(j)}$ and/or, if necessary,
switch the roles of $\lambda ^+_k$ and $\lambda ^-_k$), and
   \[ \sqrt[\ast]{\chi _p\big(\mu _k(\varphi ^{(j)})\big)} =
      \sqrt[c]{\chi _p\big(\mu _k(\varphi ^{(j)})\big)} .
   \]
Let $0 < \varepsilon\ll 1$. As 
$\lim\limits_{j \to \infty }\gamma _k(\varphi^{(j)}) = 0$ 
as well as
$\lim\limits_{j \to \infty} d_k(\varphi ^{(j)}) = \mu _k
      (\varphi ) - \tau _k(\varphi ) \not= 0$
there exists $j_0 \geq 1$ so that
   \begin{equation}
   \label{4.9} \Big\arrowvert \frac{\gamma _k(\varphi ^{(j)})}{d_k
               (\varphi ^{(j)})}\Big\arrowvert \leq \varepsilon / 2 \quad
               \forall j \geq j_0 .
   \end{equation}
With the substitution $\lambda(t)=\lambda_k^-+t d_k$, $d_k=\mu_-\lambda_k^-$, one gets
\[ 
\beta ^n_k(\varphi^{(j)}) = \left( \int ^\varepsilon _0 + \int ^1_\varepsilon \right)
      \frac{\sigma ^n_k - \lambda }{w_k(\lambda )}\,\zeta ^n_k(\lambda )\,d_k\,dt.
\]
By using the estimates \eqref{4.8}--\eqref{4.9} one sees
that for any $j \geq j_0$
   \[ \Big| \int ^\varepsilon _0 \frac{\sigma ^n_k - \lambda }
      {w_k(\lambda )}\,\zeta ^n_k(\lambda )\,d_k\,dt \Big|
      \le C \sqrt{\varepsilon }
   \]
where $C > 0$ is a constant independent of $j$. To estimate the integral
   \[ J_\varepsilon (\varphi ^{(j)}) := \int ^1_\varepsilon \frac{\sigma ^n_k
      - \lambda }{w_k(\lambda )}\,\zeta ^n_k(\lambda )\,d_k\,dt
   \]
note that for any $\varepsilon \le t \leq 1$ and $j \geq j_0$
   \[ \Big\arrowvert \frac{\gamma ^2_k / 4}{(\tau _k-\lambda )^2}\Big\arrowvert
      = \Big \arrowvert t\,\frac{2 d_k}{\gamma _k} - 1 \Big\arrowvert ^{-2}
      \leq 3^{-2}
   \]
and thus according to the definition \eqref{eq:standard_root} of the standard root
   \[ w_k(\lambda ) = (\tau _k - \lambda ) \sqrt[+]{1 - \frac{\gamma ^2_k / 4}
      {(\tau _k - \lambda ) ^2}}
   \]
for $\varepsilon \leq t \leq 1$ and $\varphi ^{(j)}$ with $j \geq j_0$. Thus
   \[ J_\varepsilon (\varphi ^{(j)}) = \int ^1_\varepsilon \left( 1 +
      \frac{\sigma ^n_k - \tau _k}{\tau _k - \lambda } \right) \left( 1 -
      \frac{\gamma ^2_k / 4}{(\tau _k - \lambda )^2} \right)^{-1/2}\!\!\!\!\zeta^n_k(\lambda )\,d_k\,dt .
   \]
As $\sigma ^n_k - \tau _k = \gamma ^2_k \ell
^2_k$ (property (D2)) one has for $\varepsilon \leq t \leq 1$
   \[ \frac{\sigma ^n_k - \tau _k}{\tau _k - \lambda } = \frac{\gamma ^2_k
      \ell ^2_k}{\tau _k - \lambda ^-_k - t d_k} \to 0 \quad
      \mbox{as } j \to \infty
   \]
as well as
   \[ \frac{\gamma _k}{\tau _k - \lambda } = \frac{\gamma _k}{\tau _k -
      \lambda ^-_k - t d_k} \to 0 \quad \mbox{ as } j
      \to \infty .
   \]
By dominated convergence it then follows that
   \[ J_\varepsilon (\varphi ^{(j)}) \to \int ^1_\varepsilon \zeta
      ^n_k(\lambda , \varphi )\,d_k(\varphi )\,dt \quad \mbox{ as } j
      \to \infty .
   \]
But $\beta^n_k(\varphi )-\int^1_\varepsilon\zeta^n_k(\lambda,\varphi )\,d_k(\varphi )\,dt = O(\varepsilon )$ 
by the continuity of $\zeta ^n_k$ in $\lambda $. Altogether we showed that there exits $j_1 \geq j_0$,
depending on $\varepsilon $, so that for any $j \geq j_1$,
   \[ \big|\beta ^n_k(\varphi^{(j)}) - \beta ^n_k(\varphi )\big| \le C
      \sqrt{\varepsilon }
   \]
where $C$ can be chosen independently of $\varepsilon $. As $\varepsilon >
0$ is arbitrarily small this establishes the claimed convergence.
It remains to check the weak analyticity on ${\mathcal E}_k$ and ${\mathcal
Z}_k$. On ${\mathcal E}_k$ this is trivial since $\beta ^n_k \big\arrowvert
_{{\mathcal E}_k} \equiv 0$. On ${\mathcal Z}_k$ we can write
$ \beta ^n_k = \int ^{\mu _k}_{\tau _k} \varepsilon _k \zeta ^n_k (\lambda )\,d\lambda$
where $\varepsilon _k \in \{ 0, \pm 1\} $ is defined on ${\mathcal Z}_k
\backslash {\mathcal E}_k$ by
$ \sqrt[\ast ]{\chi _p(\mu _k)} = \varepsilon _k \sqrt[c]{\chi _p(\mu_k)}$
and is zero on ${\mathcal Z}_k \cap {\mathcal E}_k$. Now consider a disk
$D \subseteq {\mathcal Z}_k$. As ${\mathcal E} _k$ is an analytic subvariety
one has either $D \subseteq {\mathcal Z}_k \cap {\mathcal E}_k$ -- in which
case $\beta ^n_k \big\arrowvert _D \equiv 0$ -- or $D \cap {\mathcal E}_k$
is finite. As $\int ^{\mu _k}_{\tau _k} \zeta ^n_k(\lambda )\,d\lambda \big
\arrowvert _D$ is analytic and $\beta ^n_k$ is continuous on $D$ it then
follows that $\int ^{\mu _k}_{\tau _n} \zeta ^n_k (\lambda )\,d\lambda \big
\arrowvert _D \equiv 0$ or $\varepsilon _k \big\arrowvert _{D \backslash
{\mathcal E} _k}$ is constant. In both cases it follows that $\beta ^n_k
\big\arrowvert _D$ is analytic. This establishes the claimed analyticity.

Item (ii) is proved in an analogous way except for the fact that
switching from $\lambda ^-_n$ to $\varrho ^-_n$ may change the value of
$\beta ^n_n$ by $\pi $ in view of the normalization condition in Proposition \ref{lem:differentials_on_U_tn}. 
Hence we have
$\beta ^n_n = \int ^{\mu _n}_{\varrho ^-_n} \frac{\zeta _n(\lambda )}
{\sqrt[\ast ]{\chi _p(\lambda )}}\,d\lambda$ modulo $\pi$.
\end{proof}

\medskip

As a next step we show that by choosing appropriate integration paths in \eqref{4.1} the quantity
$\beta^n(\varphi)$ is well defined modulo $2\pi$ on $U_{\rm tn}$ and real valued,
when restricted to $U_{\rm tn}\cap i L^2_r$. In \cite{KLTZ} it is proved  that such a choice
is possible in a small neighborhood of zero in $L^2_c$. More specifically, we have the following

\begin{Lem}\label{lem:near_zero} 
For any $\varphi\in U_{s_1}$, where $U_{s_1}\subseteq U_{\rm tn}\cap\U_0$ and $\U_0$ is the open ball in 
$L^2_c$ centered at zero, given by Theorem \ref{th:main_near_zero}, and for any $n\in\Z$ and $|k| \leq R$, 
the path of integration in 
$\beta^n_k=\int ^{\mu ^\ast _k}_{\lambda^-_k} \frac{\zeta _n(\lambda )}{\sqrt{\chi _p(\lambda )}}\,d\lambda$ 
can be chosen to lie in the handle ${\mathcal D}_{\varphi,k}\equiv\pi_1^{-1}(D_k)$ which is a Riemann surface
biholomorphic to $\big\{z\in\C\,\big|\,1<|z|<2\big\}$ since $\lambda ^-_k \ne\lambda ^+_k$.
For $|k|\le R$ with $k\ne n$  the quantity $\beta^n_k$ is well defined and analytic on $U_{s_1}$ 
whereas for $k=n$ it is defined modulo $2\pi$ and as such analytic.
Furthermore, for any $n\in\Z$ and $|k|\le R$, $\beta^n_k$ is {\em real-valued} when restricted to 
$U_{s_1}\cap i L^2_r$. As a consequence,  for any $n\in\Z$, the sum $\beta^n=\sum _{|k|\le R}\beta^n_k$ is 
real valued when restricted to $U_{s_1}\cap i L^2_r$. 
\end{Lem}

Let us introduce for any $\varphi\in U_{\rm tn}$ the set
\[
M_R\equiv M_R(\varphi):=\big\{\mu_k\,\big|\,|k|\le R\big\}.
\]
By the definition \eqref{eq:dirichlet_divisor} for any $|k|\le R$, the Dirichlet divisor $\mu_k^*(\varphi)$ on 
the Riemann surface $\mathcal{C}_{\varphi,R}$ is uniquely determined 
by $\mu_k(\varphi)$. Similarly we introduce
\[
\Lambda^-_R\equiv\Lambda^-_R(\varphi):=\big\{\lambda_k^-\,\big|\,|k|\le R\big\}.
\]
and recall that $\lambda^-_k(\varphi)$, $|k|\le R$, are simple periodic eigenvalues of $L(\varphi)$
that satisfy $\im(\lambda^-_k)<0$.
Then, we have the following important

\begin{Prop}\label{prop:beta^n_tn}
After shrinking the tubular neighborhood $U_{\rm tn}$ of the path $\gamma : [0,1]\to U_{\rm tn}\cap i L^2_r$,
if necessary, so that $U_{\rm tn}$ and $U_{\rm tn}\cap i L^2_r$ are still connected, there exist for any 
$\varphi\in U_{\rm tn}$ a bijective correspondence 
\[
\Lambda^-_R(\varphi)\to M_R(\varphi),\quad z\mapsto\mu_\varphi(z),
\]
and for any $z\in\Lambda^-_R(\varphi)$ a continuous, piecewise $C^1$-smooth path
$\PP^*\big[z,\mu_\varphi(z)\big]$ on $\mathcal{C}_{\varphi,R}$ from $z^*=(z,0)$ to $\mu_\varphi^*(z)$ 
so that for any $n\in\Z$,
\begin{equation}\label{eq:betatilde^n}
{\tilde\beta}^n(\varphi):=\sum_{z\in\Lambda_R^-(\varphi)}\int_{\PP^*[z,\mu_\varphi(z)]}
\frac{\zeta_n(\lambda,\varphi)}{\sqrt{\Delta^2(\lambda,\varphi)-4}}\,d\lambda,
\end{equation}
defined modulo $2\pi$ if $|n|\le R$, is analytic on $U_{\rm tn}$, and real valued when restricted to
$U_{\rm tn}\cap i L^2_r$.
\end{Prop}

\begin{Rem}
Since the curve $\gamma$ is simple and can be assumed to be piecewise $C^1$-smooth, by shrinking $U_{\rm tn}$ once 
more if necessary, we can ensure that \eqref{eq:betatilde^n} is single valued on $U_{\rm tn}$ and hence there is no 
need to define in the case $|n|\le R$, ${\tilde\beta}^n$ modulo $2\pi$ as stated in Proposition \ref{prop:beta^n_tn}. 
We chose to state Proposition \ref{prop:beta^n_tn} as is because we want to use its proof without modification in the subsequent Section
when constructing angle coordinates in a tubular neighborhood of the isospectral set $\iso_o(\psi)$ which is {\em not} simply 
connected. 
\end{Rem}

\begin{proof}[Proof of Proposition \ref{prop:beta^n_tn}]
For a given $0\le\tau\le 1$, let $\psi:=\gamma(\tau)$ and for any $z\in\C$ and $\varepsilon>0$ denote by 
$D^\varepsilon(z)$ the open disk of radius $\varepsilon$ in $\C$ centered at $z$.
Let $\overline{D}^\varepsilon(z)$ be the corresponding closed disk.
We refer to the point $Q_z$ on the boundary of $D^\varepsilon(z)$ with the smallest real part among 
the points of boundary of the disk as the {\em base point} of $D^\varepsilon(z)$. 
Choose $\varepsilon\equiv\varepsilon_\psi>0$ so that the following holds: For any 
$z,z'\in\Lambda_R(\psi)\cup M_R(\psi)$ where $\Lambda_R(\psi)\equiv\big\{\lambda_k^\pm\,\big|\,|k|\le R\big\}$ 
we have 
\begin{itemize} 
\item[(1)] $\overline{D}^\varepsilon(z)\cap \overline{D}^\varepsilon(z')=\emptyset$ if $z\ne z'$.
\item[(2)] $\overline{D}^\varepsilon(z)\setminus\{z\}$ does not contain any periodic eigenvalue of $L(\psi)$.
\item[(3)] $\overline{D}^\varepsilon(z)\subseteq B_R$ and if $z\in\Lambda^-_R(\psi)$,
$\overline{D}^\varepsilon(z)\subseteq\big\{\lambda\in\C \,\big|\,\im(\lambda)<0\big\}$.
\end{itemize}
Now we choose an open ball $U_\psi\subseteq U_{\rm tn}$ in $L^2_c$ centered at $\psi$ so that for any 
$\varphi\in U_{\rm tn}$ the following holds: For any Dirichlet eigenvalue $\mu\in M_R(\psi)$ with (algebraic)
multiplicity $m_D\ge 1$ there exist exactly $m_D$ Dirichlet eigenvalues of $L(\varphi)$ in the disk $D^\varepsilon(\mu)$
and for any periodic eigenvalue $\nu\in\Lambda^-_R(\psi)$ there exists a unique periodic eigenvalue $\nu_\varphi$ of 
$L(\varphi)$ in $D^\varepsilon(\nu)$. Denote by $B^\bullet_{R,\psi}$ the complement in $B_R$ of the union of disks
$D^\varepsilon(z)$, $z\in\Lambda_R(\psi)\cup M_R(\psi)$,
\[
B^\bullet_{R,\psi}:=B_R\setminus\Big(\bigcup_{z\in\Lambda_R(\psi)\cup M_R(\psi)}D^\varepsilon(z)\Big).
\]
Furthermore, in what follows, $\pi_1$ will denote the projection $\pi_1 : \C^2\to\C$, $(\lambda,w)\mapsto\lambda$,  
onto the first component of $\C^2$. Then the set $\big(\pi_1|_{\mathcal{C}_{\psi,R}}\big)^{-1}\big(B^\bullet_{R,\psi}\big)$ is 
a connected Riemann surface obtained from the Riemann surface $\mathcal{C}_{\psi,R}$ by removing a certain number 
of open disks. Choose an arbitrary bijection 
\begin{equation*}\label{eq:initial_correspondence}
\Lambda^-_R(\psi)\to M_R(\psi),\quad\nu\mapsto\mu_\psi(\nu).
\end{equation*}
Then, for any $\nu\in\Lambda^-_R(\psi)$, take a $C^1$-smooth curve in $B^\bullet_{R,\psi}$ that connects
the base point of $D^\varepsilon(\nu)$ with the base point of $D^\varepsilon(\mu_\psi(\nu))$.
We denote this curve by $Y_\psi[\nu,\mu_\psi(\nu)]$. In this way we obtain a set of $2R+1$ curves 
\[
\big\{Y_\psi[\nu,\mu_\psi(\nu)]\,\big|\,\nu\in\Lambda^-_R(\psi)\big\}
\]
in $B^\bullet_{R,\psi}$. By construction, these curves depend on the choice of $\psi$ and the bijective correspondence.
Now take $\varphi\in U_\psi$. For any $z\in\Lambda^-_R(\varphi)$ with the condition that $z\in D^\varepsilon(\nu)$ where 
$\nu\in\Lambda^-_R(\psi)$, denote by $\mathcal{P}_{\varphi,\psi}[z,\mu_\varphi(z)]$ the concatenated path
\begin{equation}\label{eq:concatenated_path}
[z,Q_\nu]\cup Y_\psi[\nu,\mu_\psi(\nu)]\cup[Q_{\mu_\psi(\nu)},\mu_\varphi(z)]
\end{equation}
where by $\mu_\varphi(z)$ we denote one of the Dirichlet eigenvalues in $M_R(\varphi)$ that lies in the disk 
$D^\varepsilon(\mu_\psi(\nu))$ and $[a,b]$ denotes the line segment connecting two complex numbers $a,b\in\C$.
Requiring that every Dirichlet eigenvalue in $M_R(\varphi)$ appears as the endpoint of such a  
concatenated curve precisely once, we construct for the considered potential $\varphi\in U_\psi$ a bijective correspondence 
\begin{equation}\label{eq:correspondence}
\Lambda^-_R(\varphi)\to M_R(\varphi),\quad z\mapsto\mu_\varphi(z)\equiv\mu_{\varphi,\psi}(z).
\end{equation}
\begin{Rem}\label{rem:multiple_choice}
In the case when $\mu_\psi(\nu)$ has algebraic multiplicity $m_D\ge 2$ there are multiple options for the choice
of the third part $[Q_{\mu_\psi(\nu)},\mu_\varphi(z)]$ of the path $\mathcal{P}_{\varphi,\psi}[z,\mu_\varphi(z)]$.
This leads to different bijective correspondences \eqref{eq:correspondence}. The final result however will {\em not} depend on
the choice of the correspondence in \eqref{eq:correspondence}.
\end{Rem}
We now want to lift the path $\mathcal{P}_{\varphi,\psi}[z,\mu_\varphi(z)]$ constructed above to the Riemann surface
$\mathcal{C}_{\varphi,R}$. Let us first treat the case where $\mu_\psi(\nu)$ is not a ramification point of 
$\mathcal{C}_{\psi,R}$, i.e. $\mu_\psi(\nu)\notin\Lambda_R(\psi)$. Then the preimage 
$\big(\pi_1|_{\mathcal{C}_{\psi,R}}\big)^{-1}\big(\overline{D}^\varepsilon(\mu_\psi(\nu))\big)$
consists of two disjoint closed disks. Denote by $Q_{\mu_\psi(\nu),\psi}^*$ the lift of $Q_{\mu_\psi(\nu)}$ which 
is in the disk containing the Dirichlet divisor $\mu_\psi^*(\nu)$. Similarly, for $\varphi\in U_\psi$ the preimage
$\big(\pi_1|_{\mathcal{C}_{\varphi,R}}\big)^{-1}\big(\overline{D}^\varepsilon(\mu_\psi(\nu))\big)$
consists of two disjoint disks. We denote by $Q_{\mu_\psi(\nu),\varphi}^*$ the lift of $Q_{\mu_\psi(\nu)}$ which 
is in the disk containing the Dirichlet divisor $\mu_\varphi^*(z)$. Denote by $Q_{\nu,\varphi}^*$ the starting point 
of the lift $Y^*_{\varphi,\psi}[\nu,\mu_\psi(\nu)]$ of $Y_\psi[\nu,\mu_\psi(\nu)]$ by 
$\big(\pi_1|_{\mathcal{C}_{\varphi,R}}\big)^{-1}$ that is ending at $Q_{\mu_\psi(\nu),\varphi}^*$. By construction, 
\[
\pi_1(Q_{\nu,\varphi}^*)=Q_\nu\quad\text{and}\quad\pi_1(Q_{\mu_\psi(\nu),\varphi}^*)=Q_{\mu_\psi(\nu)}.
\]
This yields a uniquely determined lift $\mathcal{P}^*_{\varphi,\psi}[z,\mu_\varphi(z)]$ of
$\mathcal{P}_{\varphi,\psi}[z,\mu_\varphi(z)]$ to $\mathcal{C}_{\varphi,R}$, starting at
$(z,0)$ and ending at $\mu_\varphi^*(z)$.
Now, let us turn to the case where $\mu_\psi(\nu)\in\Lambda_R(\psi)$ and hence $\mu_\psi(\nu)$ is 
a ramification point of $\pi_1|_{\mathcal{C}_{\psi,R}} : \mathcal{C}_{\psi,R}\to\C$.
Then the preimage 
$\big(\pi_1|_{\mathcal{C}_{\psi,R}}\big)^{-1}\big(\overline{D}^\varepsilon(\mu_\psi(\nu))\big)$
is a closed disk in $\mathcal{C}_{\psi,R}$ and the restriction of the map $\pi_1$ to this disk is two-to-one except 
at the branching point $\mu^*_\psi(\nu)$ where it is one-to-one.
Denote by $Q^*_{\mu_\psi(\nu),\psi}$ one of the preimages of $Q_{\mu_\psi(\nu)}$.
For $\varphi\in U_\psi$ the preimage
$\big(\pi_1|_{\mathcal{C}_{\varphi,R}}\big)^{-1}\big(\overline{D}^\varepsilon(\mu_\psi(\nu))\big)$
is also a closed disk in $\mathcal{C}_{\varphi,R}$ and contains a unique branched point in its interior. 
We denote by $Q_{\mu_\psi(\nu),\varphi}^*$ the preimage of $Q_{\mu_\psi(\nu)}$ defined uniquely by the condition
that the map
\[
U_\psi\to\C^2,\quad\varphi\mapsto Q_{\mu_\psi(\nu),\varphi}^*,
\]
is analytic and $Q_{\mu_\psi(\nu),\varphi}^*\big|_{\varphi=\psi}=Q^*_{\mu_\psi(\nu),\psi}$. 
Then, by proceeding in the same way as in the previous case, we construct the point $Q_{\nu,\varphi}^*$, the lift 
$Y^*_{\varphi,\psi}[\nu,\mu_\psi(\nu)]$ of $Y_{\varphi,\psi}[\nu,\mu_\psi(\nu)]$, and the uniquely determined lift 
$\mathcal{P}^*_{\varphi,\psi}[z,\mu_\varphi(z)]$ of $\mathcal{P}_{\varphi,\psi}[z,\mu_\varphi(z)]$ to 
$\mathcal{C}_{\varphi,R}$ that starts at $(z,0)$, passes consecutively through $Q_{\nu,\varphi}^*$ and 
$Q_{\mu_\psi(\nu),\varphi}^*$, and then ends at the Dirichlet divisor $\mu_\varphi^*(z)$.
It then follows that for any $n\in\Z$,
\begin{equation}\label{eq:beta_tilde}
\beta^n_{\psi}(\varphi):=\sum_{z\in\Lambda_R^-(\varphi)}
\int_{\PP^*_{\varphi,\psi}[z,\mu_\varphi(z)]}
\frac{\zeta_n(\lambda,\varphi)}{\sqrt{\Delta^2(\lambda,\varphi)-4}}\,d\lambda,
\end{equation}
is well defined and analytic on $U_\psi$.  
Indeed, if $\nu\in\Lambda^-_R(\psi)$ so that $\mu_\psi(\nu)$ is a simple Dirichlet eigenvalue of $L(\psi)$ the integral
$\int_{\PP^*_{\varphi,\psi}[z,\mu_\varphi(z)]}
\frac{\zeta_n(\lambda,\varphi)}{\sqrt{\Delta^2(\lambda,\varphi)-4}}\,d\lambda$
is clearly analytic on $U_\psi$ where $z\in\Lambda^-_R(\varphi)$ is the periodic eigenvalue of $L(\varphi)$ in the
disk $D^\varepsilon(\nu)$.
In case $\mu_\psi(\nu)$ is a Dirichlet eigenvalue of $L(\psi)$ of multiplicity $m_D\ge 2$, denote by 
$\nu_j$, $1\le j\le m_D$, the periodic eigenvalues of $L(\psi)$ such that $\mu_\psi(\nu_j)=\mu_\psi(\nu)$ and
for any $\varphi\in U_\psi$ by $z_j$, $1\le j\le m_D$, the periodic eigenvalues of $L(\varphi)$ with 
$z_j\in D^\epsilon(\nu_j)$. Then the argument principle implies that
\begin{equation}\label{eq:argument_principle}
\sum_{1\le j\le m_D}
\int_{\PP^*_{\varphi,\psi}[z_j,\mu_\varphi(z_j)]}
\frac{\zeta_n(\lambda,\varphi)}{\sqrt{\Delta^2(\lambda,\varphi)-4}}\,d\lambda,
\end{equation}
is analytic in $U_\psi$ -- see Lemma \ref{lem:argument_principle} in the Appendix for more details.
One can easily check that the value \eqref{eq:beta_tilde} does {\em not} depend on the choice described 
in Remark \ref{rem:multiple_choice}.
Note that the restriction of $\beta^n_{\psi}$ to $U_\psi\cap i L^2_r$ is {\em not}
necessarily real valued. For latter reference we record that by construction, for any $\varphi\in U_\psi$, the path 
$\PP^*_{\varphi,\psi}[z,\mu_\varphi(z)]$ is a concatenation of three paths which we write as
\begin{equation}\label{eq:concatenation*}
\PP^*_{\varphi,\psi}[z,\mu_\varphi(z)]=
[z,Q_\nu]^*\cup Y^*_{\varphi,\psi}[\nu,\mu_\psi(\nu)]\cup[Q_{\mu_\psi(\nu)},\mu_\varphi(z)]^*.
\end{equation}
Now consider $0\le\tau_0\le 1$ with $\tau_0\ne\tau$ and let $\psi_0:=\gamma(\tau_0)$.
Assume that following the same construction as for $\psi=\gamma(\tau)$, one finds an open ball
$U_{\psi_0}\subseteq U_{\rm tn}$ of $\psi_0$ in $L^2_c$ with $U_\psi\cap U_{\psi_0}\ne\emptyset$
and for any $\varphi\in U_{\psi_0}$ a system of paths 
$\PP^*_{\varphi,\psi_0}[z,\mu_{\varphi,\psi_0}(z)]$, $z\in\Lambda^-(\varphi)$,
so that the restriction of $\beta^n_{\psi_0}$ to $U_{\psi_0}\cap i L^2_r$ is {\em real valued}.
We now want to show that one can modify the system of paths 
$\PP^*_{\varphi,\psi}[z,\mu_\varphi(z)]$, $z\in\Lambda^-(\varphi)$, so that for any $n\in\Z$,
\begin{equation}\label{eq:beta1=beta2}
\beta^n_\psi(\varphi)=\beta^n_{\psi_0}(\varphi),\quad
\forall\varphi\in U_\psi\cap U_{\psi_0},
\end{equation}
where $\beta^n_\psi(\varphi)$ denotes the value \eqref{eq:beta_tilde} with the system of paths
modified. Then by Lemma \ref{lem:real-analyticity} for any $n\in\Z$, the quantity $\beta^n_\psi$ will be real 
valued when restricted to $U_\psi\cap i L^2_r$.
Indeed, take $\varphi\in U_\psi\cap U_{\psi_0}$. The first problem arises if the bijection
\[
\mu_{\varphi,\psi_0} : \Lambda^-_R\to M_R(\varphi)
\] 
corresponding to the neighborhood $U_{\psi_0}$ is different from the bijection
\[
\mu_\varphi \equiv\mu_{\varphi,\psi} : \Lambda^-_R\to M_R(\varphi)
\]
corresponding to the neighborhood $U_\psi$.
In such a case, denote by $\sigma_\varphi : \Lambda^-_R(\varphi)\to\Lambda^-_R(\varphi)$ the
permutation with the property that $\mu_{\varphi,\psi_0}=\mu_{\varphi,\psi}\circ\sigma_\varphi$.
Since $\sigma_\varphi$ can be written as a product of transpositions we can assume without loss
of generality that $\sigma_\varphi$ is a transposition interchanging two periodic eigenvalues 
$z$ and $z'$ in $\Lambda^-_R(\varphi)$. Then $z\in D^\varepsilon(\nu)$ and $z'\in D^\varepsilon(\nu')$
for some $\nu,\nu'\in\Lambda^-_R(\psi)$. In view of \eqref{eq:concatenation*},
\[
\PP^*_{\varphi,\psi}[z,\mu_\varphi(z)]=
[z,Q_\nu]^*\cup Y^*_{\varphi,\psi}[\nu,\mu_\psi(\nu)]\cup[Q_{\mu_\psi(\nu)},\mu_\varphi(z)]^*
\]
and
\[
\PP^*_{\varphi,\psi}[z',\mu_\varphi(z')]=
[z',Q_{\nu'}]^*\cup Y^*_{\varphi,\psi}[\nu',\mu_\psi(\nu')]\cup[Q_{\mu_\psi(\nu')},\mu_\varphi(z')]^*.
\]
Choose a path $Y^*_{\nu,\nu'}$ on the Riemann surface
\[
\mathcal{C}_{\psi,R}^\bullet:=\big(\pi_1|_{\mathcal{C}_{\psi,R}}\big)^{-1}(B_{R,\psi}^\bullet)
\]
that connects the point $Q^*_\nu$ with $Q^*_{\nu'}$.  Let $Y_{\nu,\nu'}$ be its projection 
into $B_{R,\psi}^\bullet$ by the map $\pi_1$ and denote by $Y^*_{\nu,\nu';\varphi}$ the unique lift of
$Y_{\nu,\nu'}$ by $\pi_1|_{\mathcal{C}_{\varphi,R}} : \mathcal{C}_{\varphi,R}\to\C$ which connects
$Q^*_{\nu,\varphi}$ with $Q^*_{\nu',\varphi}$. We then replace $\PP^*_{\varphi,\psi}[z,\mu_\varphi(z)]$ by
the concatenated curve
\[
[z,Q_\nu]^*\cup Y^*_{\nu,\nu';\varphi}\cup Y^*_{\varphi,\psi}[\nu',\mu_\psi(\nu')]\cup
[Q_{\mu_\psi(\nu')},\mu_\varphi(z')]^*
\]
and $\PP^*_{\varphi,\psi}[z',\mu_\varphi(z')]$ by
\[
[z',Q_{\nu'}]^*\cup (Y^*_{\nu,\nu';\varphi})^{-1}\cup Y^*_{\varphi,\psi}[\nu,\mu_\psi(\nu)]\cup
[Q_{\mu_\psi(\nu)},\mu_\varphi(z)]^*
\]
where $(Y^*_{\nu,\nu';\varphi})^{-1}$ denotes the path obtained from $Y^*_{\nu,\nu';\varphi}$ by traversing it
in opposite direction. Since this change of paths does not affect the value of 
$\beta^n_\psi(\varphi)$ we can assume without loss of generality that 
$\mu_{\varphi,\psi}=\mu_{\varphi,\psi_0}$.
Therefore, for any $z\in\Lambda^-_R(\varphi)$, the paths
$\PP^*_{\varphi,\psi}[z,\mu_\varphi(z)]$ and 
$\PP^*_{\varphi,\psi_0}[z,\mu_{\varphi,\psi_0}(z)]$ in $\mathcal{C}_{\varphi,R}$ have the same initial points
and the same end points. 
Now, take $\varphi_0\in U_\psi\cap U_{\psi_0}$. Then, for any $z\in\Lambda^-_R(\varphi_0)$ choose an arbitrary 
point on $Y^*_{\varphi_0,\psi}[\nu,\mu_\psi(\nu)]$ of the path 
$\PP^*_{\varphi_0,\psi}[z,\mu_{\varphi_0}(z)]$ and modify it by adding a bouquet of $a$-cycles and $b$-cycles of 
the Riemann surface $\mathcal{C}_{\varphi_0,R}$, contained in 
$\big(\pi_1|_{\mathcal{C}_{\varphi_0,R}}\big)^{-1}(B_{R,\psi}^\bullet)$ so that when modified 
in this way the path $\PP^*_{\varphi_0,\psi}[z,\mu_{\varphi_0}(z)]$ is homologous to
$\PP^*_{\varphi_0,\psi_0}[z,\mu_{\varphi_0,\psi_0}(z)]$ on $\mathcal{C}_{\varphi_0,R}$.
Hence, for any $n\in\Z$, $\beta^n_\psi(\varphi_0)$ defined by \eqref{eq:beta_tilde} with these modified paths equals
$\beta^n_{\psi_0}(\varphi_0)$. For any $\nu\in\Lambda^-_R(\psi)$, let 
\[
{\widetilde Y}_\psi[\nu,\mu_\psi(\nu)]:=\pi_1\big({\widetilde Y}^*_{\varphi_0,\psi}[\nu,\mu_\psi(\nu)]\big)
\]
where ${\widetilde Y}^*_{\varphi_0,\psi}[\nu,\mu_\psi(\nu)]$ denotes the middle part of the path 
$\PP^*_{\varphi_0,\psi}[z,\mu_{\varphi_0}(z)]$ modified as described above. 
Note that ${\widetilde Y}_\psi[\nu,\mu_\psi(\nu]\subseteq B^\bullet_R$.
Now, use ${\widetilde Y}_\psi[\nu,\mu_\psi(\nu)]$ instead of $Y_\psi[\nu,\mu_\psi(\nu)]$ to construct the paths
$\PP_{\varphi,\psi}[z,\mu_\varphi(z)]$ and their lifts $\PP^*_{\varphi,\psi}[z,\mu_\varphi(z)]$ for any
$\varphi\in U_\psi$ and $z\in\Lambda^-(\varphi)$.
By this construction and since \eqref{eq:beta_tilde} is independent on the choice described in 
Remark \ref{rem:multiple_choice}, the equality \eqref{eq:beta1=beta2} follows.
As mentioned above, Lemma \ref{lem:real-analyticity} then implies that ${\tilde\beta}_\psi^n : U_\psi\to\C$
is real valued when restricted to $U_\psi\cap i L^2_r$ for any $n\in\Z$. 

Since the set $\gamma\big([0,1]\big):=\big\{\gamma(t)\,\big|\,t\in[0,1]\big\}\subseteq U_{\rm tn}\cap i L^2_r$ is 
compact in $L^2_c$ there are finitely many potentials $\varphi_1,...,\varphi_M\in\gamma([0,1])$, open balls
$U(\varphi_1),...,U(\varphi_M)$ in $U_{\rm tn}$, and for any $1\le j\le M$ and any $\varphi\in U(\varphi_j)$, a system of paths 
\[
\big\{\PP^*_{\varphi,\varphi_j}[z,\mu_{\varphi,\varphi_j}(z)]\,\big|\,z\in\Lambda_R^-(\varphi)\big\},
\]
constructed as described above so that
\[
\gamma\big([0,1]\big)\subseteq\bigcup_{j=1}^M U(\varphi_j)\subseteq U_{\rm tn}.
\]
Without loss of generality we can assume that $\varphi_1=\psi^{(0)}\in\mathcal{U}_0\cap U_{\rm tn}$ where 
$\mathcal{U}_0$ is the open ball in $L^2_c$  centered at zero given by Theorem \ref{th:main_near_zero}.
By choosing for any $\varphi\in U(\varphi_1)$ the system of paths as above and, at the same time, as described in 
Lemma \ref{lem:near_zero}, it follows that $\beta^n_{\varphi_1} : U(\varphi_1)\to\C$ is real-valued on 
$U(\varphi_1)\cap i L^2_r$. By modifying the systems of paths as described above we conclude that for any $1\le j<k\le M$ 
and $\varphi\in U(\varphi_j)\cap U(\varphi_k)$ the quantities $\beta^n_{\varphi_j}(\varphi)$ and $\beta^n_{\varphi_k}(\varphi)$
are real valued and coincide up to a linear combination of $a$ and $b$-periods with integer coefficients.
Taking into account Lemma \ref{lem:lattice_imaginary} it then follows that modulo $2\pi$,
\[
\beta^n_{\varphi_j}(\varphi)\equiv\beta^n_{\varphi_k}(\varphi),\quad
\forall\varphi\in U(\varphi_j)\cap U(\varphi_k).
\]
By setting ${\tilde\beta}^n|_{U(\varphi_j)}:=\beta^n_{\varphi_j}$ for $1\le j\le M$ we then obtain
a well defined function modulo $2\pi$,
\begin{equation}\label{eq:beta_global}
{\tilde\beta}^n : U_{\rm tn}'\to\C,\quad U_{\rm tn}':=\bigcup_{j=1}^M U(\varphi_j),
\end{equation}
so that ${\tilde\beta}^n|_{U_{\rm tn}'\cap i L^2_r}$ is real-valued.
This completes the proof of Proposition \ref{prop:beta^n_tn}.
\end{proof}

In the proof of Proposition \ref{prop:beta^n_tn} we used the following

\begin{Lem}\label{lem:lattice_imaginary}
For any $\varphi\in U_{\rm tn}\cap i L^2_r$ and for any $n\in\Z$, the lattice of periods $\mathcal{L}_n(\varphi)$,
introduced in \eqref{4.3}, consists of imaginary numbers, $\mathcal{L}_n(\varphi)\subseteq i\R$.
\end{Lem}

\begin{proof}[Proof of Lemma \ref{lem:lattice_imaginary}]
Take $\varphi\in U_{\rm tn}\cap i L^2_r$ and assume that $\varphi\in U_{s_k}$ for some $1\le k\le N$ (see \eqref{eq:U_tn}).
Then, by Proposition \ref{lem:differentials_on_U_tn} for any $n\in\Z$,
\begin{equation}\label{eq:normalization_condition1}
\int_{\Gamma_m(s_k)}\frac{\zeta_n(\lambda,\varphi)}{\sqrt[c]{\Delta^2(\lambda,\varphi)-4}}\,d\lambda=2\pi\delta_{mn},
\quad m\in\Z,
\end{equation}
where by construction,
\begin{equation}\label{eq:conjugation1}
\overline{\Gamma_m(s_k)}=\Gamma_m(s_k).
\end{equation}
It follows from Lemma \ref{lem:spectrum_symmetries} and the definition of the canonical root 
\eqref{eq:canonical_root} that
\begin{equation}\label{eq:conjugation2}
\overline{\big(\sqrt[c]{\Delta^2(\lambda,\varphi)-4}\big)}=-\sqrt[c]{\Delta^2(\overline{\lambda},\varphi)-4}.
\end{equation}
By taking the complex conjugate of both sides of \eqref{eq:normalization_condition1}, then using
\eqref{eq:conjugation1} and \eqref{eq:conjugation2}, and finally by passing to the complex conjugate variable in the integral,
we obtain that for any $n\in\Z$,
\begin{equation}\label{eq:normalization_condition2}
\int_{\Gamma_m(s_k)}\frac{\overline{\zeta_n(\overline{\lambda},\varphi)}}
{\sqrt[c]{\Delta^2(\lambda,\varphi)-4}}\,d\lambda=2\pi\delta_{mn},
\quad m\in\Z.
\end{equation}
Now, by comparing \eqref{eq:normalization_condition1} and \eqref{eq:normalization_condition2} we conclude
from \cite[Proposition 5.2]{KT2} (cf. Remark \ref{rem:uniqueness_differentials}) that
\begin{equation}\label{eq:zeta_symmetry}
\overline{\zeta_n(\overline{\lambda},\varphi)}=\zeta_n(\lambda,\varphi)
\end{equation}
for any $\varphi\in U_{\rm tn}\cap i L^2_r$, $\lambda\in\C$, and $n\in\Z$.
In particular, we see that in this case $\zeta_n(\lambda,\varphi)$ is real-valued for $\lambda\in\R$.
Now consider the $b_l$-period of the one form 
$\frac{\zeta_n(\lambda,\varphi)}{\sqrt{\Delta^2(\lambda,\varphi)-4}}\,d\lambda$ for some $1\le |l|\le R$.
For ease of notation we assume that $l\ge 1$. Then, by construction 
$\pi_1(b_l)=[\varkappa_{l-1}(s_k,\varphi),\varkappa_l(s_k,\varphi)]$, and therefore
\begin{equation}\label{eq:p_nl}
p_{nl}=\int_{b_l}\frac{\zeta_n(\lambda,\varphi)}
{\sqrt{\Delta^2(\lambda,\varphi)-4}}\,d\lambda\in\left\{
\pm2\int_{\varkappa_{l-1}(s_k,\varphi)}^{\varkappa_l(s_k,\varphi)}
\frac{\zeta_n(\lambda,\varphi)}
{\sqrt[c]{\Delta^2(\lambda,\varphi)-4}}\,d\lambda\right\}
\end{equation}
where the integration is performed over the real interval $[\varkappa_{l-1}(s_k,\varphi),\varkappa_l(s_k,\varphi)]$.
By construction $B_R\cap\R$ does not contain periodic eigenvalues of $L(\varphi)$.
In view of \eqref{eq:zeta_symmetry} and the property that $\Delta(\lambda,\varphi)\in(-2,2)$ for $\lambda\in B_R\cap\R$ 
(see \eqref{eq:real_line}) we then conclude from \eqref{eq:p_nl} that the $b_l$-period $p_{nl}$ is an imaginary number. 
\end{proof}

\medskip

Arguing as in the proof of Lemma \ref{lem:beta^n_k-asymptotics} one obtains the following
estimates for ${\tilde\beta}^n(\varphi)$.

\begin{Lem}\label{lem:beta_tilde-asymptotics}
For any $n\in\Z$,
\[ 
{\tilde\beta}^n=O(1/n)\quad\text{as}\quad|n|\to\infty,
\]
locally uniformly in $\varphi\in U_{\rm tn}$.  
\end{Lem}

In what follows we assume that the tubular neighborhood $U_{\rm tn}$ is chosen as in
Proposition \ref{prop:beta^n_tn}. We summarize the results obtained in this section so far as follows.

\begin{Prop}\label{Theorem 4.4}
For any $n \in \Z$, the following statements hold:
\begin{itemize}
\item[(i)] $\sum _{\underset{\scriptstyle{k \not= n}}{|k| > R}} \beta ^n_k$ converges locally
uniformly on $U_{\rm tn}$ to an analytic function on $U_{\rm tn}$ which is
of the order $o(1)$ as $|n|\to\infty$. 
\item[(ii)] If $|n| > R$, then $\beta ^n_n$ is defined modulo $2\pi$ on
$U_{\rm tn}\backslash{\mathcal Z}_n$. It is analytic, when taken modulo $\pi$.
\item[(iii)] For any $\varphi \in U_{\rm tn}$ and $n\in\Z$, the quantity ${\tilde\beta}^n(\varphi)$ is defined 
modulo $2\pi$ and is analytic on $U_{\rm tn}$. Furthermore, 
${\tilde\beta}^n=O(1/n)$ locally uniformly on $U_{\rm tn}$.
\end{itemize}
\end{Prop}

\begin{proof}[Proof of Proposition \ref{Theorem 4.4}]
(i) By Lemma \ref{lem:beta^n_k-asymptotics} and the Cauchy-Schwarz inequality,
\begin{align*} 
&\sum _{\underset{\scriptstyle{k \not= n}}{|k| > R}} |\beta ^n_k|
=\sum _{0 < |k - n| \le \frac{|n|}{2},|k|>R} |\beta ^n_k| +
\sum _{|k - n| >\frac{|n|}{2}, |k|>R} |\beta ^n_k| \\
&\le C \Big( \sum _{|k| \geq |n|/2} |\gamma _k|^2 +
|\mu _k - \tau _k|^2 \Big) ^{1/2}
+ C \big( \| \gamma \|_0 + \| \mu - \tau \|_0 \big)
\Big( \sum _{m \geq \frac{|n|}{2}} \frac{1}{m^2}
\Big)^{1/2}
\end{align*}
where here $\gamma = (\gamma _m)_{m \in \Z}$ and 
$\mu - \tau = (\mu_m - \tau _m)_{m \in \Z}$. Both latter displayed terms converge to zero locally
uniformly on $U_{\rm tn}$ as $|n|$ tends to infinity whence we have 
$\sum_{\underset{\scriptstyle {k \not= n}}{|k| > R}} \beta ^n_k = o(1)$. 
By Lemma \ref{lem:beta^n_k-analytic} (i) and \cite[Theorem A.4]{GK1}, it then follows that
$\sum _{\underset{\scriptstyle{k \not= n}}{|k| > R}}\beta ^n_k$ is
analytic on $U_{\rm tn}$. Item (ii) is proved in Lemma \ref{lem:beta^n_k-analytic} (ii) and
item (iii) follows from Proposition \ref{prop:beta^n_tn} and Lemma \ref{lem:beta_tilde-asymptotics}.
\end{proof}

As a consequence the angle variables \eqref{eq:angles_U_tn},
\begin{equation}\label{eq:angles_U_tn'}
\theta_n(\varphi)={\tilde\beta}^n(\varphi)+\sum_{|k|>R}\beta^n_k(\varphi),\quad
\varphi\in U_{\rm tn}\setminus\mathcal{Z}_n,\quad n\in\Z,
\end{equation}
are well defined.

\begin{Prop}\label{prop:angles_tn}
For any $n\in\Z$, the angle variable $\theta_n$ are defined on $U_{\rm tn}\setminus\mathcal{Z}_n$ modulo $2\pi$ by 
\eqref{eq:angles_U_tn'}. For $|n|\le R$, $\theta_n$ is analytic and for $|n|>R$, it is analytic on 
$U_{\rm tn}\setminus\mathcal{Z}_n$ when taken modulo $\pi$. Moreover, \eqref{eq:angles_U_tn'} is real valued when 
restricted to $\big(U_{\rm tn}\setminus\mathcal{Z}_n\big)\cap i L^2_r$. 
\end{Prop}

\begin{proof}[Proof of Proposition \ref{prop:angles_tn}]
In view of  Proposition \ref{prop:beta^n_tn} and Proposition \ref{Theorem 4.4} it only remains to prove that for any $n\in\Z$
the angle variable $\theta_n$ is real valued when restricted to $\big(U_{\rm tn}\setminus\mathcal{Z}_n\big)\cap i L^2_r$. 
This follows from  the analyticity of $\theta_n$, Lemma \ref{lem:real-analyticity}, and the fact that, by Theorem \ref{th:main_near_zero}, 
$\theta_n$ is real valued when restricted to the neighborhood of zero 
$\big(U_{s_1}\setminus\mathcal{Z}_n\big)\cap i L^2_r\subseteq\W_0\setminus\mathcal{Z}_n$ (cf. \eqref{eq:U_tn}).
\end{proof}

\medskip

By combining Lemma \ref{lem:near_zero}, Theorem \ref{th:main_near_zero} 
and Proposition \ref{prop:angles_tn}, we obtain by arguing as in the proof of Corollary \ref{coro:actions_poisson_relations},
the following theorem, which summarizes the results obtained in Section \ref{sec:actions_in_U_tn} and 
Section \ref{sec:angles_in_U_tn}.

\begin{Th}\label{th:commutation_relations_tn} 
For any $k,n\in\Z$ we have $\{I_n,I_k\}=0$ on $U_{\rm tn}$,
$\{\theta _n,I_k\}=\delta _{nk}$ on $U_{\rm tn}\backslash {\mathcal Z}_n$, and
$\{\theta_n,\theta _k\}=0$ on $U_{\rm tn}\setminus\big({\mathcal Z} _n\cup{\mathcal Z}_k\big)$.
Furthermore, the action $I_n$ and the angle $\theta_n$ are real valued when restricted to $U_{\rm tn}\cap i L^2_r$ and,
respectively,  $\big(U_{\rm tn}\setminus\mathcal{Z}_n\big)\cap i L^2_r$.
\end{Th}

\section{Actions and angles in $U_{\rm iso}$}\label{sec:actions_and_angles_in_U_iso}
Let $\psi^{(1)}\in i L^2_r$ and assume that $\psi^{(1)}$ has simple periodic spectrum.
Following the construction of action-angle coordinates along the path $\gamma : [0,1]\to U_{\rm tn}\cap i L^2_r$, 
$s\mapsto\psi^{(s)}$, in Section \ref{sec:actions_in_U_tn} and Section \ref{sec:angles_in_U_tn} we construct in 
the present Section action-angle coordinates in an open neighborhood of the isospectral set $\iso_o(\psi^{(1)})$. 
We prove additional properties of these coordinates that will be used in the subsequent sections.

For any $n\in\Z$, let $\Gamma_n:=\Gamma_n(s_N)$, $\Gamma_n':=\Gamma_n'(s_N)$ where
$\Gamma_n(s_N)$ and $\Gamma_n'(s_N)$ are the contours corresponding to the cut
$G_n(s_N)$ constructed in Section \ref{sec:actions_in_U_tn}. Here $s=s_N=1$ is the endpoint of 
the deformation $\{G_n(s)\,|\,n\in\Z\}_{s\in[0,1 ]}$ of cuts given by Lemma \ref{lem:deformation_G_n}.
For any $\psi\in\iso_o(\psi^{(1)})$ we choose an open ball $U_\psi$ of $\psi$ in $L^2_c$,
an integer $R_\psi\in\Z_{\ge 0}$, and positive constants $\varepsilon_0>0$ and $\delta_0>0$ such that 
the following holds:

\begin{itemize}
\item[(I1)] The statements of Lemma \ref{lem:counting_lemma}, Lemma \ref{lem:dirichlet_spectrum},
and Lemma \ref{lem:spectral_bands} in the Appendix with $\varepsilon=\varepsilon_0$ and
$\delta=\delta_0$, hold uniformly in $\varphi\in U_\psi$ with $R_p$, $R_D$, and $\tilde R$
replaced by $R_\psi$. Moreover, for any $\varphi\in U_\psi$ the $4 R_\psi+2$ periodic eigenvalues of $L(\varphi)$
inside the disk $B_{R_\psi}$ are {\em simple}.
\item[(I2)] For any $\varphi\in U_\psi$ and for any $n\in\Z$ the pair of periodic eigenvalues $\lambda^\pm_n$ is 
contained in the interior domain $D_n'$ encircled by the contour $\Gamma_n'$.
\item[(I3)] For any $n\in\Z$ there exists an analytic function $\zeta_n^{(\psi)} : \C\times U_\psi\to\C$ such that for any 
$m\in\Z$ one has
\begin{equation}\label{eq:normalization_condition}
\frac{1}{2\pi}\int_{\Gamma_m}\frac{\zeta_n^{(\psi)}(\lambda,\varphi)}{\sqrt[c]{\Delta^2(\lambda,\varphi)-4}}\,d\lambda=
\delta_{nm}.
\end{equation}
Moreover, for any $\varphi\in U_\psi$ and for any $n\in\Z$ the zeros $\big\{\sigma^n_k \,\big|\,k\in\Z\setminus\{n\}\big\}$
of the entire function $\zeta_n^{(\psi)}(\cdot,\varphi)$, when listed with their multiplicities, satisfy the conditions
(D1)-(D3) with $R_s$ replaced by $R_\psi$ and $\zeta_n^{(s)}$ replaced by $\zeta_n^{(\psi)}$.
The canonical root in \eqref{eq:normalization_condition} is defined by \eqref{eq:canonical_root} and the map
\[
\Big(\C\setminus\big(\bigsqcup_{k\in\Z}\overline{D_k'}\big)\Big)\times U_\psi\to\C, 
\quad(\lambda,\varphi)\mapsto\sqrt[c]{\Delta^2(\lambda,\varphi)-4},
\]
is analytic.
\end{itemize}

\medskip

Since the isospectral set $\iso_o(\psi^{(1)})$ is compact (see Proposition 2.2 in \cite{KTPreviato}), there exist finitely many 
elements $\eta^{(j)}\in\iso_o(\psi^{(1)})$, $1\le j\le J$, such that $\iso_o(\psi^{(1)})$ is contained in the connected open set
\begin{equation}\label{eq:U_iso}
U_{\rm iso}:=\bigcup_{1\le j\le J} U_{\eta^{(j)}}\subseteq L^2_c.
\end{equation}
Take $R:=\max_{1\le j\le J} R_{\eta^{(j)}}$.  Using the compactness of $\iso_o(\psi^{(1)})$ and the fact that
$\psi^{(1)}$ has simple periodic spectrum one sees that if necessary,  the neighborhood $U_{\rm iso}$ can be shrunk
so that $U_{\rm iso}$ and $U_{\rm iso}\cap i L^2_r$ are connected and for any $\varphi\in U_{\rm iso}$ the periodic 
eigenvalues of $L(\varphi)$ inside the disk $B_R$ are simple. Note that if 
$\varphi\in U_{\eta^{(k)}}\cap U_{\eta^{(l)}}$ for some $1\le k<l\le J$ then in view of the normalization condition 
\eqref{eq:normalization_condition} and \cite[Proposition 5.2]{KT2} one concludes (cf. Remark \ref{rem:uniqueness_differentials}) 
that for any $n\in\Z$ and for any $\lambda\in\C$ we have that 
$\zeta_n^{(\eta_l)}(\lambda,\varphi)=\zeta_n^{(\eta_k)}(\lambda,\varphi)$. 
This allows us to define for any $n\in\Z$ the analytic map 
\[
\zeta_n : \C\times U_{\rm iso}\to\C
\]
such that for any $\varphi\in U_{\rm iso}$ and for any $m\in\Z$,
\begin{equation}\label{eq:normalization_condition_U_iso}
\frac{1}{2\pi}\int_{\Gamma_m}\frac{\zeta_n(\lambda,\varphi)}{\sqrt[c]{\Delta^2(\lambda,\varphi)-4}}\,d\lambda=
\delta_{nm}.
\end{equation}
Hence the following holds.

\begin{Prop}\label{prop:preparation_U_iso}
Let $\psi^{(1)}\in i L^2_r$ and assume that $\spec_p L(\psi^{(1)})$ is simple.
Then there exist a connected open neighborhood $U_{\rm iso}$ of $\iso_o(\psi^{(1)})$ in $L^2_c$ with 
$U_{\rm iso}\cap i L^2_r$ connected, analytic maps $\zeta_n : \C\times U_{\rm iso}\to\C$, $n\in\Z$, and an integer 
$R\in\Z_{\ge 0}$ such that the conditions (I1)-(I3) above hold with $U_\psi$ replaced by $U_{\rm iso}$, $\zeta_n^{(\psi)}$ 
replaced by $\zeta_n$, and $R_\psi$ replaced by $R$. In particular, for any $\varphi\in U_{\rm iso}$ and $m,n\in\Z$ we 
have the normalization condition \eqref{eq:normalization_condition_U_iso}.
\end{Prop}

As in Section \ref{sec:actions_in_U_tn}, for any $n\in\Z$ and any $\varphi\in U_{\rm iso}$ define 
the {\em $n$-th action} 
\begin{equation}\label{eq:I_n_in_U_iso}
I_n(\varphi):=\frac{1}{\pi}\int_{\Gamma_n}
\frac{\lambda{\dot\Delta}(\lambda,\varphi)}{\sqrt[c]{\Delta^2(\lambda,\varphi)-4}}\,d\lambda.
\end{equation}
One easily sees from Proposition \ref{prop:preparation_U_iso}, Corollary \ref{coro:actions_poisson_relations}, and
Lemma \ref{lem:real-analyticity} that for any $n\in\Z$, the map $I_n : U_{\rm iso}\to\C$ is analytic and real-valued
when restricted to $U_{\rm iso}\cap i L^2_r$. 
Next let us define the angle variables. To this end introduce for any $n\in\Z$ the analytic subvariety of $U_{\rm iso}$,
\begin{equation}\label{eq:Z_n_in_U_iso}
\mathcal{Z}_n=\big\{\varphi\in U_{\rm iso}\,|\,\gamma_n^2(\varphi)=0\big\}.
\end{equation}
For $|n|>R$, the set $\mathcal{Z}_n\cap i L^2_r$ is a real analytic submanifold in $i L^2_r$ of real codimension two.
Since for any $\varphi\in U_{\rm iso}$ the periodic spectrum of $L(\varphi)$ is simple inside the disk $B_R$ we see
that $\mathcal{Z}_n=\emptyset$ for any $|n|\le R$.
Arguing as in the proof of Proposition \ref{prop:beta^n_tn} one shows that after shrinking $U_{\rm iso}$, if needed,
one can analytically extend the angle variable $\theta_n$, introduced near $\psi^{(1)}$ in 
Proposition \ref{prop:angles_tn}, to $U_{\rm iso}\setminus\mathcal{Z}_n$ so that it is of the form
\begin{equation}\label{eq:theta_n_U_iso}
\theta_n(\varphi)={\tilde\beta}^n(\varphi)+\sum_{|k|>R}\beta^n_k(\varphi),\quad
\beta^n_k(\varphi)=\int_{\lambda_k^-(\varphi)}^{\mu_k^*(\varphi)}
\frac{\zeta_n(\lambda,\varphi)}{\sqrt{\Delta^2(\lambda,\varphi)-4}}\,d\lambda\quad\forall\,|k|>R,
\end{equation}
\[
{\tilde\beta}^n(\varphi)=\sum_{|k|\le R}\int_{\PP^*[\lambda_k^-(\varphi),\mu_\varphi(\lambda_k^-(\varphi))]}
\frac{\zeta_n(\lambda,\varphi)}{\sqrt{\Delta^2(\lambda,\varphi)-4}}\,d\lambda,
\]
defined modulo $2\pi$ on $U_{\rm iso}\setminus\mathcal{Z}_n$, real valued when restricted to 
$\big(U_{\rm iso}\setminus\mathcal{Z}_n\big)\cap i L^2_r$, and analytic on $U_{\rm iso}\setminus\mathcal{Z}_n$
when considered modulo $\pi$ if $|n|>R$ and modulo $2\pi$ if $|n|\le R$. By the arguments yielding 
Theorem \ref{th:commutation_relations_tn} (cf. Corollary \ref{coro:actions_poisson_relations}) we then obtain

\begin{Th}\label{th:commutation_relations_iso} 
Let $\psi\in i L^2_r$ and assume that $\spec_p L(\psi)$ is simple and $U_{\rm iso}$ is the neighborood of
$\iso_o(\psi)$ introduced above.
Then for any $k,n\in\Z$ we have $\{I_n,I_k\}=0$ on $U_{\rm iso}$,
$\{\theta _n, I_k \} = \delta _{nk}$ on $U_{\rm iso}\backslash {\mathcal Z}_n$, and
$\{\theta _n,\theta _k\} = 0$ on $U_{\rm iso}\setminus\big({\mathcal Z}_n \cup {\mathcal Z}_k\big)$.
Furthermore, the action $I_n$ and the angle $\theta_n$ are real valued when restricted to $U_{\rm iso}\cap i L^2_r$ 
and, respectively, $\big(U_{\rm iso}\setminus\mathcal{Z}_n\big)\cap i L^2_r$.
\end{Th}

\begin{Rem}\label{rem:U_iso}
Note that the statements of Lemma \ref{lem:beta^n_k-asymptotics}, Lemma \ref{lem:beta_tilde-asymptotics}, and
Proposition \ref{Theorem 4.4} $(i)$ also hold in $U_{\rm iso}$.
\end{Rem}

\medskip

We finish this Section by proving properties of the actions used in the subsequent Section for constructing
Birkhoff coordinates in $U_{\rm iso}$.

\begin{Lem}\label{lem:negative_actions} 
For any $\varphi\in U_{\rm iso}\cap  i  L^2_r$ and for any $|n| > R$ we have that
$I_n(\varphi ) \le 0$. Moreover, $I_n(\varphi ) = 0$ iff $\lambda^+_n(\varphi) =\lambda^-_n(\varphi)$.
\end{Lem}

\begin{proof}[Proof of Lemma \ref{lem:negative_actions}] 
We follow the arguments of the proof of Theorem 13.1 in \cite{GK1}.
For any $\varphi\in U_{\rm iso}\cap i L^2_r$ and $|n| > R$ with $\lambda^+_n=\lambda^-_n$, 
one has $\lambda^+_n=\dot\lambda_n$ and hence by Cauchy's theorem applied to the integral in 
\eqref{eq:I_n_in_U_iso}, $I_n(\varphi ) = 0$.  In the case where $\lambda^+_n\ne\lambda^-_n$, let $g_n$ be 
the spectral band of Lemma \ref{lem:spectral_bands} in Appendix, which by Proposition \ref{prop:preparation_U_iso} 
and item (I1) holds with $\tilde R$ replaced by $R$. Recall that $\Delta (\lambda ^\pm _n) = (-1)^n 2$ and, for
$\lambda \in g_n \backslash \{ \lambda ^\pm _n \}$, $\Delta(\lambda)$ is real and $-2<\Delta(\lambda)<2$.  
One easily sees that
\[
\sqrt[c]{\Delta (\lambda \pm 0)^2 - 4} = 
\pm i (-1)^{n+1}\sqrt[+]{4 -\Delta (\lambda )^2}\quad\forall\lambda \in g_n .
\]
By deforming $\Gamma _n$ to $g_n$ and denoting by $g^+_n$ the arc $g_n$ with the
orientation determined by starting at $\lambda ^-_n$ one gets
\[ 
I_n = \frac{1}{\pi } \int _{g^+_n} \frac{\lambda \dot \Delta (\lambda )}
{ i (-1)^{n+1} \sqrt[+]{4 - \Delta (\lambda )^2}}\,d\lambda - 
\frac{1}{\pi }\int_{g^+_n}\frac{\lambda\dot\Delta (\lambda )}{- i (-1)^{n+1}\sqrt[+]{4 -\Delta (\lambda )^2}}\,d\lambda .
\]
Moreover, for any $\lambda \in g_n$ the quantity 
$(-1)^n\Delta(\lambda)\pm i\sqrt[+]{4 - \Delta(\lambda )^2}$ is in the domain of the principal
branch of the logarithm, $\log z=\log|z|+ i \arg z$ where $-\pi<\arg(z)<\pi$. 
Integrating by parts one then gets
\begin{align*} I_n = &- \frac{1}{\pi } \int _{g^+_n} \log \left( (-1)^n
\Delta (\lambda ) +  i  \sqrt[+]{4 - \Delta (\lambda )^2}\right) d\lambda \\
&+ \frac{1 }{\pi } \int _{g^+_n} \log \left( (-1)^n \Delta(\lambda ) -  
i\sqrt[+]{4 - \Delta (\lambda )^2}\right)\,d\lambda .
\end{align*}
Let $f_\pm(\lambda):= (-1)^n\Delta(\lambda)\pm i\sqrt[+]{4-\Delta(\lambda )^2}$ and note 
that for any $\lambda \in g_n$, $|f_+(\lambda )| = |f_-(\lambda )|=2$ whereas 
$\arg f_\pm (\lambda ) = 0$ for $\lambda\in\{\lambda^\pm _n\}$ and
$0 <\pm\arg f_\pm (\lambda)< \pi$ for any $\lambda\in g_n\setminus\{\lambda^\pm_n \}$. Hence
$I_n=-\frac{1}{\pi}\int_{g^+_n}\big(\arg f_+(\lambda )-\arg f_-(\lambda )\big)\,d\lambda < 0$.
\end{proof}   

Lemma \ref{lem:negative_actions} is applied to study the quotient $I_n / \gamma ^2_n$.
We have the following
   
\begin{Lem}\label{lem:actions_asymptotics} 
\begin{itemize}
\item[(i)] For any $|n| > R$, the quotient 
$I_n / \gamma ^2_n : U_{\rm iso}\setminus{\mathcal Z}_n\to\C$ extends analytically to $U_{\rm iso}$ so that
\begin{equation}\label{3.8} 
\frac{4 I_n}{\gamma ^2_n} = 1 + \ell ^2_n, \quad |n| > R 
\end{equation}
locally uniformly on $U_{\rm iso}$. Furthermore, $I_n / \gamma ^2_n$ is real on $U_{\rm iso}\cap i L^2_r$.
\item[(ii)] For any $|n| > R$ and $\varphi \in U_{\rm iso}\cap  i  L^2_r$, 
\[
4 I_n /\gamma ^2_n>0. 
\]
Furthermore, after shrinking $U_{\rm iso}$ if necessary, for any $|n| > R$, the real part of $4 I_n / \gamma ^2_n$ 
is bounded away from $0$ uniformly in $|n| > R$ and locally uniformly on $U_{\rm iso}$.
Moreover, the square root $\xi _n:= \sqrt[+]{4 I_n / \gamma ^2_n}$ satisfies the asymptotics 
\[
\xi _n = 1 + \ell ^2_n
\] 
locally uniformly on $U_{\rm iso}$.
\end{itemize}
\end{Lem}

\begin{proof}[Proof of Lemma \ref{lem:actions_asymptotics}] 
We follow the arguments of the proof of Theorem 13.3 in \cite{GK1}.
(i) We show that for any $|n|>R$ the quantity $I_n / \gamma ^2_n$ continuously
extends to all of $U_{\rm iso}$ and its restriction to ${\mathcal Z}_n$ is weakly analytic. 
By \cite[Theorem A.6]{GK1}, it then follows that $I_n / \gamma ^2_n$ is 
analytic on $U_{\rm iso}$. Let $\varphi \in U_{\rm iso}\setminus{\mathcal Z}_n$ for some given $|n| > R$. 
By the definition \eqref{eq:canonical_root} of the canonical root and the product representation of $\dot\Delta(\lambda )$ 
(cf. Lemma \ref{lem:product1}) one has for $\lambda $ near $\Gamma_n$,
\[
\frac{\dot\Delta (\lambda )}{\sqrt[c]{\Delta (\lambda )^2 - 4}} 
=\frac{\dot \lambda _n - \lambda }{ i  w_n(\lambda )}\,\chi_n(\lambda)
\quad\text{\rm  and}\quad
\chi_n(\lambda)=\prod_{k \ne n}\frac{\dot\lambda_k - \lambda }{w_k(\lambda )}
\]
where, by \eqref{eq:w_k}, $w_k(\lambda)=\sqrt[\rm st]{(\lambda ^+_k - \lambda )(\lambda ^-_k - \lambda )}$. 
By the definition \eqref{eq:I_n_in_U_iso} of $I_n$,
\[
I_n = \frac{ i }{\pi } \int _{\Gamma _n}
\frac{(\lambda-\dot\lambda _n)^2}{\sqrt[\rm st]{(\lambda^+_n-\lambda )(\lambda^-_n-\lambda)}}\,
\chi_n(\lambda)\,d\lambda.
\]
The assumption $\gamma _n \not= 0$, allows to make the substitution 
$\lambda= \tau _n + z \gamma _n / 2$, $\tau _n = (\lambda ^+_n + \lambda ^-_n) / 2$,
which in view of the definition \eqref{eq:standard_root} of the standard root leads to
\[ 
\sqrt[\rm st]{(\lambda ^+_n - \lambda )(\lambda ^-_n - \lambda )} \big
\arrowvert _{z - i0} =  i  \frac{\gamma _n}{2} \sqrt[+]{1 - z^2},\quad
- 1 \leq z \leq 1 .
\]
Hence, with $z_n = 2(\dot \lambda _n - \tau _n) / \gamma _n$,
\[ 
\frac{4 I_n}{\gamma^2_n} = 
\frac{2}{\pi}\int^1_{-1}\frac{(z - z_n)^2}{\sqrt[+]{1 - z ^2}}\,\chi_n(\tau_n+z\gamma_n/2)\,dz .
\]
By Lemma \ref{lem:counting_lemma} (ii), $z_n \rightarrow 0$ as $\gamma_n\to 0$. 
Thus
\[ 
\frac{4 I_n}{\gamma^2_n}\to\chi_n(\tau _n)\,
\frac{2}{\pi}\int^1_{-1}\frac{z^2}{\sqrt[+]{1-z^2}}\,dz=\chi_n(\tau_n) .
\]
This shows that $I_n / \gamma ^2_n$ is continuous on all of $U_{\rm iso}$.
Note that $\chi _n(\tau _n) \ne 0$. Moreover, by the argument principle, $\tau_n$ is analytic on 
$U_{\rm iso}$ and arguing as in the proof of \cite[Lemma 12.7]{GK1} one sees that $\chi _n$ is analytic on 
$D_n \times U_{\rm iso}$. Hence the composition $\chi _n(\tau _n)$ is analytic on $U_{\rm iso}$ and 
thus in particular weakly analytic on ${\mathcal Z}_n$. By \cite[Theorem A.6]{GK1},
$I_n / \gamma ^2_n$ extends analytically to all of $U_{\rm iso}$. Arguing as
in in the proof of \cite[Lemma 12.10]{GK1} one sees that $\chi _n(\lambda ) = 1 + \ell ^2_n$ for $\lambda$
near the interval
\[
[\lambda ^-_n, \lambda ^+_n] = 
\big\{ (1 - t) \lambda ^-_n + t \lambda ^+_n\,\big|\,0 \leq t \leq 1\big\}
\]
locally uniformly on $U_{\rm iso}$. By the asymptotics $z_n = \gamma _n \ell ^2_n$ (cf. Lemma \ref{lem:counting_lemma} $(ii)$)
it then follows that $4 I_n/\gamma^2_n = 1+\ell ^2_n$ locally uniformly on $U_{\rm iso}$.

(ii)  By Lemma \ref{lem:spectrum_symmetries} one has for any $\varphi\in U_{\rm iso}\cap i L^2_r$, $\gamma^2_n\le 0$. 
By Lemma \ref{lem:negative_actions} it then follows that for any $|n| > R$ with
$\gamma _n \ne 0$ we have $I_n/\gamma^2_n> 0$ whereas if $\gamma_n = 0$, one has by the proof 
of item (i) that $4 I_n/\gamma^2_n = \chi_n(\tau_n)\ne 0$. By the continuity of $I_n/\gamma^2_n$ on 
$U_{\rm iso}\cap i  L^2_r$ one then concludes that $I_n / \gamma ^2_n > 0$ on $U_{\rm iso}\cap i L^2_r$. 
Furthermore, by the asymptotics established in item (i), $4 I_n/\gamma ^2_n\to 1$ locally uniformly
on $U_{\rm iso}$. By shrinking $U_{\rm iso}$ if necessary we can assure that $\re(4 I_n / \gamma ^2_n)$ is
bounded away from $0$ locally uniformly on $U_{\rm iso}$ and uniformly in $|n| > R$. Then
$\xi_n = \sqrt[+]{4I_n /\gamma^2_n}$ is well defined and real analytic on $U_{\rm iso}$. It is positive on 
$U_{\rm iso}\cap  i  L^2_r$ and satisfies the asymptotics $1+ \ell ^2_n$ locally uniformly on $U_{\rm iso}$.
\end{proof}

\section{The pre-Birkhoff map and its Jacobian}\label{sec:birkhoff_map}
In this Section we construct the pre-Birkhoff map $\Phi : U_{\rm iso}\to\ell^2_c$ in an open neighborhood 
$U_{\rm iso}$ of the isospectral set $\iso_o(\psi^{(1)})$ of an arbitrary given potential $\psi^{(1)}\in i L^2_r$ with 
simple periodic spectrum, using the action and angle variables introduced in Section \ref{sec:actions_and_angles_in_U_iso}.
We then prove that the restriction $\Phi : U_{\rm iso}\cap i L^2_r\to i\ell^2_r$ is a local diffeomorphism.\footnote{For simplicity of 
notation we will use the same symbol for the map $\Phi$ and its restriction to $U_{\rm iso}\cap i L^2_r$.}
Without further reference we will use the notations and results of the previous sections.
For any $|n| > R$ and $\varphi\in U_{\rm iso}\setminus{\mathcal Z}_n$ we define
\begin{equation}\label{eq:x,y}
x_n:= \frac{\xi _n \gamma _n}{\sqrt{2}} \cos\theta _n , \quad
y_n:= \frac{\xi _n \gamma _n}{\sqrt{2}} \sin\theta _n
\end{equation}
where $\theta_n : U_{\rm iso}\setminus\mathcal{Z}_n\to\C$ is the $n$-th angle \eqref{eq:theta_n_U_iso}, 
$\gamma_n=\lambda^+_n-\lambda^-_n$, and $\xi_n : U_{\rm iso}\to\C$ is the real-analytic non vanishing function introduced in 
Section \ref{sec:actions_and_angles_in_U_iso}. Recall from Lemma \ref{lem:actions_asymptotics} that
\[
4 I_n=(\xi_n\gamma_n)^2
\]
where the $n$-th action $I_n$ is defined by \eqref{eq:I_n_in_U_iso} and $\xi_n$ satisfies
\begin{equation}\label{eq:xi_asymptotics}
\xi _n = 1 + \ell ^2_n
\end{equation}
locally uniformly in $U_{\rm iso}$.
Note that on $U_{\rm iso}\cap i L^2_r$, $\xi_n$ is real valued and $\gamma_n\in i\R$ for any $|n|>R$. Since
$\theta_n$, defined modulo $2\pi$, is real valued on $\big(U_{\rm iso}\setminus\mathcal{Z}_n\big)\cap i L^2_r$
it then follows that $x_n, y_n\in i\R$ on $\big(U_{\rm iso}\setminus\mathcal{Z}_n\big)\cap i L^2_r$, $|n|>R$.

Using the arguments from \cite[Section 16]{GK1} we now show that $(x_n,y_n)_{|n|>R}$ can be analytically 
extended to $U_{\rm iso}$ in $L^2_c$. 

First note that due to the lexicographic ordering $\lambda ^-_n \preccurlyeq\lambda ^+_n$, the $n$-th gap length $\gamma _n$ 
is not necessarily continuous on $U_{\rm iso}$. On the other hand, by Proposition \ref{Theorem 4.4} the quantity $\beta ^n_n$ is 
defined modulo $2\pi$ on $U_{\rm iso}\setminus{\mathcal Z}_n$ and real analytic when taken modulo $\pi$
whereas in view of the definition \eqref{eq:theta_n_U_iso} of the angle $\theta_n$ and Proposition \ref{prop:beta^n_tn},
the difference $\theta_n-\beta ^n_n$ is real analytic on $U_{\rm iso}$.
We will first focus our attention on the complex functions
\[ 
z^+_n:=\gamma _n e^{i \beta ^n_n}, \quad 
z^-_n:=\gamma _n e^{- i \beta ^n_n} , \quad |n| > R .
\]

\begin{Lem}\label{Lemma 5.1} 
The functions $z^\pm_n=\gamma_n e^{\pm  i \beta ^n_n}$ are analytic on 
$U_{\rm iso}\setminus{\mathcal Z}_n$ for any $|n| > R$.
\end{Lem}

\begin{proof}[Proof of Lemma \ref{Lemma 5.1}]
We follow the arguments of the proof of Lemma 16.1 in \cite{GK1}.
Fix $|n| > R$ arbitrarily. Arguing as in \cite[Proposition 7.5]{GK1} one easily sees
that locally around any potential in $U_{\rm iso}\setminus{\mathcal Z}_n$,
there exist analytic functions $\varrho ^+_n$ and $\varrho ^-_n$
such that the set equality 
$\{ \varrho ^-_n, \varrho ^+_n\} = \{ \lambda ^-_n, \lambda ^+_n\}$ holds.
Let
\[ 
\tilde \gamma _n := \varrho ^+_n - \varrho ^-_n,\quad
{\tilde\beta}^n_n := \int ^{\mu _n}_{\varrho ^-_n} 
\frac{\zeta _n (\lambda )}{\sqrt[\ast ]{\Delta (\lambda )^2 - 4}}\,d\lambda\quad(\mathop{\rm mod} 2\pi).
\]
If $\varrho^-_n = \lambda^-_n$ (and then $\varrho ^+_n = \lambda ^+_n$)
one has $\gamma _n = \tilde \gamma _n$ and $\beta ^n_n = \tilde \beta ^n_n$
whereas if $\varrho ^-_n = \lambda ^+_n$ (and hence $\varrho ^+_n = \lambda^-_n$), 
in view of the normalization condition 
\[
\int ^{\lambda ^+_n}_{\lambda^-_n} 
\frac{\zeta _n (\lambda )}{\sqrt[\ast ]{\Delta (\lambda )^2 - 4}}\,d\lambda 
\in\pi+2\pi\Z,
\] 
one has
\[ 
\gamma_n=-\tilde \gamma _n,\quad
\beta ^n_n=\int ^{\mu _n}_{\lambda^-_n}
\frac{\zeta _n(\lambda )}{\sqrt[\ast ]{\Delta(\lambda )^2 - 4}}\,d\lambda
={\tilde\beta}^n_n+\pi\quad(\mathop{\rm mod} 2\pi). 
\]
Thus in both cases $\gamma _n e^{\pm i \beta ^n_n}=\tilde \gamma _n e ^{\pm  i  \tilde \beta ^n_n }$.
As the right hand side of the latter identity is analytic the
Lemma follows.
\end{proof}

\medskip

Next we study the limiting behavior of $z^\pm _n$, $|n| > R$, as $\varphi$
approaches a potential $\psi\in U_{\rm iso}$ with the $n$-th gap collapsed. This limit
is different form zero when $\psi$ is in the set 
\[
{\mathcal F}_n:=\big\{\psi\in U_{\rm iso}\,\big|\,\mu _n \notin G_n\big\}
\] 
where $G_n=[\lambda ^-_n,\lambda ^+_n]$. 
Note that the set ${\mathcal F}_n$ is open. On ${\mathcal F}_n$,
the sign function
\begin{equation}\label{5.1} 
\varepsilon_n:=\frac{\sqrt[\ast ]{\Delta(\mu _n)^2-4}}{\sqrt[c]{\Delta (\mu _n)^2 - 4}}
\end{equation}
is well defined and locally constant.

\medskip

\begin{Lem}\label{Lemma 5.2} 
Let $|n| > R$ and $\psi \in {\mathcal Z}_n$. If $\varphi\in {\mathcal F}_n \backslash {\mathcal Z}_n$ 
tends to $\psi $, then
\[ 
\gamma _n e^{\pm  i  \beta ^n_n} \to 
\begin{cases} 
2(\tau_n-\mu _n)(1\pm\varepsilon_n) e^{\chi _n}&\text{\rm if}\quad\psi\in
{\mathcal F}_n\cap{\mathcal Z}_n \\ 0 &\text{\rm if}\quad\psi\in
{\mathcal Z}_n\setminus{\mathcal F}_n 
\end{cases}
\]
where
\[ 
\chi_n=\int ^{\mu_n}_{\tau_n} 
\frac{Z_n(\tau_n)-Z_n(\lambda )}{\tau_n-\lambda }\,d\lambda,\quad 
Z_n(\lambda)=-\prod_{m\ne n}\frac{\sigma ^n_m-\lambda}
{w_m(\lambda)},
\]
and $w_m(\lambda)=\sqrt[\rm st]{(\lambda_m^+-\lambda)(\lambda_m^--\lambda)}$.
\end{Lem}

\begin{proof}[Proof of Lemma \ref{Lemma 5.2}] 
We follow the arguments of the proof of Lemma 16.2 in \cite{GK1}.
Without loss of generality we may choose for $\varphi\in{\mathcal F}_n \backslash{\mathcal Z}_n$ a path of integration in 
the integral of $\beta ^n_n$ which meets $G_n$ only at the initial point
$\lambda ^-_n$. Using the definition \eqref{5.1} of the sign $\varepsilon_n$ and the definition \eqref{eq:canonical_root}
of $\sqrt[c]{\Delta (\lambda )^2 -4}$, we then can write
\begin{equation}\label{5.1bis} 
\frac{\zeta _n(\lambda )}{\sqrt[\ast ]{\Delta (\lambda )^2- 4}} = 
- i\varepsilon _n \frac{Z_n(\lambda )}{w_n(\lambda )}
\end{equation}
and get modulo $2\pi$,
\[ 
i \beta ^n_n =  i  \int ^{\mu _n}_{\lambda ^-_n} \frac{\zeta _n
(\lambda )}{\sqrt[\ast ]{\Delta (\lambda )^2 - 4}}\,d\lambda = 
\varepsilon_n\int ^{\mu _n}_{\lambda ^-_n}
\frac{Z_n(\lambda )}{w_n(\lambda )}\,d\lambda,
\]
where $w_n(\lambda )$ is well defined as we consider a path of integration
from $\lambda ^-_n$ to $\mu _n$ inside $D_n$ which meets $G_n$ only at its initial point. 
We decompose the numerator $Z_n(\lambda )$ into three terms,
\[ 
Z_n(\lambda ) = \big(Z_n(\lambda ) - Z_n(\tau _n)\big) +
\big(Z_n(\tau _n) + 1\big) - 1
\]
and denote the corresponding integrals by $o_n, v_n$, and $\omega _n$,
respectively. The limit of the first term is straightforward. Note that
$Z_n$ is analytic in some neighborhood 
$D_n \times V_\psi \subseteq\C \times L^2_c$ of $(n\pi , \psi )$. 
Moreover, if $\varphi\to \psi $, then 
$\lambda ^\pm _n(\varphi ) \to \tau_n(\psi )$, 
$w_n(\lambda )\to\tau_n-\lambda $, and 
$\mu_n(\varphi )\to\mu _n(\psi )$. Thus by the definition of $\chi_n(\psi)$,
\[ 
o_n = \int ^{\mu _n}_{\lambda ^-_n} \frac{Z_n(\lambda ) -
Z_n(\tau _n)}{w_n(\lambda )}\,d\lambda\to
\int ^{\mu _n}_{\tau _n} 
\frac{Z_n(\lambda ) - Z_n(\tau _n)}{\tau _n -\lambda }\,d\lambda = - \chi _n(\psi ) .
\]
For the second term we have in view of the identity \eqref{5.2} below and the
estimate $Z_n(\tau _n) + 1 = O(\gamma _n)$ by Lemma \ref{Lemma 5.3} below
\[ 
v_n = \big(Z_n(\tau _n) +1\big)\int ^{\mu _n}_{\lambda ^-_n} 
\frac{1}{w_n(\lambda )} d\lambda\to 0 .
\]
Now consider the third term. Using a limiting argument we can compute it
modulo $2\pi i$ on ${\mathcal F}_n \backslash {\mathcal Z}_n$ by
choosing the line segment $\lambda=\tau _n+t\gamma_n /2$ with $-1\le t \le \varrho _n$
as path of integration where $\varrho _n =2(\mu _n - \tau _n) / \gamma _n$. 
In case the interval $[\lambda^-_n,\mu_n]$ intersects $G_n\setminus\{\lambda_n^-\}$ it
actually contains all of $G_n$. One easily verifies that in this case the choice of the sign of $w_n(\lambda)$
along $G_n$ does not matter. We then get modulo $2\pi  i$
\begin{equation}\label{5.2} 
\omega _n = \int ^{\mu _n}_{\lambda ^-_n} 
\frac{d\lambda }{w_n(\lambda ) } = \int ^{\mu _n}_{\lambda ^-_n} 
\frac{d\lambda }{\sqrt[\rm st]{(\lambda ^+_n - \lambda )(\lambda ^-_n- \lambda )}} 
= \int ^{\varrho _n}_{-1} \frac{dt}{\sqrt[\rm st]{(1-t)(-1-t)}}
\end{equation}
with $\sqrt[\rm st]{(1-t)(-1-t)} = - t \sqrt[+]{1 - t^{-2}}$ for $|t| \to
\infty $. Note that for $\varphi \in {\mathcal F}_n \backslash {\mathcal Z}_n$ one has 
$\varrho _n \in \C \backslash [-1,1]$. We claim that
\begin{equation}\label{5.3} 
e^{-\varepsilon _n \omega _n} =- \varrho _n + 
\varepsilon_n\sqrt[\rm st]{(1 - \varrho _n)(-1-\varrho _n)} .
\end{equation}
Indeed both sides of \eqref{5.3}, viewed as functions of $\varrho _n$, are solutions of the
initial value problem
\[ 
\frac{f'(w)}{f(w)} = \frac{-\varepsilon_n}{\sqrt[\rm st]{(1 - w)(-1-w)}},\quad f(-1) = 1,
\quad w\in\C\setminus[-1,1].
\]
Now we compute the limit $\varphi \to \psi $. First consider the case
where $\psi \in {\mathcal F}_n \cap {\mathcal Z}_n$. Then $\mu _n - \tau _n$
does not converge to zero. This implies $\varrho ^{-1}_n \to 0$ and further
\begin{eqnarray}
\gamma _n e^{-\varepsilon _n \omega _n} &=&-\gamma_n\varrho_n +
\varepsilon_n\gamma_n\sqrt[\rm st]{(1-\varrho_n)(-1-\varrho_n)}\nonumber\\
&=&-\gamma_n\varrho_n-\varepsilon_n\gamma_n\varrho_n \sqrt[+]{1-\varrho_n^{-2}}\nonumber\\
&=&2(\tau _n - \mu _n)(1 + \varepsilon _n \sqrt[+]{1 -\varepsilon ^{-2}_n})\nonumber\\
&\to&2(\tau _n - \mu _n)(1 + \varepsilon _n).\label{5.4} 
\end{eqnarray}
In the case where $\psi \in {\mathcal Z}_n \backslash {\mathcal F}_n$,
one has $\gamma_n \varrho _n = 2(\mu _n - \tau _n) \to 0$ and thus concludes that
\[ 
\gamma _n e^{-\varepsilon _n \omega _n} = - \gamma_n\varrho_n +
\varepsilon _n \gamma _n \sqrt[\rm st]{(1 - \varrho _n)(-1-\varrho _n)}
\to 0 .
\]
(Actually, this case is included in the above result since it does not matter
that $\varepsilon _n$ is not well defined outside of ${\mathcal F}_n$.)
So in both cases we obtain
\[
\gamma_n e^{-\varepsilon_n\omega_n}\to 
2(\tau_n-\mu_n)(1+\varepsilon_n) .
\]
Combining the results for all three integrals we obtain
\begin{equation}\label{5.5} 
\gamma _n e^{ i \beta^n_n} = 
\gamma_n e^{-\varepsilon _n \omega _n}e^{\varepsilon_n(o_n+v_n)}\to 
2(\tau_n-\mu_n)(1+\varepsilon_n)e^{\varepsilon_n\chi_n}.
\end{equation}
This agrees with the claim for $z^+_n$ for $\varepsilon _n = -1$ where it
vanishes and for $\varepsilon_n = 1$, where $e^{\varepsilon _n \chi _n}= e^{\chi _n}$.
For $z^-_n$ we just have to switch the sign of $\varepsilon_n$ in \eqref{5.3} to obtain
\[ 
\gamma_n e^{- i\beta^n_n} = 
\gamma_n e^{\varepsilon_n\omega _n}e^{-\varepsilon_n(o_n+v_n)}\to 
2(\tau_n-\mu_n)(1-\varepsilon_n)e^{-\varepsilon_n\chi_n} .
\]
In particular the limit vanishes for $\varepsilon_n=1$.
\end{proof}

\begin{Lem}\label{Lemma 5.3} 
For $\lambda\in G_n$, $|n| > R$, one has $Z_n(\lambda)=-1+O(\gamma _n)$
locally uniformly on $U_{\rm iso}$.
\end{Lem}

\begin{proof}[Proof of Lemma \ref{Lemma 5.3}]
We follow the arguments of the proof of Lemma 16.3 in \cite{GK1}.
In analogy to \eqref{5.1bis} we write
\[ 
\frac{\zeta _n(\lambda )}{\sqrt[c]{\Delta (\lambda )^2 - 4}} =
\frac{Z_n(\lambda )}{ i  w_n(\lambda )} , \quad Z_n(\lambda )
= - \prod _{m \not= n} \frac{\sigma ^n_m - \lambda }{w_m(\lambda )}.
\]
Integrating over $\Gamma _n$ and referring to Proposition \ref{prop:preparation_U_iso}
we obtain for $\tau \in G_n$,
\[ 
1 = \frac{1}{2\pi  i } \int _{\Gamma _n}\frac{Z_n(\lambda )}{w_n(\lambda )}\,d \lambda \\
=\frac{1}{2\pi i}Z_n(\tau )\int_{\Gamma _n}\frac{1}{w_n(\lambda )}\,d\lambda +
\frac{1}{2\pi i}\int_{\Gamma_n}\frac{Z_n(\lambda ) - Z_n(\tau )}{w_n(\lambda )}\,d\lambda .
\]
Since $\frac{1}{2\pi i}\int_{\Gamma_n}\frac{1}{w_n(\lambda)}\,d\lambda=-1$ we get
\[ 
1 = -Z_n(\tau ) + \frac{1}{2\pi i}\int_{\Gamma_n}
\frac{Z_n(\lambda) - Z_n(\tau )}{w_n(\lambda )}\,d\lambda
=-Z_n(\tau)+O\left(\big\arrowvert Z_n -Z_n(\tau )\big\arrowvert _{G_n}\right)
\]
where the latter asymptotics follow from \cite[Lemma 14.3]{GK1}. Note that $Z_n(\lambda)$ is analytic on $D_n$ by
\cite[Corollary 12.8]{GK1}. Since $Z_n$ is bounded on $D_n$ locally uniformly in $\varphi $ and uniformly in $n$ by
\cite[Lemma 12.10]{GK1}, it then follows that the same is true for $\dot Z_n(\lambda ),
\lambda \in G_n$, by Cauchy's estimate. Therefore
\[ 
\max _{\lambda \in G_n} \big\arrowvert Z_n(\lambda ) - Z_n(\tau )
\big\arrowvert \leq \max _{\lambda \in G_n} |\dot Z_n(\lambda )| |\gamma _n| = O(\gamma _n)
\]
locally uniformly in $\varphi $ and uniformly in $n$. This proves the claim.
\end{proof}

\medskip

For any $|n|>R$, we now extend $z^\pm _n$ on all of $U_{\rm iso}$ as follows
\[ 
z^\pm_n = 
\begin{cases} 
2(\tau _n - \mu _n)(1\pm\varepsilon _n)e^{\chi _n}&\text{\rm on}\quad{\mathcal Z}_n \cap {\mathcal F}_n \\
0 &\text{\rm on}\quad{\mathcal Z}_n \backslash{\mathcal F}_n 
\end{cases}
\]
where $\chi_n$ is given by Lemma \ref{Lemma 5.2}. To establish the proper target space of $(x_n,y_n)_{n\ge 1}$
we need the following asymptotic estimates.

\begin{Lem}\label{Lemma 5.4} 
For $|n| > R$ one has $z^\pm _n = O\big(|\gamma_n|+|\mu_n-\tau_n|\big)$
locally uniformly on $U_{\rm iso}$.
\end{Lem}

\begin{proof}[Proof of Lemma \ref{Lemma 5.4}]
We follow the arguments of the proof of Lemma 16.4 in \cite{GK1}.
From the proof of Lemma~\ref{Lemma 5.2}, in particular equations
\eqref{5.4} and \eqref{5.5}, one sees that for $\varphi $ in 
${\mathcal F}_n \backslash {\mathcal Z}_n$,
\[ 
\gamma_n e^{ i \beta^n_n} = 
\left(-\gamma_n\varrho _n+
\varepsilon_n\gamma _n\sqrt[\rm st]{(1-\varrho_n)(-1-\varrho_n)}\right)
e^{\varepsilon _n(v_n + o_n)} .
\]
In the case $2|\mu _n - \tau _n| \leq |\gamma _n|$, i.e., $|\varrho _n| \leq 1$, one has
\begin{equation}\label{5.6} 
\big\arrowvert - \gamma _n \varrho _n +\varepsilon _n \gamma _n
               \sqrt[\rm st]{(1 - \varrho _n)(- 1 - \varrho _n)}\big\arrowvert \leq
               3|\gamma _n|
\end{equation}
while in the case $2|\mu _n - \tau _n| > |\gamma _n|$, i.e., $|\varrho _n| > 1$,
\[ 
\gamma_n e^{i \beta^n_n} = 2(\tau_n - \mu_n)
\left(1+\varepsilon_n\sqrt[+]{1+\varrho_n^{-2}}\right) 
e^{\varepsilon_n(v_n + o_n)}
\]
yielding the estimate
\begin{equation}\label{5.7} 
\left|2(\tau _n - \mu _n)
\left(1+\varepsilon_n\sqrt[+]{1-\varrho_ n^{-2}}\right)\right|
\le 6 |\mu _n-\tau _n|.
\end{equation}
The exponential term $e^{\varepsilon _n(v_n + o_n)}$ is locally uniformly
bounded in view of \cite[Lemma 12.10]{GK1}. So we get on 
${\mathcal F}_n\setminus{\mathcal Z}_n$,
\begin{equation}\label{5.8} 
z^+_n = O\big( |\gamma _n| + |\mu _n - \tau _n|\big) .
\end{equation}
By Lemma~\ref{Lemma 5.2}, \eqref{5.7} continues to hold on 
${\mathcal F}_n \cap {\mathcal Z}_n$. Furthermore, one easily verifies that \eqref{5.6}
is also valid on $U_{\rm iso}\setminus{\mathcal F}_n$ for any choice of
$\varepsilon _n \in \{ \pm 1\}$. Hence \eqref{5.8} holds in a locally
uniform fashion on all of $U_{\rm iso}$. The argument for $\gamma_n e^{- i\beta_n^n}$
is the same.
\end{proof}

\medskip

We are now ready to prove

\begin{Prop}\label{Theorem 5.5} 
For any $|n|>R$, the functions $z^\pm _n$ are
analytic on $U_{\rm iso}$.
\end{Prop}

\begin{proof}[Proof of Proposition \ref{Theorem 5.5}] 
We follow the arguments of the proof of Proposition 16.5 in \cite{GK1}.
First, we apply \cite[Theorem A.6]{GK1} to the functions $z^\pm _n$
on the domain $U_{\rm iso}$ with the subvariety ${\mathcal Z}_n$. These
functions are analytic on $U_{\rm iso}\setminus{\mathcal Z}_n$ by
Lemma~\ref{Lemma 5.1}. We claim that they are also continuous at points of 
$\mathcal{Z}_n$. First note that their restrictions to ${\mathcal Z}_n$ are continuous
by inspection. Approaching a point in ${\mathcal Z}_n$ from within
${\mathcal F}_n \backslash {\mathcal Z}_n$, the corresponding values
$z^\pm _n$ converge to the ones of the limiting potential
by Lemma \ref{Lemma 5.2}. On the other hand, approaching a point in
${\mathcal Z}_n$ from outside of ${\mathcal F}_n\cup{\mathcal Z}_n$
one has $|\mu _n - \tau _n| \leq |\gamma _n|$ and thus 
$z^\pm_n =\gamma _n e^{\pm  i \beta^n_n}$ converges to zero by Lemma \ref{Lemma 5.4}.
Thus the functions $z^\pm _n$ are continuous on all of $U_{\rm iso}$. To show
that their restrictions to ${\mathcal Z}_n$ are weakly analytic, let
$D$ be a one-dimensional complex disk contained in ${\mathcal Z}_n$.
If the center of $D$ is in ${\mathcal F}_n$, then the entire disk $D$
is in ${\mathcal F}_n$, if chosen sufficiently small. The analyticity
of $z^\pm _n = \gamma _n e ^{\pm  i \beta^n_n}$ on $D$ is then evident from
the above formula, the definition of $\chi _n$, and the local
constancy of $\varepsilon _n$ on ${\mathcal F}_n$. If the center of
$D$ does not belong to ${\mathcal F}_n$, then consider the analytic
function $\mu _n - \tau _n$ on $D$. This function either vanishes
identically on $D$, in which case $z^\pm _n$ vanishes identically,
too. Or it vanishes in only finitely many points. Outside these
points, $D$ is in ${\mathcal F}_n$, hence $z^\pm _n$ is analytic
there. By continuity and analytic continuation, these functions are
analytic on all of $D$.
\end{proof}

\medskip

For $\varphi\in U_{\rm iso}$ we now define 
$\Phi (\varphi ) = \big(x_n(\varphi ), y_n(\varphi )\big)_{n \in \Z}$ 
where for $|n| > R$,
\begin{equation}\label{eq:x,y'}
\left\{
\begin{array}{l}
x_n=\frac{\xi _n}{\sqrt{8}}\left(z_n^+e^{- i(\theta _n-\beta_n^n)}+z_n^-e^{- i(\theta _n-\beta_n^n)}\right),\\
y_n =\frac{\xi _n}{\sqrt{8} i}\left(z_n^+e^{- i(\theta _n-\beta_n^n)}-z_n^-e^{- i(\theta _n-\beta_n^n)}\right) ,
\end{array}
\right.
\end{equation}
and for $|n|\le R$,
\[ 
x_n = \sqrt{2 I_n}\cos\theta _n, \quad y_n = \sqrt{2 I_n}\sin \theta _n.
\]
Here the roots $\sqrt{2 I_n}$ are defined as follows: First note that by adding constants to the actions and by shrinking 
the neighborhood $U_{\rm iso}$, if necessary, we can ensure that $\re(I_n)\le-1$ on $U_{\rm iso}$ for any $|n| \le R$. 
Then, $\sqrt{2 I_n}$ is analytic on $U_{\rm iso}$ where $\sqrt{\cdot}$ is the branch of the square root satisfying 
\[
\sqrt{2 I_n}=i\sqrt[+]{-2 I_n}
\] 
on $U_{\rm iso}\cap  i  L^2_r$. From the preceeding asymptotic estimates it is then evident that $\Phi $ defines a continuous, 
locally bounded map into $\ell^2_c= \ell^2\times\ell ^2$ which, when restricted to $U_{\rm iso}\cap  i  L^2_r$, takes values 
in $ i\ell ^2_r$. Moreover, each component is real analytic.
In what follows we identify the real Hilbert space $i\ell^2_r$ with $\ell^2(\Z,i\R^2)$ by the $\R$-linear isomorphism
\begin{equation}\label{eq:identification}
\ell^2(\Z,i\R^2)\to i\ell^2_r,\quad\big(i(u_n,v_n) \big)_{n\in\Z}\mapsto\big(z_n(\varphi),w_n(\varphi) \big)_{n\in\Z}
\end{equation}
where $z_n=(v_n+i u_n)/\sqrt{2}$ and $w_n=-\overline{z_{(-n)}}$ for any $n\in\Z$.
We also write $x_n$ for $i u_n$ and $y_n$ for $i v_n$.
We have proved the following

\begin{Prop}\label{Theorem 5.6} 
The map
\[ 
\Phi : U_{\rm iso}\cap  i  L^2_r \to  i  \ell ^2_r , \quad
\varphi\mapsto\big(x_n(\varphi),y_n(\varphi) \big)_{n\in\Z},
\]
is real analytic and extends to an analytic map $U_{\rm iso}\to\ell ^2_c$.
\end{Prop}

\medskip

Using Theorem \ref{th:commutation_relations_iso} we now can establish the following

\begin{Prop}\label{Theorem 5.7} 
For any $m, n \in \Z$ and $\varphi\in U_{\rm iso} \cap  i  L^2_r$,
\[
\{ x_m, x_n \} = 0, \quad \{ x_m, y_n \} = - \delta _{mn} , \quad
\{ y_m, y_n \} = 0 .
\]
\end{Prop}

\begin{proof}[Proof of Proposition \ref{Theorem 5.7}] 
Let $n \in \Z$. For any $|n| > R$, the set ${\mathcal Z}_n \cap  i  L^2_r$ is a real analytic 
submanifold of $U_{\rm iso}\cap i  L^2_r$ of codimension two whereas 
${\mathcal Z}_n \cap  i  L^2_r =\emptyset $ for $|n| \leq R$ (cf. \eqref{eq:Z_n_in_U_iso}). 
Since the coordinates $x_n$ and $y_n$ are smooth it suffices to check that for any $m \in \Z$ the claimed 
commutation relations hold on the subset 
$\big(i L^2_r \cap U_{\rm iso}\big)\setminus\big({\mathcal Z}_n\cup{\mathcal Z}_m\big)$.
Then for any $m\in\Z$, on $\big(U_{\rm iso}\setminus\mathcal{Z}_m\big)\cap i L^2_r$
\[
x_m=i\sqrt[+]{-2 I_m}\cos\theta_m,\quad y_m=i\sqrt[+]{-2 I_m}\sin\theta_m.
\]
Arguing as in the proof of \cite[Theorem 18.8]{GK1} one has by the chain rule
\begin{align*} 
\{ x_m, y_n\} &= \partial _{\theta _m} x_m \partial _{\theta
                     _n} y_n \{ \theta _m, \theta _n\} + \partial _{\theta
                     _m} x_m \partial _{I_n} y_n \{ \theta _m, I_n \} \\
                  &+ \partial _{I_m} x_m \partial _{\theta _n} y_n \{ I_m,
                     \theta _n\} + \partial _{I_m} x_m \partial _{I_n}
                     y_n \{ I_m, I_n \} .
\end{align*}
By Theorem \ref{th:commutation_relations_iso} it then follows that
\begin{eqnarray*}
\{x_m,y_n\}&=&\left(\partial _{\theta _m} x_m \partial _{I_n} y_n - 
\partial _{I_m} x_m \partial _{\theta _n} y_n\right)\delta_{mn}\\
&=&\left(-\sqrt{2 I_m}\sin\theta_m\frac{1}{\sqrt{2 I_n}}\sin\theta_n-
\sqrt{2 I_m}\cos\theta_m\frac{1}{\sqrt{2 I_n}}\cos\theta_n\right)\delta_{mn}\\
&=&-\delta_{mn}.
\end{eqnarray*}
All other commutation relations between coordinates are verified in a similar fashion.
\end{proof}

\medskip\medskip

In the remaining part of this Section we prove that for any $\varphi \in U_{\rm iso}\cap  i  L^2_r$, 
the Jacobian $d_\varphi \Phi $ of $\Phi $ is a Fredholm operator of index $0$. 
First we need to introduce some more notation and establish some auxiliary results. 
Recall that ${\mathcal Z}_n\cap i  L^2_r$, $|n| > R$, is a real analytic submanifold of 
real codimension two (cf. \eqref{eq:Z_n_in_U_iso}). Note that for 
$\varphi\in{\mathcal Z}_n\cap i  L^2_r$ with $|n|> R$, 
the double periodic eigenvalue $\lambda ^+_n(\varphi ) = \lambda ^-_n(\varphi )$ has geometric 
multiplicity two and  hence is also a Dirichlet eigenvalue. It turns out to be convenient to introduce for any $|n|> R$,
\[
{\mathcal M}_n:=\big\{\varphi \in U_{\rm iso}\,\big|\,\mu_n(\varphi)\in\{ \lambda ^\pm _n(\varphi)\}\big\} .
\]
Note that ${\mathcal Z}_n\cap i L^2_r\subseteq{\mathcal M}_n\cap i L^2_r$.

\begin{Lem}\label{Lemma 6.0} 
For any $|n| > R$, $\big(U_{\rm iso}\setminus{\mathcal M}_n\big)\cap i L^2_r$ is open and dense in 
$U_{\rm iso}\cap i  L^2_r$.
\end{Lem}

\begin{proof}[Proof of Lemma \ref{Lemma 6.0}] 
Note that $ {\mathcal M}_n=\big\{\varphi \in U_{\rm iso}\,\big|\,\chi_p(\mu_n,\varphi )=0\big\}$
where we recall that $\chi _p(\lambda , \varphi ) = \Delta (\lambda ,
\varphi )^2 - 4$. Using that $\chi _p : \C \times L^2_c \to\C$, 
$(\lambda,\varphi )\mapsto\chi_p(\lambda,\varphi )$, and
for any $|n| > R$, $\mu _n : U_{\rm iso}\to\C$, $\varphi\mapsto\mu _n(\varphi )$, 
are analytic it follows that the composition $F_n : U_{\rm iso}\to\C$,
$\varphi\mapsto\chi_p\big(\mu_n(\varphi),\varphi\big)$ is analytic. The claimed result follows if 
one can show that $F_n$ does not vanish identically on $U_{\rm iso} \cap  i  L^2_r$. 
Assume on the contrary that $F_n$ vanishes
identically on $U_{\rm iso} \cap  i  L^2_r$. By analyticity it then vanishes
on all of $U_{\rm iso}$. Actually, the $n$-th Dirichlet eigenvalue $\mu _n$ is well defined and analytic
on $U_{\rm iso}\cup U_{\rm tn}$ where $U_{\rm tn}$ is the tubular neighborhood introduced in
Section \ref{sec:actions_in_U_tn}. Thus $F_n$ defines an analytic function on $U_{\rm iso}\cup U_{\rm tn}$ 
which vanishes also on $U_{\rm tn}$. Note that $U_{\rm tn}$ contains potentials $\psi$ in $U_{\rm tn}\cap L^2_r$ near zero
for which $\lambda ^-_n < \lambda ^+_n$ and 
$\big\{\varphi\in L^2_r\,\big|\,\spec_p\big(L(\varphi)\big)=\spec_p\big(L(\psi)\big)\big\}\subseteq U_{\rm tn}\cap L^2_r$. 
The vanishing of $F_n$ on $U_{\rm tn}\cap L^2_r$ then contradicts \cite[Corollary 9.4]{GK1}.
\end{proof}

\medskip

In a first step we establish asymptotics for the gradients of $z^\pm_n$ as
$|n| \to \infty $. Since finite gap potentials are dense in $U_{\rm iso}\cap  i  L^2_r$ (\cite[Corollary 1.1]{KST}), 
it turns out that for our purposes it is sufficient to establish them for such potentials. 
Recall that for any $n \in \Z$ we have that $e^+_n = (0, e^{-2\pi  i  n x})$ and $e^-_n = (e^{2\pi i  n x},0)$.

\begin{Lem}\label{Lemma 6.1} 
At any finite gap potential in $U_{\rm iso}\cap i  L^2_r$, one has
\[ 
\partial z^\pm _n = - 2e^\pm _n + \ell ^2_n .
\]
These estimates hold uniformly on $0\le x\le 1$ in the sense that 
$\|\partial z^\pm _n+2e^\pm _n\|_{L^\infty}=\ell^2_n$.
\end{Lem}

\begin{proof}[Proof of Lemma \ref{Lemma 6.1}]
We follow the proof of \cite[Lemma 17.1]{GK1}. In view of
Lemma~\ref{Lemma 6.0} we may approximate a given potential $\psi \in
{\mathcal Z}_n \cap  i  L^2_r, |n| > R$, by potentials in $ i  L^2_r \cap
(U_{\rm iso}\setminus{\mathcal M}_n)$. Then it can be approximated
by potentials $\varphi $ in $ i  L^2_r \cap U_{\rm iso}$ satisfying either
$|\mu _n - \tau _n| \leq |\gamma _n| / 2, \mu _n \not= \lambda ^+_n,
\lambda ^-_n$ (case 1) or $|\mu _n - \tau _n| > |\gamma _n| / 2$ (case 2). 
Both cases are treated in a similar fashion so we concentrate on case 1 only.
As in \eqref{5.1} we define a sign $\varepsilon _n$ by
\[ 
\varepsilon_n = 
\frac{\sqrt[\ast ]{\Delta ^2(\lambda ) - 4}}{\sqrt[c]{\Delta ^2(\lambda ) - 4}} 
\Bigg\arrowvert _{\lambda=\mu _n}
\]
where the root $\sqrt[c]{\Delta ^2(\lambda ) - 4} = 2  i  \prod _{m \in
\Z} \frac{w_m(\lambda )}{\pi _m}$ is extended to $G_n$ by
extending $w_n(\lambda )$ from the left of $G_n$. Since by definition
$\zeta_n =- \frac{2}{\pi _n} \prod _{m \not= n}\frac{\sigma ^n_m - \lambda}{\pi_m}$
one has
\[ 
\frac{\zeta _n(\lambda )}{\sqrt[\ast ]{\Delta ^2(\lambda ) - 4}}\Bigg\arrowvert _{\lambda=\mu _n} 
= -  i  \varepsilon _n \frac{Z_n(\lambda )}{w_n(\lambda )}\Bigg\arrowvert _{\lambda=\mu _n}\quad 
\text{\rm with}\quad
Z_n(\lambda ) = - \prod _{m \not= n} \frac{\sigma ^n_m - \lambda }{w_m(\lambda )} .
\]
Hence
\[ 
i\beta ^n_n = \varepsilon _n \int ^{\mu _n}_{\lambda ^-_n} \frac{Z_n(\lambda )}{w_n(\lambda )}\,d\lambda
\]
and we are in the same situation as in the proof of Lemma~\ref{Lemma 5.2}.
With the notation introduced there we conclude again that
$z_n^+=\gamma _n e^{ i  \beta _n^n} = \gamma _n e^{-\varepsilon _n \omega _n} 
e^{\varepsilon _n(o_n + v_n)}$
and
\[ 
\gamma _n e^{-\varepsilon _n \omega _n} = -\gamma _n \varrho _n +
\varepsilon _n \gamma _n \sqrt[\rm st]{(1 - \varrho _n)(-1 - \varrho _n)}
\]
where $\varrho _n = 2(\mu _n - \tau _n) / \gamma _n$. As by assumption,
$2|\mu _n - \tau _n| \leq |\gamma _n|$, $\mu _n \notin \{ \lambda ^\pm _n\}$ and 
$\gamma _n \in  i  {\mathbb R}_{> 0}$ one has $|\varrho_n| \le 1$, $\varrho _n \not= \pm 1$.
By extending the standard root (cf. \eqref{eq:standard_root}) by continuity from above to the interval $[-1,1]$,
one obtains that
\[
\sqrt[\rm st]{(1-\varrho_n)(-1-\varrho _n)} = 
\begin{cases} 
-i\sqrt[+]{1 - \varrho ^2_n} &\text{\rm if}\quad\im(\varrho_n)\ge 0\\
+i\sqrt[+]{1 - \varrho ^2_n} &\text{\rm if}\quad\im(\varrho_n) < 0.
\end{cases}
\]
As $-i\gamma_n >0$ it then follows that $\gamma_n e^{-\varepsilon_n\omega_n}=2(\tau_n-\mu_n)+2\varepsilon_n r_n$ 
where $r_n = w_n(\mu_n)$. As both $\gamma_n e^{-\varepsilon_n\omega _n}$ and $e^{\varepsilon _n(o_n + v_n)}$ are 
analytic and as $o_n + v_n$ vanishes for $\varphi \to \psi $ (see proof of Lemma~\ref{Lemma 5.2}) we get by the product rule
\[
\partial z^+_n=\partial\big(\gamma _n e^{ i  \beta ^n_n}\big)\to 
\partial\big(\gamma _n e^{-\varepsilon _n \omega _n}\big) = 
2(\partial \tau _n - \partial \mu _n) +2\varepsilon _n\partial r_n.
\]
To study the gradient of $r_n$ we use the representation \eqref{5.1bis} at $\mu _n$ to write
\[ 
i\varepsilon_n r_n = 
\frac{Z_n(\mu _n)}{\zeta _n(\mu _n)} \sqrt[\ast ]{\Delta ^2(\mu _n) - 4} = 
\phi _n\cdot\delta(\mu_n)
\]
where $\phi _n := \frac{Z_n(\mu _n)}{\zeta_n(\mu _n)}$, $\delta(\mu_n):=\grave{m}_2(\mu_n)+\grave{m}_3(\mu_n)$
and, by the definition of the $*$-root,
\[ 
\sqrt[\ast ]{\Delta (\mu_n)^2 - 4} = \delta (\mu_n).
\]
By \cite[Lemma 4.4]{GK1} and the asymptotics of \cite[Lemma 5.3]{GK1}
\[ 
i\partial\delta(\mu _n)
=(-1)^n\big(e^+_n - e^-_n\big) + \ell ^2_n 
\]
where these estimates hold uniformly for $0\le x\le 1$.
In addition, by \cite[Theorem 2.4, Lemma 5.3]{GK1}, $\delta(\mu_n)=\ell^2_n$ and
$\dot \delta (\mu _n) = \ell ^2_n$.
It then follows from \cite[Lemma 7.6]{GK1} that
$\dot \delta (\mu _n) \partial \mu _n = \ell ^2_n .$
Moreover, by Lemma \ref{Lemma 5.3}, $Z_n(\mu _n) = -1 + O(\gamma _n)$
and by \cite[Lemma C.4]{GK1}, $\zeta _n(\mu _n) = - 2(-1)^n + \ell ^2_n$ and thus 
$2\phi _n = (-1)^n + \ell ^2_n$ and with Cauchy's estimate $\partial\phi_n=O(1)$, 
leading to $\delta(\mu_n)\partial\phi_n=\ell ^2_n$. On the other hand
\begin{align*} 
2\phi _n \partial (\delta (\mu _n)) &= 2\phi _n\partial\delta(\mu _n)
                     + 2\phi _n \dot \delta (\mu _n) \partial \mu _n \\
&=\left( (-1)^n + \ell ^2_n\right) \partial \delta(\mu _n) + \ell ^2_n =  
-i\big(e^+_n - e^-_n\big) + \ell ^2_n .
\end{align*}
Hence
\[
2 \varepsilon _n \partial r_n =  
-i\partial \big(2\phi_n\delta (\mu _n)\big) = e^-_n - e^+_n + \ell ^2_n .
\]
Finally, by \cite[Lemma 7.6]{GK1} and by \cite[Lemma 12.3]{GK1} and its proof as well as Lemma 12.4 in \cite{GK1}
\[
2\big(\partial \tau _n - \partial \mu _n\big) = - \big(e^+_n + e^-_n\big) + \ell ^2_n .
\]
Altogether we thus have proved that for finite gap potentials
\[ 
\partial z^+_n = 2\big(\partial \tau _n - \partial \mu _n\big) + 2 \varepsilon_n \partial r_n = 
- 2 e^+_n + \ell ^2_n\,.
\]
Analogously, one has
\[ 
\partial z^-_n = 2\big(\partial \tau _n - \partial \mu _n\big) - 2 \varepsilon_n \partial r_n = 
- 2 e^-_n + \ell ^2_n.
\]
\end{proof}

\begin{Lem}\label{Proposition 6.2} 
At any finite gap potential in $U_{\rm iso}\cap i L^2_r$,
\[ 
\partial x_n = - \frac{1}{\sqrt{2}}\big(e^+_n + e^-_n\big) + \ell ^2_n , \quad
\partial y_n = - \frac{1}{\sqrt{2} i  }\big(e^+_n - e^-_n\big) + \ell ^2_n .
\]
\end{Lem}

\begin{proof}[Proof of Lemma \ref{Proposition 6.2}]
We follow the arguments in the proof of Theorem 17.2 in \cite{GK1}.
By the definition of $x_n$ and $y_n$, for any $|n| > R$,
\begin{align}
\label{6.1}&x_n = \frac{\xi _n}{2 \sqrt{2}} 
\left(z^+_n e^{ i(\theta_n -  \beta^n_n)} + 
z^-_n e^{- i (\theta_n- \beta^n_n)} \right)\\
\label{6.2}&y_n = \frac{\xi _n}{2 \sqrt{2} i  } 
\left( z^+_n e^{ i (\theta _n - \beta ^n_n)} - 
z^-_n e^{- i  (\theta _n -\beta ^n_n)} \right).
\end{align}
At a finite gap potential we have $z^\pm _n = 0$ for $|n|$ sufficiently large, 
$\xi_n = 1 + \ell ^2_n$ by \eqref{eq:xi_asymptotics} and 
$\theta _n - \beta ^n_n = O \left( \frac{1}{n} \right)$ by 
Remark \ref{rem:U_iso}, using that $|\gamma_m|+|\mu_m-\tau_m|=0$ for all but finitely many $m$. 
Furthermore by Lemma \ref{lem:beta^n_k-asymptotics} and the asymptotics  of $\lambda ^\pm _n$ and $\mu _n$, the 
$(z^\pm _n)_{n\in \Z}$ are locally bounded in $\ell ^2$. By applying the product
rule and Cauchy's estimate we thus obtain from the formulas \eqref{6.1} and \eqref{6.2} that
\[ 
\partial x_n = \frac{1}{2\sqrt{2}}\big(\partial z^+_n +\partial z^-_n\big) + \ell ^2_n ,\quad
\partial y_n = \frac{1}{2\sqrt{2} i}\big(\partial z^+_n -\partial z^-_n\big) + \ell ^2_n .
\]
From Lemma \ref{Lemma 6.1} we then get the claimed asymptotics.
\end{proof}

\medskip

Now consider the Jacobian of $\Phi $. At any $\varphi\in U_{\rm iso}\cap  i  L^2_r$
it is the linear map given by
\[ 
d_\varphi\Phi :  i  L^2_r \to  i  \ell ^2_r ,\quad 
h\mapsto\left(\big(\langle b^+_n,h\rangle_r\big)_{n\in\Z},\big(\langle b^-_n,h\rangle_r\big)_{n\in\Z}\right)
\]
where 
\[
b^+_n := \partial x_n,\quad b^-_n := \partial y_n
\] 
and $\langle \cdot ,\cdot \rangle _r$ denotes the bilinear form
\[ 
L^2_c \times L^2_c \to \C , \quad 
(h,g) \mapsto \int ^1_0 (h_1 g_1 + h_2 g_2) dx .
\]
For any $n\in\Z$, introduce
\[ 
d^+_n := - \frac{1}{\sqrt{2}}\big(e^+_n + e^-_n\big),\quad 
d^-_n := - \frac{1}{\sqrt{2} i}\big(e^+_n - e^-_n\big) .
\]
As $d_n^+$, $d_n^-$, $n\in\Z$, represent a Fourier basis of $ i  L^2_r$, the linear map
\[ 
F :  i  L^2_r \to  i  \ell ^2_r, \quad
h\mapsto\left( \langle d^+_n,h \rangle _r, \langle d^-_n, h \rangle _r \right) _{n \in \Z}
\]
is a linear isomorphism. To prove the same for the Jacobian $d_\varphi\Phi $ it therefore 
suffices to show that
\[ 
B_\varphi := F^{-1} d_\varphi \Phi :  i  L^2_r \to  i L^2_r
\]
is a linear isomorphism for any $\varphi$ in $U_{\rm iso}\cap  i  L^2_r$.
Clearly, $B_\varphi $ is continuous in $\varphi $ by the analyticity of
$\Phi $ and is given by
\[
B_\varphi h = 
\sum _{n \in \Z} \langle b^+_{-n}, h \rangle_r d^+_n + 
\langle b^-_{-n}, h\rangle_r d^-_n .
\]
Its adjoint $A_\varphi $ with respect to $\langle\cdot,\cdot\rangle_r$
is then a bounded linear operator on $ i  L^2_r$ which also depends continuously
on $\varphi$ and is given by
\begin{equation}\label{eq:A}
A_\varphi h = 
\sum _{n \in \Z} \langle d^+_{-n}, h \rangle _r b^+_n + 
\langle d^-_{-n}, h \rangle _r b^-_n .
\end{equation}
Moreover, $B_\varphi $ is a linear isomorphism if and only if $A_\varphi$ is. 
We obtain the following

\begin{Lem}\label{prop:jacobian} 
For any $\varphi \in U_{\rm iso}\cap  i  L^2_r$, the differential $d_\varphi \Phi $ is 
a linear isomorphism if and only if the operator $A_\varphi $ is a linear isomorphism. 
The latter is a compact perturbation of the identity on $ i  L^2_r$.
\end{Lem}

\begin{proof}[Proof of Lemma \ref{prop:jacobian}]
We follow the arguments of the proof of Lemma 17.3 in \cite{GK1}.
It remains to prove the compactness claim. At any finite
gap potential in $U_{\rm iso}\cap i  L^2_r$ we have by
Lemma \ref{Proposition 6.2}
\[ 
\sum _{n \in \Z} \| (A_\varphi - Id)d^\pm _n \| ^2 =
\sum _{n \in \Z} \| b^\pm _n - d^\pm _n \| ^2 < \infty .
\]
Thus $A_\varphi - Id$ is Hilbert-Schmidt and therefore compact. Since $A_\varphi $ depends 
continuously on $\varphi$ and finite gap potentials are dense in $U_{\rm iso} \cap  i  L^2_r$ we
see that $A_\varphi-Id$ is compact for any $\varphi \in U_{\rm iso} \cap  i  L^2_r$.
\end{proof}

We summarize the main results of this Section as follows.

\begin{Th}\label{th:local_diffeo}
The map $\Phi : U_{\rm iso}\cap i L^2_r\to i\ell^2_r$, 
$\varphi\mapsto\big(x_n(\varphi),y_n(\varphi) \big)_{n\in\Z}$,
is real analytic and for any $m,n\in\Z$ and $\varphi\in U_{\rm iso}\cap i L^2_r$,
\[
\{x_m,x_n\}=0,\quad\{x_m,y_n\}=-\delta_{mn},\quad\{y_m,y_n\}=0\,.
\]
Furthermore, for any $\varphi\in U_{\rm iso}\cap i L^2_r$ we have that 
$d_\varphi\Phi : i L^2_r\to i\ell^2_r$ is a linear isomorphism.
\end{Th}

\begin{proof}[Proof of Theorem \ref{th:local_diffeo}]
In view of Proposition \ref{Theorem 5.6} and Proposition \ref{Theorem 5.7} it
remains to prove that for any $\varphi\in U_{\rm iso}\cap i L^2_r$
the differential $d_\varphi\Phi : i L^2_r\to i\ell^2_r$ is a linear isomorphism.
By Lemma \ref{prop:jacobian} it suffices to show that $A_\varphi : i L^2_r\to i L^2_r$
is such a map. This follows from Proposition \ref{Theorem 5.7}, Lemma \ref{prop:jacobian}, and 
Lemma F.7 in \cite{GK1}. 
\end{proof}

\section{Proof of Theorem \ref{th:main}}\label{sec:proofs}
In this Section we prove Theorem \ref{th:main}, stated in the Introduction.
Recall that the map
\begin{equation}\label{eq:birkhoff_map}
\Phi : U_{\rm iso}\cap i L^2_r\to i\ell^2_r
\end{equation}
constructed in Theorem \ref{th:local_diffeo} is a canonical local diffeomorphism.
The map $\Psi$ of Theorem \ref{th:main} will be obtained from $\Phi$ by a slight
adjustment to ensure that $\Psi$ is one-to-one. We begin with some preliminary considerations.
Let us first study isospectral sets of the ZS operator with potentials in $i L^2_r$.
By Proposition 2.2 in \cite{KTPreviato}, for any $\psi\in i L^2_r$, the isospecral set $\iso(\psi)$ is compact. 
Furthermore, recall that any connected component in a topological space is closed. If the space is locally 
path connected, then its connected components are also open (see e.g. \cite{Munkres}).
We have

\begin{Lem}\label{lem:iso_components}
For any $\psi\in i L^2_r$ with simple periodic spectrum the isospectral set $\iso(\psi)$ is locally path connected. 
In particular, the connected components of $\iso(\psi)$ are open and closed.
Since $\iso(\psi)$ is compact we conclude that $\iso(\psi)$ has finitely many connected components.
\end{Lem}

\begin{proof}[Proof of Lemma \ref{lem:iso_components}]
It is enough to proof that $\iso(\psi)$ is locally path connected.
First note that we can assume without loss of generality that $\psi\in\iso_o(\psi^{(1)}) $ where 
$\psi^{(1)}\in U_{\rm iso}\cap i L^2_r$ is the potential with simple periodic spectrum used in the construction of 
the map \eqref{eq:birkhoff_map} in Section \ref{sec:birkhoff_map}.
Since $\Phi$ is a local diffeomorphism we can find an open neighborhood 
$U_\psi$ of $\psi$ in $U_{\rm iso}\cap i L^2_r$ and an open neighborhood $V_{p^0}$ of 
$p^0:=\Phi(\psi)$ in $i\ell^2_r$ 
such that
\[
\Phi : U_\psi\to V_{p^0}
\]
is a diffeomorphism. Here $p^0=(p_n^0)_{n\in\Z}$ with $p_n^0=i(u_n^0,v_n^0)$ and
the neighborhood $V_{p^0}$ of $p^0$ in $i\ell^2_r$ is chosen of the form
\begin{equation}\label{eq:tailneighborhood}
V_{p^0}=B^{\delta_0',\delta_0''}_{|n|\le R'}(p^0)\times B^{\epsilon_0}_{|n|>R'}(0)
\end{equation}
where $R',\delta_0',\delta_0'',\epsilon_0>0$ are appropriately chosen parameters,
\begin{equation}\label{eq:finite_component}
B^{\delta_0',\delta_0''}_{|n|\le R'}(p^0):=
\bigtimes_{|n|\le R'}\Big\{p_n=i (u_n,v_n)\in i\,\R^2\,\Big|\,\big||p_n^0|-|p_n|\big|<\delta_0',
|\theta_n-\theta_n^0|<\delta_0''\Big\},
\end{equation}
\begin{equation}\label{eq:tail_component}
B^{\epsilon_0}_{|n|>R'}(0):=
\Big\{p_n=i (u_n,v_n)_{|n|>R'}\in i\ell^2_r\,\Big|\,\Big(\sum_{|n|>R'}|p_n|^2\Big)^{1/2}<\epsilon_0\Big\},
\end{equation}
and $|p_n|=\sqrt{u_n^2+v_n^2}$ and $\theta_n$ are the polar coordinates of the point $(u_n,v_n)$
in the Euclidean plane $\R^2$. 
By construction,
for any $\varphi\in U_{\rm iso}\cap i L^2_r$ and $p=\Phi(\varphi)$ we have that $iu_n=x_n(\varphi)$ and $i v_n=y_n(\varphi)$ 
and hence 
\begin{equation}\label{eq:correspondence_coordinates}
\frac{1}{2}\,|p_n|^2=-I_n(\varphi)=-\frac{1}{4}(\xi_n\gamma_n)^2\ge 0\quad\text{and}\quad \theta_n=\theta_n(\varphi).
\end{equation}
Since $\psi$ has simple periodic spectrum we have that $|p_n|>0$ for any $n\in\Z$.
We will assume that $0<\delta_0'<|p_n^0|$ for $|n|\le R'$ and $0<\delta_0''<\pi$. 

\begin{Rem}\label{rem:tail_neighborhood}
Here we used that for any $p^0\in i\ell^2_r$ and for any open neighborhood $W_{p^0}$ of $p^0$ in $i\ell^2_r$ there exists
a open neighborhood $V_{p^0}$ of $p^0$, $V_{p^0}\subseteq W_{p^0}$ of the form \eqref{eq:tailneighborhood}.
An important property of $V_{p^0}$ is that its ``tail'' component $B^{\epsilon_0}_{|n|>R'}(0)$ is a ball in $i \ell^2_r$.
centered at zero. We will call such a neighborhood a {\em tail neighborhood} of $p^0$ in $i\ell^2_r$.
\end{Rem}

\noindent Note that the action variables $I_n$, $n\in\Z$, in $U_{\rm iso}$ 
constructed in Section \ref{sec:actions_and_angles_in_U_iso} are constant on isospectral potentials. 
This follows from the fact that the contours $\Gamma_n$, $n\in\Z$, used  in the definition of the actions on $U_{\rm iso}$ are fixed.
Hence
\[
\Phi\big(\iso(\psi)\cap U_\psi\big)\subseteq
\Tor(p^0)\cap V_{p^0}
\]
where for $q^0\in i\ell^2_r$ we set
\begin{equation}\label{eq:Tor}
\Tor(q^0):=\Big\{\big(i (u_n,v_n)\big)_{n\in\Z}\in i\ell^2_r\,\Big|\,u_n^2+v_n^2=|q_n^0|^2, n\in\Z\Big\}.
\end{equation}
By a slight abuse of terminology we refer to $\Tor(q^0)$ as a torus of dimension 
$\#\big\{n\in\Z\,\big|\,|q_n^0|>0\big\}$.
Note that $\Tor(p^0)$ is a compact set in $i\ell^2_r$ that is an infinite product of circles 
whereas the set $\Tor(p^0)\cap V_{p^0}$ is the product of $2R'+1$ arcs
\begin{equation}\label{eq:finitely_many_circles}
\bigtimes_{|n|\le R'}\Big\{i (u_n,v_n)\in i\,\R^2\,\Big|\,|p_n|=|p^0_n|, |\theta_n-\theta^0_n|<\delta_0''\Big\}
\end{equation}
times the infinite product of circles
\begin{equation}\label{eq:infinitely_many_circles}
\Big\{i (u_n,v_n)_{|n|>R'}\in i\ell^2_r\,\Big|\,|p_n|=|p^0_n|,\,|n|>R'\Big\}.
\end{equation}
It follows from Lemma 8.3 in \cite{GK1} and the fact that the actions are defined only in terms of the discriminant
that $\{I_n,\Delta(\lambda)\}=0$ on $U_{\rm iso}\cap iL^2_r$ for any $n\in\Z$ and $\lambda\in\C$. 
This implies that for any $n\in\Z$ the Hamiltonian vector field $X_{I_n}$ corresponding to the action variables $I_n$ in 
$U_{\rm iso}\cap iL^2_r$ is {\em isospectral}, i.e. for any $\varphi\in U_{\rm iso}\cap iL^2_r$ the integral trajectory of 
$X_{I_n}$ in $U_{\rm iso}\cap iL^2_r$ with initial data at $\varphi$ lies in $\iso(\varphi)$.
In view of the commutation relations
\begin{equation}\label{eq:commutation_relations}
\{\theta_n,I_m\}=\delta_{nm},\quad\{I_n,I_m\}=0,\quad n,m\in\Z, 
\end{equation}
the fact that $\Phi : U_\psi\to V_{p^0}$ is a canonical diffeomorphism and \eqref{eq:correspondence_coordinates}, 
one easily sees from the closedness of $\iso(\psi)$ that the set $\iso(\psi)\cap U_\psi$ is diffeomorphic (via $\Phi$) to 
the product of the sets \eqref{eq:finitely_many_circles} and \eqref{eq:infinitely_many_circles}. 
Since this set is locally path connected so is $\iso(\psi)$. 
\end{proof}

Given a non-empty subset $A$ of $i L^2_r$ and $\delta>0$ denote by $B_\delta(A)$ the open 
$\delta$-{\em tubular} neighborhood of $A$,
\[
B_\delta(A):=\bigcup_{\varphi\in A} B_\delta(\varphi),
\]
where $B_\delta(\varphi)$ denotes the open ball of radius $\delta$ in $i L^2_r$ centered at $\varphi$.
In view of Lemma \ref{lem:iso_components}, the isospectral set $\iso(\psi)$ of any potential $\psi$ with
simple periodic spectrum consists of finitely many connected components. This implies that there exists
$\delta_0\equiv\delta_0(\psi)>0$ such that for any $0<\delta<\delta_0$ the $\delta$-tubular neighborhoods of 
the different connected components of $\iso(\psi)$ do not intersect. 
We have the following {\em Lyapunov type} stability property of $\iso(\psi)$.

\begin{Lem}\label{lem:lyapunov_stability}
Assume that $\psi\in i L^2_r$ has simple periodic spectrum.
Then, for any $0<\delta<\delta_0$ there exists $0<\delta_1\le\delta$ such that for any 
$\varphi\in B_{\delta_1}\big(\iso_o(\psi)\big)$, one has $\iso_o(\varphi)\subseteq B_\delta\big(\iso(\psi)\big)$.
Since $\iso_o(\varphi)$ is connected and $0<\delta\le\delta_0$ we conclude from the choice of $\delta_0>0$ that 
$\iso_o(\varphi)\subseteq B_\delta\big(\iso_o(\psi)\big)$.
\end{Lem}

\begin{proof}[Proof of Lemma \ref{lem:lyapunov_stability}]
Take $\psi\in i L^2_r$ with simple periodic spectrum and choose $0<\delta<\delta_0$. 
We will prove the the statement by contradiction. 
Assume that the statement of Lemma \ref{lem:lyapunov_stability} does not hold. 
Then, there exist two sequence $(\psi_k)_{k\ge 1}$ and $({\widetilde\psi}_k)_{k\ge 1}$ in $i L^2_r$ such that
\begin{equation}\label{eq:negation1}
{\widetilde\psi}_k\in\iso_o(\psi_k)\quad\text{\rm and}\quad{\widetilde\psi}_k\notin B_\delta\big(\iso(\psi)\big),
\end{equation}
and a sequence $(\psi_k^*)_{k\ge 1}$ in $\iso_o(\psi)$ such that
$\|\psi_k^*-\psi_k\|\to 0$ as $k\to\infty$.
By using the compactness of $\iso_o(\psi)$ and then passing to subsequences if necessary we obtain that there exists
$\psi^*\in\iso_o(\psi)$ such that
\begin{equation}\label{eq:negation2}
\psi_k\to\psi^*\quad\text{\rm as}\quad k\to\infty.
\end{equation}
Since by Proposition 2.3 in \cite{KTPreviato} the $L^2$-norm is a spectral invariant of the ZS operator for potentials in $i L^2_r$, 
we conclude from \eqref{eq:negation1} that for any $k\ge 1$, $\|{\widetilde\psi}_k\|=\|\psi_k\|$.
This together with \eqref{eq:negation2} then implies that 
\begin{equation}\label{eq:norms_converge}
\|{\widetilde\psi}_k\|\to\|\psi^*\|\quad\text{\rm as}\quad k\to\infty.
\end{equation}
Hence, the sequence $({\widetilde\psi}_k)_{k\ge 1}$ is bounded in $i L^2_r$.
This implies that there exists ${\widetilde\psi}\in i L^2_r$ such that $({\widetilde\psi}_k)_{k\ge 1}$
converges weakly to ${\widetilde\psi}$ in $i L^2_r$.
Proposition 1.2 in \cite{GK1} then implies that for any $\lambda\in\C$,
\[
\Delta(\lambda,{\widetilde\psi}_k)\to\Delta(\lambda,{\widetilde\psi})\quad\text{\rm as}\quad k\to\infty.
\]
By \eqref{eq:negation2}, for any $\lambda\in\C$,
\[
\Delta(\lambda,\psi_k)\to\Delta(\lambda,\psi^*)\quad\text{\rm as}\quad k\to\infty.
\]
On the other side, by \eqref{eq:negation1}, we conclude that for any $k\ge 1$ and for any $\lambda\in\C$,
\[
\Delta(\lambda,{\widetilde\psi}_k)=\Delta(\lambda,\psi_k).
\]
The three displayed formulas above then imply that $\Delta(\lambda,{\widetilde\psi})=\Delta(\lambda,\psi^*)$
for any $\lambda\in\C$. Hence,
\begin{equation}\label{eq:isospectral_potentials}
{\widetilde\psi}\in\iso(\psi^*)=\iso(\psi).
\end{equation}
By Proposition 2.3 in \cite{KTPreviato} we obtain $\|{\widetilde\psi}\|=\|\psi^*\|$
which, in view of \eqref{eq:norms_converge}, implies that 
\[
\|{\widetilde\psi}_k\|\to\|\widetilde\psi\|\quad\text{\rm as}\quad k\to\infty.
\]
Since $({\widetilde\psi}_k)_{k\ge 1}$ converges weakly to ${\widetilde\psi}$ in $i L^2_r$, we
conclude that
\[
{\widetilde\psi}_k\to{\widetilde\psi}\quad\text{\rm as}\quad k\to\infty.
\]
The second relation in \eqref{eq:negation1} then implies that ${\widetilde\psi}\notin B_\delta\big(\iso(\psi)\big)$
which contradicts \eqref{eq:isospectral_potentials}. 
\end{proof}

\begin{Rem}\label{rem:local_theta_connectedness}
Assume that $\psi\in i L^2_r$ has simple periodic spectrum.
It follows from the proof of Lemma \ref{lem:iso_components} that there exist an open neighborhood $U_\psi$
of $\psi$ in $U_{\rm iso}\cap i L^2_r$ and a tail neighborhood $V_{p^0}$ of $p^0:=\Phi(\psi)$ in $i\ell^2_r$ 
with parameters $R'>0$, $0<\delta_0'<|p^0_n|$ for $|n|\le R'$, $0<\delta_0''<\pi$, and $\epsilon_0>0$
(see \eqref{eq:tailneighborhood}) such that $\Phi : U_\psi\to V_{p^0}$ is a diffeomorphism and for any $\varphi\in U_\psi$ we have that 
$\Phi : U_\psi\to V_{p^0}$ maps the set $\iso_o(\varphi)\cap U_\psi$ bijectively onto the set $\Tor\big(\Phi(\varphi)\big)\cap V_{p^0}$ 
where $\Tor\big(\Phi(\varphi)\big)$ is given by \eqref{eq:Tor} with $q^0=\Phi(\varphi)\in i\ell^2_r$. 
\end{Rem}

\noindent Now, let $\psi:=\psi^{(1)}$ and $p^0:=\Phi(\psi)$. 
In view of the compactness of $\iso_o(\psi)$ and Remark \ref{rem:local_theta_connectedness} we can construct a {\em finite} set of 
open neighborhoods ${\widetilde U}_{{\widetilde\psi}_j}$ of ${\widetilde\psi}_j\in\iso_o(\psi)$ in $U_{\rm iso}\cap i L^2_r$, 
$1\le j\le{\widetilde N}$, such that for any $1\le j\le{\widetilde N}$, 
$\Phi : {\widetilde U}_{{\widetilde\psi}_j}\to{\widetilde V}_{{\tilde p}^j}$ is a diffeomorphism 
onto a tail neighborhood ${\widetilde V}_{{\tilde p}^j}$ of ${\tilde p}^j=\Phi({\widetilde\psi}_j)\in\Tor(p^0)$ in $i\ell^2_r$. 
In what follows we set
\begin{equation}\label{eq:U_iso_shrinked}
U_{\rm iso}\cap iL^2_r=\bigcup_{j=1}^{\widetilde N}{\widetilde U}_{\widetilde\psi_j}.
\end{equation} 
Then, by Lemma \ref{lem:lyapunov_stability}, we can choose $0<\delta<\delta_0$ such that 
$B_{\delta}\big(\iso_o(\psi)\big)\subseteq U_{\rm iso}\cap i L^2_r$ and $\iso_o(\varphi)\subseteq U_{\rm iso}\cap i L^2_r$
for any $\varphi\in B_{\delta}\big(\iso_o(\psi)\big)$.
By Remark \ref{rem:local_theta_connectedness} there exist a neighborhood $U_\psi$ of $\psi$ in $U_{\rm iso}\cap i L^2_r$
and a neighborhood $V_{p^0}$ of $p^0:=\Phi(\psi)$ in $i\ell^2_r$ such that $\Phi : U_\psi\to V_{p^0}$ is 
a diffeomorphism, $V_{p^0}$ is a tail neighborhood
\begin{equation}\label{eq:tailneighborhood1}
V_{p^0}=B^{\delta_0',\delta_0''}_{|n|\le R'}(p^0)\times B^{\epsilon_0}_{|n|>R'}(0)
\end{equation}
with parameters $R'>0$, $0<\delta_0'<|p^0_n|$ for any $|n|\le R'$, $0<\delta_0''<\pi$, and $\epsilon_0>0$,
and $B^{\delta_0',\delta_0''}_{|n|\le R'}(p^0)$ and $B^{\epsilon_0}_{|n|>R'}(0)$
are given by \eqref{eq:finite_component} and \eqref{eq:tail_component}.
By taking the parameters $\delta_0'$, $\delta_0''$, $\epsilon_0>0$ smaller and $R'>0$ larger if necessary
we can ensure that $U_\psi\subseteq B_{\delta}\big(\iso_o(\psi)\big)$ and hence for any 
$\varphi\in U_\psi$ we have that
\begin{equation}\label{eq:flows_well_defined}
\iso_o(\varphi)\subseteq U_{\rm iso}\cap i L^2_r.
\end{equation}
In view of the compactness of the isospectral component $\iso_o(\varphi)$ for any $\varphi\in i L^2_r$
and the fact that the flow of the Hamiltonian vector field $X_{I_n}$ corresponding to the action variable $I_n$, $n\in\Z$,
is {\em isospectral} (see the proof of Lemma \ref{lem:iso_components}), we conclude from \eqref{eq:flows_well_defined} that 
the integral trajectory of $X_{I_n}$ with initial data $\varphi\in U_\psi$ is defined for all $t\in\R$. Moreover,
it follows from the definition \eqref{eq:poisson_bracket} of the Poisson bracket on $i L^2_r$ that 
$X_{I_n} : U_{\rm iso}\cap i L^2_r\to i L^2_r$, $n\in\Z$, is an analytic vector field on $U_{\rm iso}\cap i L^2_r$. 
By \eqref{eq:commutation_relations} the vector fields $X_{I_n}$ and $X_{I_m}$ commute for any $m,n\in\Z$.
Further, consider the direct product of $2R'+1$ open intervals
\begin{equation}\label{eq:tail_section_finite_part}
J_{|n|\le R'}:=\bigtimes_{|n|\le R'}\Big\{p_n=i(u_n,v_n)\in i\R^2\,\Big|\,\big||p_n^0|-|p_n|\big|<\delta_0',\,
\theta_n=\theta_n^0\Big\}
\end{equation}
and denote by $\TT_{\psi}\subseteq U_{\psi}$ the preimage of 
\begin{equation}\label{eq:tail_section}
T_{p^0}:=J_{|n|\le R'}\times B^{\epsilon_0}_{|n|>R'}(0)\subseteq V_{p^0}
\end{equation}
under the diffeomorphism $\Phi : U_\psi\to V_{p^0}$. Since $\Phi : U_{\rm iso}\to i\ell^2_r$ is canonical
we have that (see \eqref{eq:commutation_relations})
\begin{equation}\label{eq:X_I_n_in_coordinates}
\Phi_*(X_{I_n})=\partial_{\theta_n},\quad n\in\Z.
\end{equation}
In particular, the vector fields $X_{I_n}$, $|n|\le R'$, are {\em transversal} to the submanifold $\TT_{\psi}$ in $U_\psi$.
Now, take an arbitrary $\varphi\in\TT_{\psi}$ and consider the orbit 
$G(\varphi):=\big\{G^\tau(\varphi)\,\big|\,\tau\in\R^{2R'+1}\big\}$ 
where
\begin{equation}\label{eq:action1}
G^\tau(\varphi):=G_{X_{I_{(-R')}}}^{\tau_{(-R')}}\circ\cdots\circ G_{X_{I_{R'}}}^{\tau_{R'}}(\varphi),
\quad\tau=(\tau_{(-R')},...,\tau_{R'})\in\R^{2R'+1},
\end{equation}
and $G_{X_{I_n}}^{\tau_n}$ with $|n|\le R'$ is the isospectral flow corresponding to 
the Hamiltonian vector field $X_{I_n}$. Since the flows of $X_{I_n}$, $n\in\Z$, are isospectral and commute we 
conclude from \eqref{eq:flows_well_defined} that for any $\varphi\in\TT_\psi\subseteq U_\psi$ and for any
$\tau\in\R^{2R'+1}$, $G^\tau(\varphi)$ in \eqref{eq:action1} is well defined and $G(\varphi)\subseteq\iso_o(\varphi)$.

\begin{Lem}\label{lem:compact_orbit}
For any $\varphi\in\TT_{\psi}$ the orbit $G(\varphi)$ is a compact smooth submanifold of $U_{\rm iso}\cap i L^2_r$
of dimension $2R'+1$. Moreover, for $\varphi_1,\varphi_2\in\TT_{\psi}$ we have that 
$G(\varphi_1)\cap G(\varphi_2)=\emptyset$ if $\varphi_1\neq\varphi_2$.
\end{Lem}

\begin{proof}[Proof of Lemma \ref{lem:compact_orbit}]
To see that $G(\varphi)$, $\varphi\in\TT_{\psi}$, is compact let $(\varphi_k)_{k\ge 1}$ be a sequence in $G(\varphi)$. 
Since the flows $X_{I_n}$, $n\in\Z$, are isospectral, we have that $G(\varphi)\subseteq\iso_o(\varphi)$.
We then conclude from the compactness of $\iso_o(\varphi)$ that there exist $\varphi^*\in\iso_o(\varphi)$ and 
a subsequence of $(\varphi_k)_{k\ge 1}$, denoted again by $(\varphi_k)_{k\ge 1}$, such that 
\begin{equation}\label{eq:convergence1}
\varphi_k\to\varphi^*,\quad\text{as}\quad k\to\infty.
\end{equation}
In view of \eqref{eq:U_iso_shrinked} and \eqref{eq:flows_well_defined},
$\varphi^*\in{\widetilde U}_{{\widetilde\psi}_j}\cap\iso_o(\varphi)$ where ${\widetilde U}_{{\widetilde\psi}_j}$ is one of 
the neighborhoods appearing in \eqref{eq:U_iso_shrinked}. 
Since $\Phi : {\widetilde U}_{{\widetilde\psi}_j}\to{\widetilde V}_{{\tilde p}^j}$ is a diffeomorphism
and since ${\widetilde V}_{{\tilde p}^j}$ is a tail neighborhood, we conclude
that $\Phi(\varphi^*)\in{\widetilde V}_{{\tilde p}^j}\cap\Tor\big(\Phi(\varphi)\big)$ 
(see Remark \ref{rem:local_theta_connectedness}). It follows from \eqref{eq:convergence1}
that there exists $k_0\ge 1$ so that $\Phi(\varphi_k)\in{\widetilde V}_{{\tilde p}^j}\cap\Tor\big(\Phi(\varphi)\big)$
for any $k\ge k_0$. This implies that $\varphi^*=G^{\tau^*}(\varphi_{k_0})$ for some $\tau^*\in\R^{2R'+1}$
where $\varphi_{k_0}\in G(\varphi)$. In particular, we see that $\varphi^*\in G(\varphi)$ and hence $G(\varphi)$ is compact. 
The coordinates $(\theta_{-R'},...,\theta_{R'})$ on ${\widetilde V}_{{\tilde p}^j}\cap\Tor\big(\Phi(\varphi)\big)$, $1\le j\le{\widetilde N}$, 
define a smooth (in fact, real analytic) submanifold structure on $G(\varphi)$ in $U_{\rm iso}\cap i L^2_r$.
It remains to prove the last statement of the Lemma.
Take $\varphi_1,\varphi_2\in\TT_{\psi}$ so that $\varphi_1\neq\varphi_2$. Then, in view of the definition of $\TT_{\psi}$, 
either there exists $|n_0|>R'$ such that 
\begin{equation}\label{eq:separation1}
(x_{n_0},y_{n_0})|_{\varphi_1}\neq(x_{n_0},y_{n_0})|_{\varphi_2}
\end{equation}
or there exists $|n_0|\le R'$ such that 
\begin{equation}\label{eq:separation2}
I_{n_0}(\varphi_1)\neq I_{n_0}(\varphi_2).
\end{equation}
Since the vector fields $X_{I_n}$ with $|n|\le R'$ preserve the functions $x_n,y_n$ with $|n|>R'$ and the functions $I_n$ 
with $|n|\le R'$ we conclude that for any $\tau,\mu\in\R^{2R'+1}$ the relations \eqref{eq:separation1} (or \eqref{eq:separation2}) holds
with $\varphi_1$ and $\varphi_2$ replaced respectively by $G^\tau(\varphi_1)$ and $G^\mu(\varphi_2)$.
This implies that $G(\varphi_1)\cap G(\varphi_2)=\emptyset$.
\end{proof}

Let us now consider the following set in $i L^2_r$,
\begin{equation}\label{eq:W}
\W:=\bigsqcup_{\varphi\in\TT_\psi} G(\varphi).
\end{equation}

\begin{Lem}\label{lem:W}
$\W$ is an open neighborhood of $\psi$ in $i L^2_r$ that is invariant with respect to the flows of the
Hamiltonian vector fields $X_{I_n}$, $n\in\Z$.
\end{Lem}

\begin{proof}[Proof of Lemma \ref{lem:W}]
It follows from the definition \eqref{eq:tailneighborhood1} of the tail neighborhood $V_{p^0}$ and the set
$T_{p^0}\subseteq V_{p^0}$ defined in \eqref{eq:tail_section} that 
\[
U_\psi=\big\{G^{\tau}(\TT_\psi)\,\big|\,|\tau_n|<\delta_0'', |n|\le R'\big\}.
\]
This implies that 
\begin{equation}\label{eq:W'}
\W=\bigsqcup_{\varphi\in\TT_\psi} G(\varphi)=\bigcup_{\tau\in\R^{2R'+1}}G^\tau(\TT_\psi)=
\bigcup_{\tau\in\R^{2R'+1}}G^\tau(U_\psi).
\end{equation}
Since the sets $G^\tau(U_\psi)$, $\tau\in\R^{2R'+1}$, are open we conclude from \eqref{eq:W'} that
$\W$ is an open set in $i L^2_r$. The invariance of $\W$ with respect to the flow of $X_{I_n}$, $n\in\Z$,
follows from the fact that the section $\TT_\psi$ in \eqref{eq:W} is invariant with respect to the flows of $X_{I_n}$, $|n|>R'$,
and since $X_{I_n}$ and $X_{I_k}$ commute for any $n,k\in\Z$.
\end{proof}

By restricting \eqref{eq:action1} to $\R^{2R'+1}\times\W$ we obtain a smooth action
\begin{equation}\label{eq:action2}
G : \R^{2R'+1}\times\W\to\W.
\end{equation}
of $\R^{2R'+1}$ on $\W$. In view of \eqref{eq:X_I_n_in_coordinates} we have the following commutative diagram
\begin{equation}\label{eq:diagram_action}
\begin{tikzcd}
\W\arrow{r}{G^\tau}\arrow{d}{\Phi}&\W\arrow{d}{\Phi}\\
W\arrow{r}{\rho^\tau}&W
\end{tikzcd}
\end{equation}
where $W$ is the tail neighborhood (cf. \eqref{eq:tailneighborhood1})
\begin{equation}\label{eq:Wdown}
W:=B^{\delta_0',2\pi}_{|n|\le R'}(p^0)\times B^{\epsilon_0}_{|n|>R'}(0)\subseteq i\ell^2_r
\end{equation}
where the parameter $\delta_0''>0$ is replaced by $2\pi$ and for any $\tau=(\tau_{(-R')},...,\tau_R)\in\R^{2R'+1}$
and any $|n|\le R'$ the map $\rho^\tau : W\to W$ rotates the component $p_n=i(u_n,v_n)\in i\R^2$ of $p\in W$ by 
the angle $\theta_n=\tau_n$ for $|n|\le R'$ while keeping  the components of $p$ for $|n|>R'$ unchanged.
The commutative diagram \eqref{eq:diagram_action} and the invariance of $\W$ (cf. Lemma \ref{lem:W}) then 
easily imply that $\Phi :\W\to W$ is {\em onto}.
Note also that by \eqref{eq:Wdown}, the open set $W$ in $i\ell^2_r$ is a direct product of $2R'+1$ two 
dimensional annuli and an open ball in $i\ell^2_r$.

\begin{Coro}\label{coro:stabilizer}
For any $\varphi\in\TT_{\psi}$ the orbit $G(\varphi)$ is diffeomorphic to the $2R'+1$ dimensional torus
\begin{equation}\label{eq:torus}
\mathcal{G}:=\R^{2R'+1}/\Span\langle e_{(-R')},...,e_{R'}\rangle_\Z
\end{equation}
where $e_{(-R')},...,e_{R'}\in(2\pi\Z)^{2R'+1}$ are linearly independent over $\R$. The vectors
$e_{(-R')},...,e_{R'}$ are independent of the choice of $\varphi\in\TT_{\psi}$.\footnote{Note that we do {\em not}
claim that $e_{(-R')},...,e_{R'}$ is a basis of $(2\pi\Z)^{2R'+1}$ over $\Z$.}
\end{Coro}

\begin{proof}[Proof of Corollary \ref{coro:stabilizer}]
By restricting the action \eqref{eq:action2} to the orbit $G(\varphi)$ with 
$\varphi\in\TT_\psi$ we obtain a smooth {\em transitive} action 
\begin{equation}\label{eq:action3}
G : \R^{2R'+1}\times G(\varphi)\to G(\varphi)
\end{equation}
of $\R^{2R'+1}$ on the submanifold $G(\varphi)$. Since $G(\varphi)$ is compact (cf. Lemma \ref{lem:compact_orbit}) 
the stabilizer $\text{\rm St}(\varphi)$ of \eqref{eq:action3} is of the form
\[
\text{\rm St}(\varphi)=\Span\langle e_{(-R')},...,e_{R'}\rangle_\Z
\] 
where the vectors $e_{(-R')},...,e_{R'}\in\R^{2R'+1}$ form a basis in $\R^{2R'+1}$ and
$G(\varphi)$ is diffeomorphic to the factor group $\R^{2R'+1}/\text{\rm St}(\varphi)$,
which is a torus of dimension $2R'+1$ (cf. \cite[\S\,49]{Arn}). It then follows from the commutative diagram 
\eqref{eq:diagram_action} that
\begin{equation}\label{eq:integer_condition}
e_{(-R')},...,e_{R'}\in(2\pi\Z)^{2R'+1}.
\end{equation}
Indeed, assume that $\tau\in\text{\rm St}(\varphi)$. Then $G^\tau(\varphi)=\varphi$ and, by 
the commutative diagram \eqref{eq:diagram_action}, we obtain that $\rho^\tau\big(\Phi(\varphi)\big)=\Phi(\varphi)$. 
This implies that $\tau\in(2\pi\Z)^{2R'+1}$, and hence, completes the proof of the first statement of the Corollary.
Let us now prove the second statement. Take $\varphi^*\in\TT_\psi$ and let 
$\tau^*\in\text{\rm St}(\varphi^*)$. In view of the continuity of the action \eqref{eq:action2},
there exists an open neighborhood $U_1$ of zero in $\R^{2R'+1}$ and an open neighborhood $U_2$ of $\varphi^*$
in $U_\psi$ so that for any $(\tau,\varphi)\in U_1\times U_2$ we have that $G^{\tau^*+\tau}(\varphi)\in U_\psi$. It then follows
from the commutative diagram \eqref{eq:diagram_action} and the fact that $\tau^*\in(2\pi\Z)^{2R'+1}$ that
\begin{equation}\label{eq:time_shift_in_coordinates}
\Phi\big(G^{\tau^*+\tau}(\varphi)\big)=\left\{
\begin{array}{l}
i\,|p_n|\,\big(\cos(\theta_n(\varphi)+\tau_n),\sin(\theta_n(\varphi)+\tau_n)\big),\quad |n|\le R',\\
p_n,\quad|n|>R',
\end{array}
\right.
\end{equation}
where $(p_n)_{n\in\Z}=\Phi(\varphi)$.
Since $\Phi : U_\psi\to V_{p^0}$ is a diffeomorphism, we conclude from \eqref{eq:time_shift_in_coordinates} that
$G^{\tau^*}(\varphi)=\varphi$ for any $\varphi\in U_2$. In view of the connectedness of $U_\psi$ 
we then obtain that $\tau^*\in\text{\rm St}(\varphi)$ for any $\varphi\in U_\psi$. Since $\TT_\psi$ is connected
this shows that the $2R'+1$ linearly independent vectors in \eqref{eq:integer_condition} can be chosen independently 
of $\varphi\in\TT_\psi$.
\end{proof}

With these preparations done we are now ready to introduce the slight adjustment of $\Phi$, mentioned at the beginning 
of the Section. By \eqref{eq:W} and Lemma \ref{lem:compact_orbit} the open neighborhood  $\W$ of $\psi$ in $i L^2_r$ is 
foliated by orbits of the action \eqref{eq:action2} with $\TT_\psi$ as a global section. By Corollary \ref{coro:stabilizer}
any orbit $G(\varphi)$, $\varphi\in\TT_\psi$, is diffeomorphic to the $2R'+1$ dimensional
torus \eqref{eq:torus}. Denote by $(t_{(-R')},...,t_{R'})$ the coordinates corresponding to the frame $(e_{(-R')},...,e_{R'})$
in $\R^{2R'+1}$. For any given $\varphi\in\TT_\psi$ denote by $(\theta^*_{(-R')},...,\theta^*_{R'})$
the {\em coordinates} on $G(\varphi)$ that are obtained by taking the pull-back of $(2\pi t_{(-R')},...,2\pi t_{R'})$ 
via the diffeomorphism $G(\varphi)\to\mathcal{G}$ given by the action \eqref{eq:action3}. 
For any $|n|\le R'$ denote the coordinates of $e_n$ by $\tau_{nk}$, $|k|\le R'$,
\begin{equation}\label{eq:e_n}
e_n=\big(\tau_{n(-R')},...,\tau_{nR'}\big)\in(2\pi\Z)^{2R'+1},\quad |n|\le R'.
\end{equation}
For $|n|\le R'$ consider the (analytic) functions $I^*_n : \W\to\R$,
\begin{equation}\label{eq:I*}
I^*_n:=\frac{1}{2\pi}\sum_{|k|\le R'}\tau_{nk} I_k,
\end{equation}
where $I_k$, $|k|\le R'$, is the $k$-th action on $U_{\rm iso}$.
Then, it follows from \eqref{eq:I*} and the second statement of Corollary \ref{coro:stabilizer} that
$X_{I^*_n}=\sum_{|k|\le R'}\tau_{nk} X_{I_k}$, and by the definition of the action \eqref{eq:action2} 
(see also \eqref{eq:action1}) and the coordinates $(\theta^*_{(-R')},...,\theta^*_{R'})$, we conclude from \eqref{eq:e_n} 
that on any orbit $G(\varphi)$, $\varphi\in\TT_\psi$, we have that
\begin{equation}\label{eq:*-canonical}
\big\{\theta^*_n,I^*_k\big\}\equiv(d\theta^*_n)\big(X_{I^*_k}\big)=d\big(2\pi t_n\big)(e_k/2\pi)=\delta_{nk}
\end{equation}
for any $|n|\le R'$ and $|k|\le R'$. This implies that on any orbit $G(\varphi)$, $\varphi\in\TT_\psi$,
\begin{equation}\label{eq:theta*}
d\theta^*_n=\sum_{|k|\le R'}\tau^*_{nk}d\theta_k
\end{equation}
where $\theta_k : \W\to\R/2\pi$, $|k|\le R'$, is the $k$-th angle on $U_{\rm iso}\cap i L^2_r$ and
$(\tau^*_{nk})_{|n|,|k|\le R'}$ are the elements of the {\em non-degenerate} $(2R'+1)\times(2R'+1)$ matrix 
$P^*:= \big(P^{-1}\big)^T$ where $P:=(\tau_{nk})_{|n|,|k|\le R'}$ and $\big(P^{-1}\big)^T$ denotes
the transpose of $P^{-1}$. Formula \eqref{eq:theta*} shows that the coordinates $\theta^*_{(-R')},...,\theta^*_{R'}$ 
taken modulo $2\pi$ are real analytic functions on $\W$. In this way, we have

\begin{Lem}\label{lem:separating_points}
The angles $(\theta^*_n)_{|n|\le R'}$ taken modulo $2\pi$, the actions $(I^*_n)_{|n|\le R'}$, and
$(x_n)_{|n|>R'}$, $(y_n)_{|n|>R'}$, are real analytic on $\W$ and separate the points on $\W$. Moreover, 
$\{x_n,y_n\}=1$, $|n|>R'$, $\{\theta^*_n,I^*_n\}=1$, $|n|\le R'$, whereas all other Poisson brackets vanish.
\end{Lem}

\begin{proof}[Proof of Lemma \ref{lem:separating_points}]
We already proved that the functions in Lemma \ref{lem:separating_points} are real analytic and that their Poisson brackets satisfy 
the relations stated in the Lemma. Moreover, by the commutation relations \eqref{eq:commutation_relations},
the functions $(I^*_n)_{|n|\le R'}$, $(x_n)_{|n|>R'}$, and $(y_n)_{|n|>R'}$, are constant on any of the orbits 
$G(\varphi)$, $\varphi\in\TT_\psi$, and they form a coordinate system on the section $\TT_\psi$ in \eqref{eq:W}. Hence, 
these functions separate the orbits $G(\varphi)$, $\varphi\in\TT_\psi$. Since the functions $(\theta^*_n)_{|n|\le R'}$ taken 
modulo $2\pi$ separate the points on any of these orbits, we conclude the proof of the Lemma.
\end{proof}

Lemma \ref{lem:separating_points} allows us to define the modification $\Psi : \W\to i\ell^2_r$ of the map
$\Phi : \W\to i\ell^2_r$ by setting
\begin{equation}\label{eq:modified_coordinates}
x_n:=\sqrt{2I^*_n}\cos\theta^*_n,\quad y_n:=\sqrt{2I^*_n}\sin\theta^*_n,\quad|n|\le R',
\end{equation}
while keeping the other components of $\Phi$ unchanged.
By shrinking the open neighborhood $\W$ if necessary we can ensure that $\Psi(\W)$ is a tail neighborhood of 
the form \eqref{eq:Wdown} with $p^0:=\Psi(\psi)$ and the same value of the parameter $\epsilon_0>0$. 
For simplicity of notation, we will denote $\Psi(\W)$ again by $W$ and for any $|n|\le R'$,write $\theta_n$ and $I_n$ instead of 
$\theta^*_n$ and, respectively, $I^*_n$. We have

\begin{Prop}\label{prop:almost_done}
The map
\begin{equation}\label{eq:Phi*}
\Psi : \W\to W
\end{equation} 
is a canonical real analytic diffeomorphism. Moreover, the open neighborhood $\W$ is a saturated neighborhood of $\iso_o(\psi)$
and $W$ is a tail neighborhood of $\Psi(\psi)$ in $i\ell^2_r$.
\end{Prop}

\begin{proof}[Proof of Proposition \ref{prop:almost_done}]
The Proposition follows from Lemma \ref{lem:separating_points}. Indeed,
it follows from Lemma \ref{lem:separating_points} and \eqref{eq:modified_coordinates} that
the map \eqref{eq:Phi*} is analytic. Furthermore, it follows from \eqref{eq:I*}, \eqref{eq:theta*}, 
\eqref{eq:modified_coordinates}, and the fact that $\Phi : \W\to i\ell^2_r$ is a local diffeomorphism (cf. Theorem \ref{th:local_diffeo})
onto its image, that \eqref{eq:Phi*} is a local diffeomorphism onto $W$. Since, by Lemma \ref{lem:separating_points},
the components of $\Psi$ separate the points on $\W$, we then conclude that 
$\Psi$ is a diffeomorphism. This map is canonical by the last statement in Lemma \ref{lem:separating_points}.
The statement that $\W$ is a saturated neighborhood of $\iso_0(\psi)$ follows from 
the invariance of the neighborhood $\W$ with respect to the complete isospectral flows $X_{I_n}$, $n\in\Z$
(see Lemma \ref{lem:W}), and the arguments in the proof of Lemma \ref{lem:iso_components} that allow
us, for any $\varphi\in\W$, to identify $\iso_o(\varphi)$ with $\Tor\big(\Psi(\varphi)\big)$ (see \eqref{eq:Tor}) in 
the tail neighborhood $W$.
\end{proof}

Finally, we prove

\begin{Prop}\label{prop:N>=1}
For any integer $N\ge 1$ the restriction of the map \eqref{eq:Phi*} to $\W\cap i H^N_r$ takes values in $i\h^N_r$ and 
$\Psi\big|_{\W\cap i H^N_r} : \W\cap i H^N_r\to W\cap i\h^N_r$ is a real analytic diffeomorphism.
\end{Prop}

\begin{proof}[Proof of Proposition \ref{prop:N>=1}]
Assume that $N\ge 1$. Since the coordinate functions $x_n$ and $y_n$ are real analytic on $\W$ so are their restrictions to 
$\W\cap i H^N_r$. Recall that by Lemma \ref{Lemma 5.4} that $z_n^\pm=O\big(|\gamma_n|+|\mu_n-\tau_n|\big)$ for $|n|>R$, 
locally uniformly on $U_{\rm iso}$. Taking into account the estimates for $\gamma_n$ on $H^N_c$ in \cite[Corollary 1.1]{KSchT2} and 
the ones for $\mu_n-\tau_n$ on $H^N_c$ in \cite[Theorem 1.1]{KSchT2} and \cite[Theorem 1.3]{KSchT2} it follows that 
\begin{equation}\label{eq:z_n-estimate}
\sum_{|n|>R}\langle n\rangle^{2N}|z_n^\pm|^2<\infty
\end{equation}
locally uniformly on $U_{\rm iso}\cap H^N_c$. Furthermore, recall from Remark \ref{rem:U_iso} that ${\tilde\beta}^n$, defined modulo $2\pi$,
is of order $O(1/n)$ as $|n|\to\infty$ locally uniformly on $U_{\rm iso}\setminus\mathcal{Z}_n$ and that 
$\sum_{|k|>R,k\ne n}\beta^n_k=O(1/n)$ as $|n|\to\infty$ locally uniformly on $U_{\rm iso}$. By combining this with \eqref{eq:z_n-estimate} 
we conclude from \eqref{eq:x,y'} that the real analytic map $\Psi\big|_{\W\cap i H^N_r} : \W\cap i H^N_r\to i\h^N_r$ is locally bounded in 
a complex neighborhood of $\W\cap i H^N_r$. By \cite[Theorem A.5]{GK1} it then follows that 
$\Psi\big|_{\W\cap i H^N_r} : \W\cap i H^N_r\to i\h^N_r$ is real analytic.
Furthermore, by the characterization of potentials $\varphi\in i L^2_r$ to be in $i H^N_r$, provided in \cite[Theorem 1.2]{KST},
it follows that $\Psi(\W\cap i H^N_r)=W\cap i\h^N_r$, implying that $\Psi\big|_{\W\cap i H^N_r} : \W\cap i H^N_r\to W\cap i\h^N_r$
is bijective. To see that the latter map is a real analytic diffeomorphism it remains to show that for any $\varphi\in\W\cap i H^N_r$,
$\big(d_\varphi\Psi\big)\big|_{i H^N_r} : i H^N_r\to i\h^N_r$ is a linear isomorphism. Since $d_\varphi\Psi : L^2_r\to i\ell^2_r$
is an isomorphism by Proposition \ref{prop:almost_done}, we conclude that $\mathop{\rm Ker}\big(d_\varphi\Psi\big)\big|_{i H^N_r}=\{0\}$. 
Hence, we will complete the proof if we show that $\big(d_\varphi\Psi\big)\big|_{i H^N_r} : i H^N_r\to i\h^N_r$ is a Fredholm operator
of index zero. This will follow once we prove that $\big(d_\varphi\Phi\big)\big|_{i H^N_r} : i H^N_r\to i\h^N_r$ is a Fredholm operator
of index zero, where $\Phi$ is the map \eqref{eq:birkhoff_map}.
In order to prove this, we show by analytic extension that the formulas for $z_n^\pm$ in \cite[Theorem 2.2]{KSchT1}, 
valid for potentials near zero, continue to hold on $U_{\rm iso}$ for any $|n|>R$. These formulas involve the quantities 
$\tau_n-\mu_n$, $\delta(\mu_n)$, $\delta_n(\mu_n)$, and $\eta_n^\pm$, which can be estimated
using \cite[Theorem 1.1, Theorem 1.3, and Theorem 1.4]{KSchT2} 
and \cite[Lemma 12.7]{GK1}. In this way we show that for any $N\ge 1$,
\begin{equation}\label{eq:one-smoothing}
\Phi-\mathcal{F} : \W\cap i H^N_r\to i\h^{N+1}_r,
\end{equation}
is a real analytic map. Here $\mathcal{F} : H^N_c\to\h^N_c$ is the Fourier transform 
$(\varphi_1,\varphi_2)\mapsto\big(-{\widehat\varphi}_1(n),-{\widehat\varphi}_2(n)\big)_{n\in\Z}$ 
(cf. the identification introduced in \eqref{eq:identification}).
Hence, for any $\varphi\in\W\cap i H^N_r$ we obtain that $d_\varphi\Phi-\mathcal{F} : i H^N_r\to i\h^{N+1}_r$ is 
a bounded linear map. Since the Sobolev embedding $i\h^{N+1}_r\to i\h^N_r$ is compact, we then conclude that 
$d_\varphi\Phi-\mathcal{F} : i H^N_r\to i\h^N_r$ is a compact operator. This completes the proof of the Proposition.
\end{proof}

Now, we are ready to proof Theorem \ref{th:main}.

\begin{proof}[Proof of Theorem \ref{th:main}]
By setting $z_n:=(x_n-i y_n)/\sqrt{2}$ and $w_n:=-\overline{z_{(-n)}}$ for any $n\in\Z$ 
Proposition \ref{prop:almost_done} implies that the statements {\em(NF1)} and {\em(NF2)}
(for $N=0$) of Theorem \ref{th:main} hold (cf. \eqref{eq:identification}). The statement {\em(NF2)} for $N\ge 1$
follows from Proposition \ref{prop:N>=1}. For proving {\em(NF3)} we follow the arguments of the proof of
Theorem 20.3 in \cite{GK1}.
\end{proof}

\section{Appendix: Auxiliary results}
First we provide details of the proof that the sum of the integrals \eqref{eq:argument_principle}, introduced in 
the proof of Proposition \ref{prop:beta^n_tn}, is analytic.
We use the notation introduced there.
Assume that for a given $\nu\in\Lambda^-_R(\psi)$, $\mu_\psi(\nu)$ is a Dirichlet eigenvalue of $L(\psi)$ of 
multiplicity $m_D\ge 2$. 
Denote by $\nu_j$, $1\le j\le m_D$, the periodic eigenvalues of $L(\psi)$ such that $\mu_\psi(\nu_j)=\mu_\psi(\nu)$ and
for any $\varphi\in U_\psi$ denote by $z_j$, $1\le j\le m_D$, the periodic eigenvalues of $L(\varphi)$ with 
$z_j\in D^\varepsilon(\nu_j)$.

\begin{Lem}\label{lem:argument_principle} 
The quantity
\begin{equation}\label{eq:argument_principle'}
\sum_{1\le j\le m_D}
\int_{\PP^*_{\varphi,\psi}[z_j,\mu_\varphi(z_j)]}
\frac{\zeta_n(\lambda,\varphi)}{\sqrt{\Delta^2(\lambda,\varphi)-4}}\,d\lambda
\end{equation}
is analytic in $U_\psi$.
\end{Lem}

\begin{proof}[Proof of Lemma \ref{lem:argument_principle}]
In view of the representation \eqref{eq:concatenation*}, it is enough to prove that
\begin{equation}\label{eq:sum1}
\sum_{1\le j\le m_D}
\int_{[Q_{\mu_\psi(\nu_j)},\mu_\varphi(z_j)]^*}
\frac{\zeta_n(\lambda,\varphi)}{\sqrt{\Delta^2(\lambda,\varphi)-4}}\,d\lambda
\end{equation}
is analytic in $U_\psi$.
In the case where $\mu_\psi(\nu)\notin\Lambda_R(\psi)$ it follows that for any $\varphi\in U_\psi$ the initial point of
$\big[Q_{\mu_\psi(\nu_j)},\mu_\varphi(z_j)\big]^*$ does not depend on $1\le j\le m_D$.
One then concludes by the argument principle that
\[
\sum_{1\le j\le m_D}
\int\limits_{[Q_{\mu_\psi(\nu_j)},\mu_\varphi(z_j)]^*}
\frac{\zeta_n(\lambda,\varphi)}{\sqrt{\Delta^2(\lambda,\varphi)-4}}\,d\lambda
=\frac{1}{2\pi i}\int_{\partial\overline{D}^\varepsilon(\mu_\psi(\nu))}\!\!\!\!\!\!\!\!
F(\mu,\varphi)\,\frac{\dot\chi_D(\mu,\varphi)}{\chi_D(\mu,\varphi)}\,d\mu
\]
where $\chi_D$ is given by \eqref{eq:chi_D} and
\[
F(\mu,\varphi):=\int_{Q_{\mu_\psi(\nu)}}^\mu
\frac{\zeta_n(\lambda,\varphi)}{\sqrt[*]{\Delta^2(\lambda,\varphi)-4}}\,d\lambda
\]
is analytic on the disk $D^\varepsilon\big(\mu_\psi(\nu)\big)$ and continuous up to its boundary.
The case when $\mu_\psi(\nu)\in\Lambda_R(\psi)$ can be treated in a similar way.
We only remark that in this case, the initial points $Q^*_{\mu_\psi(\nu_j),\varphi}$, $1\le j\le m_D$,
of the paths $\big[Q_{\mu_\psi(\nu_j)},\mu_\varphi(z_j)\big]^*$, $1\le j\le m_D$, are {\em not} necessarily 
the same. If $Q^*_{\mu_\psi(\nu_j),\varphi}=Q^*_{\mu_\psi(\nu),\varphi}$ let $\PP^*_{j,\varphi}$ be the constant
path $Q^*_{\mu_\psi(\nu),\varphi}$ and if $Q^*_{\mu_\psi(\nu_j),\varphi}\ne Q^*_{\mu_\psi(\nu),\varphi}$, let
$\PP^*_{j,\varphi}$ be the counterclockwise oriented lift of the circle $\partial\overline{D}^\varepsilon\big(\mu_\psi(\nu)\big)$
by $\pi_1|_{\mathcal{C}_{\varphi,R}} : \mathcal{C}_{\varphi,R}\to\C$, which connects
$Q^*_{\mu_\psi(\nu),\varphi}$ with $Q^*_{\mu_\psi(\nu_j),\varphi}$. We then write the path integral
$\int_{[Q_{\mu_\psi(\nu_j)},\mu_\varphi(z_j)]^*}
\frac{\zeta_n(\lambda,\varphi)}{\sqrt{\Delta^2(\lambda,\varphi)-4}}\,d\lambda$
as a sum of two path integrals with paths 
$\PP^*_{j,\varphi}\cup\big[Q_{\mu_\psi(\nu_j)},\mu_\varphi(z_j)\big]^*$ and $\big(\PP^*_{j,\varphi}\big)^{-1}$.
By the argument principle,
\[
\sum_{1\le j\le m_D}
\int\limits_{\PP^*_{j,\varphi}\cup[Q_{\mu_\psi(\nu_j)},\mu_\varphi(z_j)]^*}
\frac{\zeta_n(\lambda,\varphi)}{\sqrt{\Delta^2(\lambda,\varphi)-4}}\,d\lambda
\]
is analytic on $U_\psi$. Since
\[
\sum_{1\le j\le m_D}
\int\limits_{\big(\PP^*_{j,\varphi}\big)^{-1}}
\frac{\zeta_n(\lambda,\varphi)}{\sqrt{\Delta^2(\lambda,\varphi)-4}}\,d\lambda
\]
is also analytic on $U_\psi$, it then follows that \eqref{eq:argument_principle'} analytic on 
$U_\psi$ in this case.
\end{proof}

\medskip

The second result characterizes the spectral bands of $\spec_{\mathbb R} L(\varphi )$.  
It is used in the proof of Lemma \ref{lem:negative_actions} to show that the actions $I_n$, $|n|>R$, are non positive
on $U_{\rm iso}\cap i L^2_r$. For any given $\varepsilon>0$ and $\delta > 0$ sufficiently small, 
and $z\in\C$, denote by $B^{\varepsilon,\delta}_z$ the following {\em box} in $\C$,
\[ 
B^{\varepsilon,\delta }_z := \big\{\lambda\in\C\,\big|\,
|\re(\lambda - z)|<\varepsilon,\,\,
|\im(\lambda - z)|<\delta \big\}.
\]
In the sequel we use results on the discriminant $\Delta(\lambda,\varphi)$ of $L(\varphi)$ reviewed in Section \ref{sec:setup}.

\begin{Lem}\label{lem:spectral_bands} 
For any $\psi\in i L^2_r$, choose $R_p \in\Z_{\geq 0}$ as in 
Lemma \ref{lem:counting_lemma}. 
Then there exist an integer $\tilde R\ge R_p$, as well as $\varepsilon > 0$, $\delta > 0$, 
and a neighborhood $W_\psi $ of $\psi $ in $i L^2_r$ so that for any 
$|n| > \tilde R$ and $\varphi\in W_\psi $, 
$B^{\varepsilon , \delta }_{\dot \lambda _n(\varphi )}\subseteq D_n$ and 
$B^{\varepsilon , \delta }_{\dot \lambda _n(\varphi )}\cap \spec_{\mathbb R}L(\varphi )$ consists 
of the interval $(\dot\lambda _n(\varphi ) - \varepsilon,\dot\lambda _n(\varphi ) + \varepsilon )\subseteq\R$
and a smooth arc $g_n$ connecting $\lambda ^-_n(\varphi )$ with 
$\lambda ^+_n(\varphi )$ within $B^{\varepsilon,\delta }_{\dot \lambda_n(\varphi )}$
so that $\Delta(\lambda)$ is real valued on $g_n$ and satisfies $-2 < \Delta (\lambda ) < 2$ for any 
$\lambda \in g_n\setminus\{ \lambda ^\pm _n(\varphi ) \} $. In fact for any $\varphi \in W_\psi $,
the arc $g_n$, also referred to as {\em spectral band}, is the graph 
$\{ a_n(v) +i v : |v| \leq \im(\lambda^+ _n)\}$ of a smooth real valued function 
$a_n : [- \im(\lambda ^+_n), \im (\lambda^+_n)]\to\R$ with the property 
that $a_n(0) = \dot \lambda _n(\varphi )$, $a_n(-t) = a_n (t)$ for any $0 \le t \le\im(\lambda ^+_n)$, and
$a_n(\pm\im(\lambda ^+_n)) = \re(\lambda ^+_n)$.
\end{Lem}

For the convenience of the reader we include the proof of Lemma \ref{lem:spectral_bands} given in \cite{KLT1}.

\begin{proof}[Proof of Lemma \ref{lem:spectral_bands}]
 First let us introduce some more notation. For any $\lambda \in
\C$ write $\lambda = u +  i v$ with $u, v \in {\mathbb R}$ and
$\Delta = \Delta _1 +  i \Delta _2$ where for $\varphi \in L^2_c$ arbitrary, 
$\Delta_j(u,v) \equiv \Delta_j(u,v,\varphi )$, $j=1,2$, are given by
   \[ \Delta _1(u,v):= \re\big(\Delta (u +  iv, \varphi )\big), \qquad
      \Delta _2(u,v):= \im\big(\Delta (u +  iv, \varphi )\big)  .
   \]
In a first step we want to study $\Lambda \cap D_n$ for $\varphi \in  iL
^2_r$ where
   \[ \Lambda = \{ u +  iv\,|\, u,v \in {\mathbb R},\ \Delta _2(u,v)
      = 0 \} .
   \]
By Lemma \ref{lem:product1} $(iii)$ for any $\varphi \in  i L^2_r$, $\Delta _2(u,0,\varphi ) = 0$.
Hence 
\[
F : {\mathbb R} \times {\mathbb R} \times  i L^2_r \to
{\mathbb R}, \ (u,v, \varphi ) \mapsto \Delta _2(u,v,\varphi ) / v
\] 
is well defined. As $F$ and $\Delta _2$ have the same zero set on 
${\mathbb R} \times({\mathbb R} \backslash \{ 0 \})$ we investigate $\{ (u,v) \in {\mathbb R}^2 :
F(u,v,\varphi ) = 0 \} $ further. We note that
   \begin{equation}\label{2.20} 
       F(u,v) \equiv F(u,v,\varphi ) = \int ^1_0\big(\partial _v \Delta_2\big)(u,tv)\,dt
   \end{equation}
and hence $F(u,0) = \partial _v \Delta _2(u,0)$. Since $\Delta _2(u,0)$ vanishes for any
$u \in {\mathbb R}$, so does $\partial _u \Delta _2(u,0)$. By the
Cauchy-Riemann equation it then follows that for $u \in {\mathbb R}$, 
$\dot\Delta (u,0) = 0$ if and only if $\partial _v \Delta _2(u,0) = 0$. The real roots of $F(\cdot ,
0) = \partial _v \Delta _2(\cdot , 0)$ are thus given by $\dot \lambda _n$, 
$n \in \Z$, with $\dot \lambda _n \in {\mathbb R}$. Since $\Delta $ is an
analytic function on $\C \times L^2_c$ one concludes that $\Delta_2(\cdot,\cdot ,\cdot )$ 
is a real analytic function on $\R\times\R\times i L^2_r$, and by formula \eqref{2.20},
so is $F$.

For $\psi \in  i L^2_r$ let $R_p \in \Z_{\ge 0}$ and the neighborhood $V_\psi$
of $\psi $ be as in Lemma \ref{lem:counting_lemma}. 
For any $|k| > R_p$, let $S_k$ be the map
\[
S_k : B_{\ell^\infty} \times (-1,1) \times\big(V_\psi\cap i L^2_r\big)\to
      {\mathbb R}, \, \, \big( (\zeta _n)_{|n| > R_p}, v, \varphi \big) \mapsto
      F(\dot \lambda _k + \zeta_k, v, \varphi )
\]
where $\dot \lambda _k = \dot \lambda _k(\varphi )$ is in $D_k$ and $B_{\ell^\infty} $
denotes the open unit ball in the space of real valued sequences
   \[ \ell ^\infty = \big\{ \zeta = (\zeta _n)_{|n| > R_p}\,\big|\,
      \| \zeta \| _\infty < \infty\big\},\quad\| \zeta \| _\infty := \sup _{|n| > R_p} |\zeta_n|\,.
   \]
Note that $S_k$ is the composition of the map
$B_{\ell^\infty} \times (-1,1) \times\big(V_\psi\cap i L^2_r\big)\to\R\times (-1,1)
      \times\big(V_\psi\cap i L^2_r\big)$,
   \[ (\zeta , v, \varphi ) \mapsto\big(\dot \lambda _k
      (\varphi ) + \zeta _k, v, \varphi\big),
   \]
with $F$ and hence real analytic. It follows from the asymptotics of
$\Delta $ of \cite[Lemma 4.3] {GK1} and the asymptotics of $\dot\lambda_k$
stated in Lemma \ref{lem:counting_lemma} that for any sequence 
$\zeta = (\zeta _n)_{|n| > R_p} \in B_{\ell^\infty}$, $(S_k)_{|k| > R_p}$ is also in $\ell ^\infty$. 
We claim that
   \[ S: B_{\ell^\infty} \times (-1,1) \times\big(V_\psi\cap i L^2_r\big)\to \ell ^\infty , \
      (\zeta , v, \varphi ) \mapsto\big(S_k(\zeta , v, \varphi )\big)_{|k| > R_p}
   \]
is smooth. To see it note that by the Cauchy-Riemann equation,
$\Delta _2(u,v,\varphi )$ and its derivatives, when restricted to
${\mathbb R} \times {\mathbb R} \times  i L^2_r$, can be estimated
in terms of $\Delta (\lambda , \varphi )$ and its derivatives. Hence the
asymptotic estimates of $\Delta (\lambda , \varphi )$ (see \cite[Theorem 2.2, Lemma 4.3] {GK1}) 
together with Cauchy's estimate can be used to estimate
the derivatives of $S_k$ to conclude that the map $S = (S_k)_{|k| > R_p}$ is
smooth.

We now would like to apply the implicit function theorem to $S$. First note
that $S\big\arrowvert _{\zeta = 0, v = 0, \psi } = 0$ as $\dot \lambda_k(\psi )$, $k\in\Z$, 
are the roots of $\dot \Delta (\cdot,\psi )$. In addition, since $\dot\lambda_k(\psi)$, $|k|>R_p$, are
real and simple roots of ${\dot\Delta}(\cdot,\psi)$, one has
   \[ \partial _{\zeta _k} S_k \big\arrowvert _{\zeta = 0, v = 0,
      \psi } = \partial _u \partial _v \Delta _2\big\arrowvert
      _{u = \dot \lambda _k(\psi ),v = 0,\psi}
   \]
whereas by the definition of $S_k$, $\partial_{\zeta _n} S_k$ vanishes identically
for any $n \not= k$. By the asymptotics
$\dot \lambda _k(\psi ) = k\pi + \ell ^2_k$ (cf. Lemma \ref{lem:counting_lemma}) and the one of $\Delta $
(cf \cite[Lemma 4.3] {GK1}) it follows from Cauchy's estimate that
   \[ \partial _u \partial _v \Delta _2\big\arrowvert _{u=\dot\lambda_k(\psi),v = 0,\psi} = 
    \partial_u \partial _v \Delta(u +  iv, 0)
      \big\arrowvert _{u=\dot \lambda_k(\psi),v=0} + \ell ^2_k .
   \]
On the other hand $\Delta (\lambda , 0) = 2 \cos \lambda $ and thus 
$\im\big(\Delta (u +  iv, 0)\big) = - 2 \sin u \sinh v$ implying that
   \[ \partial _u \partial _v \im\big(\Delta (u +  iv,0)\big)\big\arrowvert
      _{u=\dot\lambda_k(\psi),v = 0} = 2(-1)^{k+1}\big(1 + \ell^2_k\big) .
   \]
Altogether we then conclude that
   \[ \partial _\zeta S \big\arrowvert _{\zeta = 0, v = 0,\psi} = 2 \mathop{\rm diag}
      \left( \big((-1)^{k+1}\big)_{|k|> R_p} + \ell ^2_k \right) .
   \]
Thus there exists an integer $\tilde R\ge R_p$ so that for the restriction of $S$ to
$B_{\tilde\ell^\infty}\times(-1,1)\times\big(V_\psi\cap i L^2_r\big)$,
   \[ \tilde S : \tilde B_{\tilde\ell^\infty} \times (-1, 1) \times\big(V_\psi\cap i L^2_r\big)\to
      \tilde\ell^\infty,
   \]
the differential $\partial _{\tilde\zeta}{\tilde S}\big\arrowvert _{{\tilde\zeta}=0, v = 0,\psi}$ 
is invertible. Here $B_{\tilde\ell^\infty} $ denotes the unit ball in
   \[{\tilde\ell}^\infty = \big\{{\tilde\zeta} = (\zeta _n)_{|n|> \tilde R}
      \subseteq {\mathbb R}\,\big|\, \|{\tilde\zeta}\| _\infty < \infty\big\} .
   \]
By the implicit function theorem there exist $\delta > 0$, a neighborhood $W_\psi \subseteq V_\psi$ of 
$\psi$ in $L^2_c$, $0<\varepsilon<1$, and a smooth map
   \[ h : (-\delta , \delta ) \times W_\psi \to B_{\tilde\ell^\infty}
      (\varepsilon) , \ (v,\varphi ) \mapsto h(v,\varphi ) = \big(h_n(v,\varphi )\big)_{|n| > \tilde R}
   \]
so that $h(0,\psi ) = 0$ and $\tilde S(h(v,\varphi ), v, \varphi ) = 0$ for
any $(v,\varphi ) \in (-\delta , \delta ) \times W_\psi $. Here
   \[ B_{\tilde\ell^\infty}(\varepsilon) = \big\{{\tilde\zeta}\in 
       \tilde\ell^\infty\,\big|\,\|{\tilde\zeta}\| _\infty < \varepsilon\big\}
   \]
and $(v,\varphi ) \mapsto (h(v,\varphi ),v,\varphi )$ parametrizes the
zero level set of $\tilde S$ in $B_{\tilde\ell^\infty}(\varepsilon) \times (-
\delta , \delta ) \times W_\psi $. In particular, for any $|n| > \tilde
R$ and $\varphi $ in $W_\psi $, the intersection of $\{ F(u,v,\varphi ) =
0\}$ with the box $B^{\delta , \varepsilon }_{\dot \lambda _n(\varphi )}$
is smoothly parametrized by
   \[ z_n(\cdot , \varphi ) := (-\delta , \delta ) \to B^{\delta ,
      \varepsilon }_{\dot \lambda _n(\varphi )} , \ v \mapsto \dot \lambda
      _n(\varphi ) + h_n(v,\varphi ) +  iv .
   \]
By choosing $\tilde R$ larger and by shrinking $W_\psi $, if necessary,
one can assure that $\lambda^\pm _n\equiv\lambda ^\pm _n(\varphi )$ is in
$B^{\delta , \varepsilon }_{\dot \lambda _n(\varphi )}$ for any $\varphi
\in W_\psi $ and $|n| > \tilde R$. Since $\Delta (\lambda ^\pm _n) = (-1)^n 2$
and $\Delta (\dot \lambda _n) \in {\mathbb R}$, the pair 
$\lambda ^+_n$, $\lambda ^-_n$ as well as $\dot \lambda _n$ are in the range of $z_n$.
In fact, one has
   \[ z_n(0,\varphi ) ={\dot\lambda}_n \mbox{ and } z_n\big(\pm\im(\lambda ^+_n),\varphi\big) = 
   \lambda ^\pm _n,\quad\forall\varphi\in W_\psi.
   \]
Furthermore recall that $\chi _p(\dot \lambda _n, \varphi ) \leq 0$ since $\dot\lambda_n\in\R$. 
As $\lambda ^\pm _n$ are the only roots of $\chi _p(\cdot,\varphi )$ in $D_n$, it
follows that if $\dot \lambda _n \not= \lambda ^\pm _n$, then $\chi_p(\dot \lambda _n,\varphi ) < 0$ and, 
hence
   \[ \chi _p(z_n(v), \varphi ) < 0 \quad \forall - \im(\lambda^+_n) < v < \im(\lambda ^+_n) .
   \]
Altogether we thus have proved that for any $|n|>\tilde R$
   \[ g_n = \big\{ z_n(v)\,\big|\, - \im(\lambda ^+_n) \le v \le \im(\lambda ^+_n)\big\}
   \]
is a smooth arc on which $\Delta $ takes values in $[-2,2]$. We also have all the
properties listed in the statement of the Lemma.
\end{proof}

\end{document}